%% file: main.tex
\documentclass[onefignum,onetabnum]{siamart190516}  

\include{shared}

\newif\ifshowdetails
\newif\ifshowextensivedetails
\newif\ifshowresults
\newif\ifshowalgorithms
\newif\ifshowappendix
\newif\ifshowall

\showalltrue

\ifshowall
  \showdetailstrue
  \showextensivedetailstrue
  \showresultstrue
  \showalgorithmstrue
  \showappendixtrue
\else
  \showdetailsfalse
  \showextensivedetailsfalse
  \showresultstrue
  \showalgorithmsfalse
  \showappendixtrue
\fi

\ifpdf
\hypersetup{
  pdftitle={Stochastic Learning for Binary Bayesian OED},
  pdfauthor={A. Attia, S. Leyffer, and T. Munson}
}
\fi

\begin{document}


\maketitle

\begin{abstract}
  We present a novel stochastic approach to binary optimization for optimal
  experimental design (OED) for Bayesian inverse problems governed by
  mathematical models such as partial differential equations.
  The OED utility function, namely, the regularized optimality criterion, 
  is cast into a stochastic objective function in the form of an 
  expectation over a multivariate Bernoulli distribution.
  The probabilistic objective is then solved by using a stochastic optimization
  routine to find an optimal observational policy.
  The proposed approach is analyzed from an optimization perspective 
  and also from a machine learning perspective with correspondence 
  to  policy gradient reinforcement learning.
  The approach is demonstrated numerically by using an idealized two-dimensional
  Bayesian linear inverse problem, and validated by extensive numerical experiments carried out for sensor placement 
  in a parameter identification setup.
  %
\end{abstract}





\section{Introduction} 
  \label{sec:introduction}
  Optimal experimental design (OED) is the general mathematical formalism 
  for configuring an observational setup.
  This can refer to the frequency of data collection or the spatiotemporal
  distribution of observational gridpoints for physical experiments.
  OED has seen a recent surge of interest in the field sensor placement for 
  applications of model-constrained Bayesian inverse problems; see, for 
  example,~\cite{Pukelsheim93,FedorovLee00,Pazman86,Ucinski00,HaberHoreshTenorio08,HaberMagnantLuceroEtAl12,HaberHoreshTenorio10,HuanMarzouk13,bui2013computational,alexanderian2016bayesian,AlexanderianSaibaba17,AlexanderianPetraStadlerEtAl16,AlexanderianPetraStadlerEtAl14,alexanderian2020optimal,attia2018goal,attia2020optimal}.

  Many physical phenomena can be simulated by using mathematical models. The
  accuracy of the mathematical models is limited, however, by the level at which the physics of the true
  system is captured by the model and by the numerical errors produced, for example, by computer simulations.
  Model-based simulations are often corrected based on snapshots of reality,
  such as sensor-based measurements. 
  Such corrections involve first inferring the model parameter from the noisy
  measurements, a problem referred to as  the ``inverse problem,'' which can be solved
  by different inversion or data assimilation methods; see e.g., ~\cite{bannister2017review,daley1993atmospheric,navon2009data,attia2016reducedhmcsmoother,attia2015hmcfilter}.
  These observations themselves are noisy
  but in general follow a known probability distribution.
  The OED problem is concerned with finding the most informative, and thus
  optimal, observational grid configuration out of a set of candidates, 
  which, when deployed, will result in a reliable solution of the inverse problem.
  
  OED is most beneficial when there is a limit on the number of sensors that can
  be used, for example, when sensors are expensive to deploy and/or operate.
  In order to solve an OED problem for sensor placement, a binary optimization problem is
  formulated by assigning a binary design variable $\design_i$ for each
  candidate sensor location.
  The objective is often set to a utility function that summarizes the 
  uncertainty in the inversion parameter or the amount of information gained 
  from the observations.
  The optimal design is then defined as the optimizer of this objective function.
  The objective is constrained by the model evolution equations, such as
  partial differential equations, and also by any regularity or sparsity
  constraints on the design.
  Here we focus on sparse and binary designs.
  A popular choice of a penalty function to enforce a sparse and binary design 
  is the $\ell_0$ penalty~\cite{AlexanderianPetraStadlerEtAl14}.

  Solving such binary optimization problems is computationally prohibitive, and
  often the binary optimization problem is replaced with a relaxation.
  In the relaxation of an OED problem, the design variable is allowed to 
  take any value in the interval $[0,1]$ and is often interpreted as 
  an importance weight of the candidate location.
  Once a solution of the relaxed problem is obtained numerically, for example, by a
  gradient-based optimization procedure, a binarization (rounding) procedure is 
  required to transform the solution of the relaxed problem into an equivalent 
  solution of the original problem.
  In general, however,  the numerical solutions of 
  the two problems are not guaranteed to be related or even similar.

  Using a gradient-based approach to solve the relaxed OED problem requires many 
  evaluations of the simulation model in order to 
  evaluate the objective function and its gradient.
  Moreover, this approach requires the penalty function to be differentiable 
  with respect to the design variables. 
  To this end, since $\ell_0$ is discontinuous, numerous efforts have been 
  dedicated to approximating the effect of $\ell_0$ sparsification 
  with other approaches.
  For example, in~\cite{HaberHoreshTenorio08}, an $\ell_1$ penalty followed by a
  thresholding method is used; however, $\ell_1$ is nonsmooth and hence is 
  nondifferentiable.
  Continuation procedures are used in~\cite{AlexanderianPetraStadlerEtAl14,koval2020optimal},
  where a series of OED problems are solved with a sequence of penalty functions
  that successively approximate $\ell_0$ sparsification.
  Another approach that can potentially induce a binary design is the 
  sum-up-rounding  algorithm~\cite{yu2018scalable}, which also provides
  optimality guarantees.

  Here we propose a new efficient approach for directly solving the original 
  binary OED problem.
  In this framework, the objective function, namely, the regularized optimality
  criterion, is cast into an objective 
  function in the form of an expectation over a multivariate 
  Bernoulli distribution. 
  A stochastic optimization approach is then applied to solve the reformulated
  optimization problem in order to find an optimal observational policy.
  
  Related work on stochastic optimization for OED has been  explored 
  in~\cite{huan2014gradient}, where the utility function, in other words, the optimality
  criterion, is approximated by using Monte Carlo estimators.
  A stochastic optimization procedure is then applied with the gradient of the
  utility function with respect to the design being approximated by using
  Robbins--Monro stochastic approximation~\cite{robbins1951stochastic}, 
  or sample average approximation~\cite{spall2005introduction,nemirovski2009robust}.

  The main differences from the standard OED approach and previous stochastic
  OED approaches, which highlight our main contributions in this work, 
  are as follows. 
  First, the proposed methodology searches for an optimal probability
  distribution representing the optimal binary design; hence it does not 
  require relaxation or binarization to follow the optimization step as in the
  traditional OED formulation. 
  Second, in our approach the penalized OED criterion is not required to be
  differentiable with respect to the design. Thus, one can incorporate the
  $\ell_0$ regularization norm to enforce sparsification without needing to
  approximate it with the $\ell_1$ penalty term or apply a continuation
  procedure to retrieve a binary optimal design. 
  Third, the proposed stochastic OED formulation does not require formulating
  the gradient of the objective function with respect to the design, thus
  implying a massive reduction in computation cost.
  Fourth, the proposed approach and solution algorithms are completely 
  independent from the specific choice of the OED optimality criterion 
  and the penalty function and hence  can be used with both linear and nonlinear
  problems.

  In this work, for simplicity we provide numerical experiments for 
  Bayesian inverse problems constrained by linear forward operators, 
  and we consider an A-optimal design, that is, a design that minimizes the trace of the
  posterior covariances of the inversion parameter.
  The approach proposed, however, extends easily to other OED optimality criteria, 
  such as D-optimality.
  We provide an analysis of the method from a mathematical point of view and
   an interpretation from a machine learning (ML) perspective. 
  Specifically, the proposed algorithm has
  a direct link to policy gradient algorithms widely used in neuro-dynamic
  programming and reinforcement
  learning~\cite{BertsekasTsitsiklis96,williams1992simple,sutton2000policy}.
  Numerical experiments using a toy example are provided to help with 
  understanding the problem; the proposed algorithm; and the relation between
  the binary OED problem, the relaxed OED problem, and the proposed formulation.
  Moreover, extensive numerical experiments are performed for optimal sensor
  placement for an advection-diffusion problem that simulates the 
  spatiotemporal evolution of the concentration of a contaminant in 
  a bounded domain. 

  The  paper is organized as follows. 
  \Cref{sec:background} gives a brief description of the Bayesian inverse
  problem and the standard formulation of an optimal experimental 
  design in the context of Bayesian inversion.
  In~\Cref{sec:problem_formulation}, we describe our proposed approach for
  solving the binary OED problem and present a detailed analysis of the
  proposed algorithms.
  An explanatory example and extensive numerical experiments are given
  in~\Cref{sec:numerical_experiments}.
  Discussion and concluding remarks are presented in~\Cref{sec:conclusions}.

  Throughout the paper, we use boldface symbols for vectors and matrices.
  $\diag{\vec{x}}$ is a diagonal matrix with diagonal entries set to the
  entries of the vector $\vec{x}$. 
  The $i$th cardinality vector in the Euclidean space $\Rnum^{n}$ is 
  denoted by $\vec{e}_i\in\Rnum^{n}$.
  Subscripts refer to entries of vectors and matrices, and square brackets are
  used to symbolize instances of a vector, for example, randomly drawn from a
  particular distribution.
  Superscripts with round brackets are reserved for iterations in 
  optimization routines.


\section{OED for Bayesian Inversion} 
  \label{sec:background}
  Consider the forward model described by
  \begin{equation}\label{eqn:forward_problem}
    \obs  = \F(\iparam) + \vec{\delta} \,,
  \end{equation}
  where $\iparam \in \Rnum^{\Nstate}$ is the discretized model parameter and
  $\obs \in \Rnum^{\Nobs}$ is the observation, where
  $\vec{\delta} \in \Rnum^{\Nobs}$ is the observation error.
  Assuming Gaussian observational noise 
  $\vec{\delta} \sim \GM{\vec{0}}{\Cobsnoise}$, then 
  the data likelihood takes the form
  \begin{equation} \label{eqn:Gaussian_likelihood}
    \Like{\obs}{ \iparam } \propto
      \exp{\left( - \frac{1}{2} 
        \sqwnorm{ \F(\iparam) - \obs }{ \Cobsnoise^{-1} } \right) } \,,
  \end{equation}
  where $\Cobsnoise$ is the observation error covariance matrix (positive definite) and 
  the matrix-weighted norm is defined as
  $\sqwnorm{\vec{x}}{\mat{A}} = \vec{x}\tran \mat{A} \vec{x} $.

  An inverse problem refers to the retrieval of the model parameter $\iparam$ from the
  noisy observation $\obs$, conditioned by the model dynamics.
  In Bayesian inversion, the goal is to study the probability distribution
  of $\iparam$ conditioned by the observation $\obs$ that is the posterior
  obtained by applying Bayes' theorem, 
  \begin{equation}\label{eqn:Bayes}
    \CondProb{\iparam}{\obs} \propto \CondProb{\obs}{\iparam} \Prob(\iparam) \,
  \end{equation}
  where $\Prob(\iparam)$ is the prior distribution of the inversion parameter
  $\iparam$, which in many cases is assumed to be Gaussian
  $\iparam \sim \GM{\iparb}{\Cparampriormat}$.
  If the parameter-to-observable map $\F$ is linear, the posterior  
  $\CondProb{\iparam}{\obs} $ is Gaussian $\GM{\ipara}{\Cparampostmat}$ 
  with
  \begin{equation}\label{eqn:Posterior_Params}
    \Cparampostmat = \left(\F \adj \Cobsnoise^{-1} \F 
      + \Cparampriormat^{-1} \right)^{-1} \,, \quad
    \ipara = \Cparampostmat \left( \Cparampriormat^{-1} \iparb 
      + \F\adj \Cobsnoise^{-1}\, \obs \right) \,,
  \end{equation}
  where $\F\adj$ is the forward operator adjoint.
  If the forward operator $\F$ is nonlinear, however, the posterior distribution
  is no longer  Gaussian. A Gaussian approximation of the posterior can be
  obtained in this case by linearization of $\F$ around the maximum a posteriori 
  (MAP) estimate, which is inherently data dependent.
  For simplicity, we will focus on linear models; however, the approach
  proposed in this work is not limited to the linear case. 
  Further analysis for nonlinear settings will follow in separate works.

  In a Bayesian OED context (see, e.g.,
  \cite{alexanderian2020optimal,alexanderian2016bayesian,AlexanderianPetraStadlerEtAl14,attia2018goal,attia2020optimal,HaberHoreshTenorio08,HaberHoreshTenorio10}), 
  the observation covariance $\Cobsnoise$ is replaced with a weighted version
  $\wdesignmat(\design)$, parameterized by the design $\design$, 
  resulting in the following weighted data-likelihood $\Like{\obs}{
  \iparam; \design}$ and posterior covariance $\wCparampostmat$. 
  \begin{equation} \label{eqn:weighted_likelihood_and_post_covariance}
    \Like{\obs}{ \iparam; \design}
      \propto \exp{\left( - \frac{1}{2} 
      \sqwnorm{ \F(\iparam) - \obs}{\wdesignmat(\design)} \right) } \,, 
      \quad
    \wCparampostmat
      = \left( \F\adj \wdesignmat(\design) \F
        + \Cparampriormat^{-1} \right)\inv 
      \,.
  \end{equation}

  The exact form of the posterior covariance $\wCparampostmat$ depends on how
  $\wdesignmat(\design)$ is formulated.
  We are  interested mainly in binary designs $\design\in \{0,1\}^{\Nsens}$,
  required, for example, in sensor placement applications; see, for example,~\cite{alexanderian2020optimal,attia2020optimal}.
  In this case, we assume a set of $\Nsens$ candidate sensor locations, and we
  seek the optimal subset of locations.
  Note that selecting a subset of the observations is equivalent to applying a 
  projection operator onto the subspace spanned by the activated sensors. 
  This is  equivalent to defining the weighed design matrix as 
  $\wdesignmat:=\Cobsnoise^{-1/2} \designmat \Cobsnoise^{-1/2}$,
  where $\designmat:=\diag{\design}$ is a diagonal matrix with binary values on
  its diagonal. The $i$th entry of the design is set to $1$ to activate a sensor
  and is set to $0$ to turn it off.

  The optimal design $\design\opt$ is the one that optimizes a predefined criterion
  $\optcriterion(\cdot)$ of choice. 
  The most popular optimality criteria in the Bayesian inversion context are 
  A- and D-optimality.
  Both criteria define the optimal design as the one that minimizes a scalar
  summary of posterior uncertainty associated with an inversion
  parameter or state of an inverse problem. 
  Specifically, $\optcriterion(\design):=\Trace{\Cparampost(\design)}$ for an A-optimal
  design, and $\optcriterion(\design):=\logdet{\Cparampost(\design)}$ for a D-optimal
  design, i.e., the sum of the eigenvalues of $\Cparampost(\design)$, 
  and the sum of the logarithms of the eigenvalues of $\Cparampost(\design)$, respectively.
  
  To prevent experimental designs that are simply dense, one typically adds
  a regularization term $\Phi(\design)$ that promotes sparsity, for example.
  Thus, to find an optimal design $\design\opt$, one needs to solve the following
  binary optimization problem,
  \begin{equation}\label{eqn:binary_optimization}
    \design\opt 
    = \argmin_{\design \in \{0, 1\}^{\Nsens}} \obj(\design) 
      := \optcriterion(\design)\, + \regpenalty \, \Phi(\design) \,,
  \end{equation}
  where the function $\Phi(\design)$ promotes regularization or sparsity on the
  design and  $\regpenalty$ is a user-defined regularization parameter.
  For example, $\Phi$ could represent a budget constraint on the form
  $\Phi(\design) := \wnorm{\design}{0} \leq k \,; k\in \mathtt{Z}_{+}$ 
  or a sparsifying (possibly discontinuous) function, for example, 
  $\wnorm{\design}{0}$ or $\wnorm{\design}{1}$.

  Traditional binary optimization approaches~\cite{wolsey1999integer} 
  for solving the 
  optimization problem are  expensive, rendering the exact solution
  of~\eqref{eqn:binary_optimization} computationally intractable.
  The optimization problem~\eqref{eqn:relaxed_optimization} is a continuous
  relaxation of~\eqref{eqn:binary_optimization} and is often used in practice
  as a surrogate for solving~\eqref{eqn:binary_optimization} with a suitable
  rounding scheme:
  \begin{equation}\label{eqn:relaxed_optimization}
    \design\opt 
    = \argmin_{\design \in [0, 1]^{\Nsens}} \obj(\design) 
      := \optcriterion(\design)\, + \regpenalty \, \Phi(\design)\,. 
  \end{equation}

  In sensor placement we seek a sparse design in order to minimize the deployment cost, and
  thus we would utilize $\ell_0$ as a penalty  function. Solving the relaxed
  optimization problem~\eqref{eqn:relaxed_optimization} follows a
  gradient-based approach, and thus using $\ell_0$ as a penalty function is
  replaced with the nonsmooth $\ell_1$ norm; 
  see~\cite{AlexanderianPetraStadlerEtAl14} for more details.

  The relaxed problem~\eqref{eqn:relaxed_optimization} provides a lower bound 
  on~\eqref{eqn:binary_optimization}, and any binary solution 
  $\design\opt\in\{0,1\}^\Nsens$ of~\eqref{eqn:relaxed_optimization}  
  is also an optimal solution of~\eqref{eqn:binary_optimization}.
  Given a solution of~\eqref{eqn:relaxed_optimization} that is not binary, 
  we can obtain a binary solution by rounding; however, there is no guarantee
  that the resulting binary solution is optimal, unless we use sum-up-rounding. 
  %
  %
  In~\Cref{sec:problem_formulation} we present a new approach for directly solving
  the original OED binary optimization problem~\eqref{eqn:binary_optimization}
  without the need to relax the design space or to round the relaxations.

\section{Stochastic Learning Approach for Binary OED}
  \label{sec:problem_formulation}
  Our main goal is to find the solution of the original binary regularized OED 
  problem~\eqref{eqn:binary_optimization} without solving a mixed-integer
  programming  problem or resorting to the relaxation 
  approach widely followed in OED for Bayesian inversion (as summarized
  in~\Cref{sec:background}).
  We propose to reformulate and solve the binary optimization 
  problem~\eqref{eqn:binary_optimization} as a stochastic optimization problem 
  defined over the parameters of a probability distribution.
  Specifically, we assume $\design$ is a random variable that follows a
  multivariate Bernoulli distribution.
  %

  \subsection{Stochastic formulation of the OED problem}
  \label{subsec:stochastic_OED}
    We associate with each candidate sensor location $x_i$ a probability of
    activation $\hyperparam_i$. 
    Specifically, we assume that $\design_i,\,i=1,2,\ldots,\Nsens$ are independent 
    Bernoulli random variables associated with the candidate sensor locations
    $x_i,\,i=1,2,\ldots,\Nsens$.  
    The activation (success) probabilities of the respective sensors 
    are $\hyperparam_i,\, i=1,2,\ldots,\Nsens$; that is,
    $\Prob(\design_i=1)=\hyperparam_i,\, \Prob(\design_i=0)=1-\hyperparam_i$.
    The probability associated with any observational configuration
    is then described by the joint probability of all candidate locations. 
    Note that assuming independent activation variables $\design_i$ does not 
    interfere the correlation structure of the observational errors, manifested
    by the observation error covariance matrix $\Cobsnoise$.
    Conversely, setting $\design_i$ to $0$ corresponds to removing the $i$th
    row and column from the precision matrix $\Cobsnoise\inv$, while setting the
    design variable value to $1$ corresponds to keeping the corresponding row
    and column, respectively.
    We let $\CondProb{\design}{\hyperparam}$ denote the joint multivariate Bernoulli
    probability, with the following probability mass function (PMF):
    \begin{equation}\label{eqn:joint_Bernoulli_pmf_prod}
      \CondProb{\design}{\hyperparam}
        := \prod_{i=1}^{\Nsens} {\hyperparam_i ^ {\design_i}\,
          \left(1-\hyperparam_i\right)^{1-\design_i}} \,,\quad 
            \design_i \in  \{0, 1\},\, \quad
            \hyperparam_i \in[0,1]   \,.
    \end{equation}

    We then replace the original problem~\eqref{eqn:binary_optimization}
    with the following stochastic optimization problem: 
    %
    \begin{equation}\label{eqn:stochastic_optimization}
      \hyperparam\opt = \argmin_{\hyperparam \in [0, 1]^{\Nsens} } 
        \stochobj(\hyperparam)
      := \Expect{\design\sim\CondProb{\design}{\hyperparam}}{\obj(\design)} 
      = \Expect{\design\sim \CondProb{\design}{\hyperparam}}
        {\optcriterion(\design)+ \regpenalty\, \Phi(\design)} \,.
    \end{equation}

    Because the support of the probability distribution is discrete, the possible
    values of $\design\in\{0,1\}^{\Nsens}$ are countable and can be assigned
    unique indexes.
    An index $k$ is assigned to each possible realization of
    $\design:=(\design_1,\design_2,\ldots,\design_{\Nsens})$, based on the
    values of its components using the
    relation
    \begin{equation}\label{eqn:binary_index}
      k = 1 + \sum_{i=1}^{\Nsens}{\design_i\, 2 ^{i-1}} \,, \quad \design_i\in\{0,1\} \,.
    \end{equation}
    Thus, all possible designs are labeled as
    $\design[k],k\!=\!1,2,\ldots,2^{\Nsens}$.  
    With the indexing scheme~\eqref{eqn:binary_index}, the optimization
    problem~\eqref{eqn:stochastic_optimization} takes the following equivalent
    form:
    \begin{equation}\label{eqn:stochastic_optimization_expanded}
      \hyperparam\opt = \argmin_{\hyperparam \in [0, 1]^{\Nsens} } 
        \stochobj(\hyperparam)
      := \sum_{k=1}^{2^{\Nsens}} \obj(\design[k]) \CondProb{\design[k]}{\hyperparam} \,.
    \end{equation}
    %
    
    To solve~\eqref{eqn:stochastic_optimization_expanded}, one can follow a
    gradient-based optimization approach to find the optimal parameter
    $\hyperparam\opt$.
    The gradient of the objective in~\eqref{eqn:stochastic_optimization_expanded} with
    respect to the distribution parameters $\hyperparam$ is 
    \begin{equation}\label{eqn:exact_gradient}
      \nabla_{\hyperparam}\,\stochobj(\hyperparam)
        = \nabla_{\hyperparam}\, 
          \Expect{\design\sim \CondProb{\design}{\hyperparam}}{\obj(\design)} 
        = \nabla_{\hyperparam}\, 
          \sum_{\design}\obj(\design) \CondProb{\design}{\hyperparam} 
        = \sum_{k=1}^{2^{\Nsens}} \obj(\design[k]) \nabla_{\hyperparam}\,
          \CondProb{\design[k]}{\hyperparam} 
        \,,
    \end{equation}
    where $\nabla_{\hyperparam}\, \CondProb{\design}{\hyperparam} $ is the
    gradient of the joint Bernoulli PMF~\eqref{eqn:joint_Bernoulli_pmf_prod}
    (see~\Cref{app:Bernoulli} for details):
    \begin{equation}\label{eqn:Bernoulli_PMF_gradient}
        %
      \nabla_{\hyperparam}\, \CondProb{\design[k]}{\hyperparam}
        = \sum_{j=1}^{\Nsens} 
          \left. \del{ \CondProb{\design[k]}{\hyperparam} }{\hyperparam_j} 
          \right|_{\design=\design[k]}
        = \sum_{j=1}^{\Nsens} (-1)^{1-\design_j[k]}
          \prod_{\substack{i=1 \\ i\neq j}}^{\Nsens} {\hyperparam_i ^
          {\design_i[k]}
          \left(1-\hyperparam_i\right)^{1-\design_i[k]}} \vec{e}_j \,.
    \end{equation}

    In~\Cref{app:Bernoulli}, we provide a detailed discussion of the multivariate
    Bernoulli distribution, with identities and lemmas that will be useful for
    the following discussion.
    
    
    Clearly,~\eqref{eqn:stochastic_optimization_expanded} and~\eqref{eqn:exact_gradient} 
    are not practical formula, 
    because their statement alone requires the evaluation of all possible 
    designs $\obj(\design[k])$, which is equivalent to complete 
    enumeration of~\eqref{eqn:binary_optimization}. 
    We show below how we can utilize stochastic gradient approaches 
    to avoid complete enumeration.
    Practical and efficient solution of~\eqref{eqn:stochastic_optimization} is
    discussed in~\Cref{subsec:stochastic_OED_solution}. 
    We first establish in~\Cref{subsec:benefits_of_stochastic_OED} the
    connection between the solution of the two problems
    \eqref{eqn:binary_optimization} and \eqref{eqn:stochastic_optimization}
    and discuss the benefits of the proposed approach.

  \subsection{Benefits of the stochastic formulation}
    \label{subsec:benefits_of_stochastic_OED}
    %
    
    The connections between the two problems 
    (\ref{eqn:binary_optimization} and~\ref{eqn:stochastic_optimization}) and
    their respective solutions are summarized
    by~\Cref{prop:binary_stochastic_connection}
    and~\Cref{lemma:binary_stochastic_solutions}.
    The relation between the objective functions in these two problems, 
    and their domains and codomains, is sketched in~\Cref{fig:two_problems}.

    \begin{proposition}\label{prop:binary_stochastic_connection}
      Consider the two functions 
      $\obj:\Omega_\design:=\{0,1\}^{\Nsens} \to \Rnum$ and 
      $\stochobj:\Omega_\hyperparam:=[0,1]^{\Nsens} \to \Rnum$,
      defined in~\eqref{eqn:binary_optimization} and~\eqref{eqn:stochastic_optimization}, 
      respectively.
      Let $\design \in\Omega_\design$ be a random vector following the multivariate 
      Bernoulli distribution~\eqref{eqn:joint_Bernoulli_pmf_prod}, 
      with parameter $\hyperparam \in \Omega_\hyperparam$.
      Let 
      $ C:=\{c\!\in\!\Rnum\,|\,c\!:=\!\obj(\design)\,,\, \design \in \Omega_\design \}$, be 
      the set of objective values corresponding to all possible designs,
      $\design\in\Omega_\design$. 
      Then:
      \begin{enumerate}[a)]
        %
          %
        \item $\Omega_\hyperparam$ is the convex hull of $\Omega_\design$ in
          $\Rnum^{\Nsens}$. 
        \item The values of $\stochobj$~\eqref{eqn:stochastic_optimization}
          form the convex hull of $C$ in $\Rnum$, denoted by $conv(C)$, with points
          $c\in C$ being the extreme points of $conv(C)$; that is,
          $conv(C)\equiv
          \{\stochobj(\hyperparam)|\hyperparam\in\Omega_\hyperparam \}$.
          %
        \item $\stochobj(\vec{x})=\obj(\vec{x})\, \forall
          \vec{x}\in{\Omega_\design}$.
        \item For any realization of $\hyperparam\in\Omega_\hyperparam$, 
          it holds that $\min\{C\} \leq \stochobj(\hyperparam) \leq
          \max\{C\}$; moreover,
          $\stochobj(\hyperparam\opt)=\min\{C\}=\obj(\design\opt)$.
          %
        %
      \end{enumerate}
    \end{proposition}
    \begin{proof}
      \begin{enumerate}[a)]
        \item This follows from the definition of $\Omega_\design$ and
          $\Omega_\hyperparam$ and the fact that $\Omega_\hyperparam$ is the
          hypercube in $\Rnum^{\Nsens}$ whose vertices formulate the set $\Omega_\design$.
        \item For any realization $\hyperparam$,  the function $\stochobj$
          defines a convex combination of all points in
          $C$, because the coefficients
          $\CondProb{\design}{\hyperparam},\,\design\in\Omega_\design$ are
          probabilities satisfying 
          that $0\leq \CondProb{\design}{\hyperparam}\leq 1$ and
          $\sum_{\design\in\Omega_\design}\CondProb{\design}{\hyperparam} = 1$.
        \item Setting $\hyperparam$ to any $\vec{x}\in\Omega_\design$ yields a
          degenerate multivariate Bernoulli distribution, which is a Dirac
          measure  $\delta_{\design}(\{\vec{x}\})$ defined on the set $\{\vec{x}\}$.
          Thus, $\CondProb{\design}{\hyperparam=\vec{x}} = \delta_{\design}(\{\vec{x}\})$,
          and
          $ \stochobj(\vec{x})
            = \sum_{\design\in\Omega_\hyperparam}
              \CondProb{\design}{\hyperparam\!=\!\vec{x}}\obj(\design) 
            = \sum_{\design\in\Omega_\hyperparam}
              \delta_{\design}(\{\vec{x}\})\, \obj(\design)
            = \obj(\vec{x}) \,
            $. 
        \item It follows from Carath\'eodory's theorem that for any realization 
          of $\hyperparam\in\Omega_\hyperparam$, the
          value of the objective $\stochobj(\hyperparam)$ falls on a line segment 
          in $\Rnum$ that connects at most two points in $C$.
          This guarantees that $ \min\{C\} \leq \stochobj(\hyperparam) \leq
          \max\{C\}$.
          Additionally, from points a) and b) above, it follows that
          $\stochobj(\hyperparam\opt)=\min\{conv(C)\}=\min\{C\}=\obj(\design\opt)$.
      \end{enumerate}
    \end{proof}

    %
    %
    \begin{figure}[ht] \centering
      \resizebox{0.60\textwidth}{!}{%
      \tikzset{every picture/.style={line width=0.75pt}} 
        \begin{tikzpicture}[x=0.75pt,y=0.75pt,yscale=-1,xscale=1]
          \draw (114,60.4) node [anchor=north west][inner sep=0.75pt] 
            { $\Omega_\design :=\{0,1\}^{\Nsens}$ };
          \draw (415,60.4) node [anchor=north west][inner sep=0.75pt]
            { $\Omega_\hyperparam :=[ 0,1]^{\Nsens}$ };
          \draw (272,51) node [anchor=north west][inner sep=0.75pt] [align=left] 
            { 
              \begin{minipage}[lt]{37.304392pt}\setlength\topsep{0pt}
                \begin{center}
                  Convex\\Hull
                \end{center}
              \end{minipage}
            };
          \draw (151,110.4) node [anchor=north west][inner sep=0.75pt] 
            { $\obj( \design )$ };
          \draw (91,165.4) node [anchor=north west][inner sep=0.75pt]    
            { $C:=\{j_{1} ,j_{2} ,\dotsc ,j_{2^{\Nsens}}\}$ };
          \draw (362,165.4) node [anchor=north west][inner sep=0.75pt]    
            { $conv( C) :=[ \min \{ C\} , \max\{ C \}]$ };
          \draw (275,156) node [anchor=north west][inner sep=0.75pt] [align=left] 
            { 
              \begin{minipage}[lt]{37.304392pt}\setlength\topsep{0pt}
                \begin{center}
                  Convex\\Hull
                \end{center}
              \end{minipage}
            };
          \draw (452,114.4) node [anchor=north west][inner sep=0.75pt]  
            { $\stochobj(\hyperparam )$ };
            %
          \draw (220,70) -- (267,70) ;
          \draw [shift={(269,70)}, rotate = 180] 
                [color={rgb, 255:red, 0; green, 0; blue, 0 } ]
                [line width=0.75]    
                (10.93,-3.29) .. controls (6.95,-1.4) and 
                (3.31,-0.3) .. (0,0) .. controls (3.31,0.3) and 
                (6.95,1.4) .. (10.93,3.29)   ;
          \draw (328,70.18) -- (410,70.67) ;
          \draw [shift={(412,70.68)}, rotate = 180.34] 
                [color={rgb, 255:red, 0; green, 0; blue, 0 } ]
                [line width=0.75]    
                (10.93,-3.29) .. controls (6.95,-1.4) and 
                (3.31,-0.3) .. (0,0) .. controls (3.31,0.3) and 
                (6.95,1.4) .. (10.93,3.29)   ;
          \draw (165.35,84) -- (165.15,104) ;
          \draw [shift={(165.13,106)}, rotate = 270.6] 
                [color={rgb, 255:red, 0; green, 0; blue, 0 } ]
                [line width=0.75]  
                (10.93,-3.29) .. controls (6.95,-1.4) and 
                (3.31,-0.3) .. (0,0) .. controls (3.31,0.3) and 
                (6.95,1.4) .. (10.93,3.29)   ;
          \draw (243,174.78) -- (270,174.88) ;
          \draw [shift={(272,174.89)}, rotate = 180.21] 
                [color={rgb, 255:red, 0; green, 0; blue, 0 } ]
                [line width=0.75]    
                (10.93,-3.29) .. controls (6.95,-1.4) and 
                (3.31,-0.3) .. (0,0) .. controls (3.31,0.3) and 
                (6.95,1.4) .. (10.93,3.29)   ;
          \draw (331,174.64) -- (357,174.33) ;
          \draw [shift={(359,174.3)}, rotate = 539.31] 
                [color={rgb, 255:red, 0; green, 0; blue, 0 } ]
                [line width=0.75]    
                (10.93,-3.29) .. controls (6.95,-1.4) and 
                (3.31,-0.3) .. (0,0) .. controls (3.31,0.3) and 
                (6.95,1.4) .. (10.93,3.29)   ;
          \draw (165.11,130) -- (165.36,159) ;
          \draw [shift={(165.38,161)}, rotate = 269.49] 
                [color={rgb, 255:red, 0; green, 0; blue, 0 } ]
                [line width=0.75]    
                (10.93,-3.29) .. controls (6.95,-1.4) and 
                (3.31,-0.3) .. (0,0) .. controls (3.31,0.3) and 
                (6.95,1.4) .. (10.93,3.29)   ;
          \draw (466,85) -- (466,108) ;
          \draw [shift={(466,110)}, rotate = 270] 
                [color={rgb, 255:red, 0; green, 0; blue, 0 } ]
                [line width=0.75]   
                (10.93,-3.29) .. controls (6.95,-1.4) and 
                (3.31,-0.3) .. (0,0) .. controls (3.31,0.3) and 
                (6.95,1.4) .. (10.93,3.29)   ;
          \draw (466.12,134) -- (466.36,159) ;
          \draw [shift={(466.38,161)}, rotate = 269.44] 
                [color={rgb, 255:red, 0; green, 0; blue, 0 } ]
                [line width=0.75]    
                (10.93,-3.29) .. controls (6.95,-1.4) and 
                (3.31,-0.3) .. (0,0) .. controls (3.31,0.3) and 
                (6.95,1.4) .. (10.93,3.29)   ;
        \end{tikzpicture}
      }
      \caption{Relation between
        problems~\ref{eqn:binary_optimization} 
        and~\ref{eqn:stochastic_optimization} 
        and their respective domains and codomains.
      }
      \label{fig:two_problems}
    \end{figure}
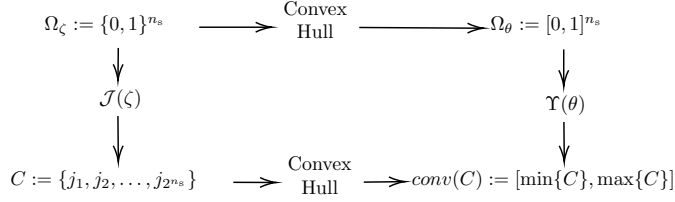

    \begin{lemma}\label{lemma:binary_stochastic_solutions}
      The optimal solutions of the two 
      problems~\eqref{eqn:binary_optimization} 
      and~\eqref{eqn:stochastic_optimization} are such that
      \[
          \argmin\limits_{\design\in\Omega_\design}{\obj(\design)} 
          \subseteq
          \argmin\limits_{\hyperparam\in\Omega_\hyperparam}{\stochobj(\hyperparam)}
        \,.
      \]
      Moreover, if the optimal solution $\design\opt$
      of~\eqref{eqn:binary_optimization} is unique, then
      $\hyperparam\opt=\design\opt$, where $\hyperparam\opt$ is the unique
      optimal solution of~\eqref{eqn:stochastic_optimization}.
    \end{lemma}
    \begin{proof}
      \Cref{prop:binary_stochastic_connection} guarantees that
      $\nexists \hyperparam \in\Omega_\hyperparam ,\,
        \stochobj(\hyperparam) < \obj(\design\opt)$. 
      Additionally, from points a) and c)
      in~\Cref{prop:binary_stochastic_connection}, it follows that
      $
        \argmin\limits_{\design\in\Omega_\design}{\obj(\design)} 
        \subseteq
        \argmin\limits_{\hyperparam\in\Omega_\hyperparam}{\stochobj(\hyperparam)}
        \,
      $.
      
      Assume $\design\opt$ is the unique optimal solution
      of~\eqref{eqn:binary_optimization};
      it follows that $\hyperparam\opt=\design\opt$ is an optimal solution
      of~\eqref{eqn:stochastic_optimization}.
      Now assume $\exists \hyperparam\in\Omega_\hyperparam$
      such that $\hyperparam \neq \hyperparam\opt ,\,
        \stochobj(\hyperparam)
        =\stochobj(\hyperparam\opt)
        =\obj(\design\opt)
        $. 
      If $\hyperparam$ is the parameter of a degenerate Bernoulli distribution,
      then $\hyperparam\in\Omega_\design$ and $\obj(\hyperparam)=\obj(\design\opt)$,  
      thus contradicting  the uniqueness of $\design\opt$.
      Conversely, if $\hyperparam$ is the parameter of a nondegenerate
      multivariate Bernoulli distribution, 
      then there are at least two designs $\design, \vec{\eta}\in\Omega_\design$  
      with nontrivial probabilities $\in(0, 1)$, with
      $\obj(\design)=\design(\vec{\eta})$,  again contradicting  the
      uniqueness assumption of $\design\opt$.
    \end{proof}

    %

    \commentout{  
      The connection between the two problems defined 
      by~\ref{eqn:binary_optimization} 
      and~\ref{eqn:stochastic_optimization}), 
      and their respective solutions, is summarized
      in~\Cref{lemma:binary_stochastic_connection_1},
      and~\Cref{lemma:binary_stochastic_connection_2}.
      \begin{lemma}\label{lemma:binary_stochastic_connection_1}
        Let $\design \in\Omega_\design:=\{0,1\}^{\Nsens}$  be a random vector following
        the multivariate Bernoulli
        distribution~\eqref{eqn:joint_Bernoulli_pmf_prod}, 
        with parameter $\hyperparam \in \Omega_\hyperparam:=[0,1]^{\Nsens}$,
        and let $\obj:\Omega_\hyperparam\!\rightarrow\!\Rnum$ be either convex or
        concave over its domain. 
        The solutions of the two optimization problems defined by 
        \eqref{eqn:stochastic_optimization} and \eqref{eqn:binary_optimization}
        coincide; that is,
        \begin{equation*}
          \hyperparam\opt = \argmin_{\hyperparam \in [0, 1]^{\Nsens} }  
            \Expect{\design\sim\CondProb{\design}{\hyperparam}}{\obj(\design)} 
          \equiv 
          \argmin_{\design \in \{0, 1\}^{\Nsens}} \obj(\design) 
          = \design\opt \,.
        \end{equation*}
      \end{lemma}
      \begin{proof}
        The proof follows immediately from Jensen's Inequality 
        and by noting that if $\obj$ is convex, then 
        $\Expect{\design\sim \CondProb{\design}{\hyperparam}}{\obj(\design)}
          \geq \obj(\Expect{\design\sim \CondProb{\design}{\hyperparam}}{\design}) \, 
          \forall \design\in\Omega_\design\,.$
        Conversely, if $\obj$ is concave, then 
        $\Expect{\design\sim \CondProb{\design}{\hyperparam}}{\obj(\design)}
          \leq \obj(\Expect{\design\sim \CondProb{\design}{\hyperparam}}{\design}) \, \forall
          \design\in\Omega_\design\,.$
        Since $\design$ follows a multivariate Bernoulli distribution with
        parameter $\hyperparam\in[0, 1]^{\Nsens}$, then the
        probability distribution of $\design$ becomes degenerate at the extremal points 
        of the domain of $\hyperparam$ because $\CondProb{\design}{\hyperparam}$
        in~\eqref{eqn:joint_Bernoulli_pmf_prod} evaluates to $\{0,1\}$ 
        for $\hyperparam\in\{0,1\}^\Nsens$. 
        It then follows that 
        $\Expect{\design\sim \CondProb{\design}{\hyperparam}}{\obj(\design)}
          = \obj(\Expect{\design\sim \CondProb{\design}{\hyperparam}}{\design})
        $ at the extremal points of the domain of $\hyperparam$. 
        This shows that the optima of the two problems
        (\Cref{eqn:binary_optimization} and \Cref{eqn:stochastic_optimization})
        coincide.
      \end{proof}
      
      Note that~\Cref{lemma:binary_stochastic_connection_1} is not limited to the
      case where $\obj$ is either convex or concave over $\Omega_\hyperparam$ 
      since our interest is only in the value of $\obj$ evaluated at binary designs, 
      that is, at $\design\in\Omega_\design$. 
      This fact is stated by~\Cref{lemma:binary_stochastic_connection_2}.
      \begin{lemma}\label{lemma:binary_stochastic_connection_2}
        Let $\design$ be defined as
        in~\Cref{lemma:binary_stochastic_connection_2}.
        Then for any function $\obj:\Omega_\design\!\rightarrow\!\Rnum$,
        the solutions of the two optimization problems defined by 
        \eqref{eqn:stochastic_optimization} and \eqref{eqn:binary_optimization},
        respectively, coincide. 
      \end{lemma}
      \begin{proof}
        Since the expected value is a measure of central tendency, then 
        for the optimal binary design $\design\opt \in\Omega_\design$ it holds that
        $\obj(\design\opt)\leq\obj(\hyperparam)\,\forall\hyperparam\in(0,1)^{\Nsens}$.
        This, together with the statement of~\Cref{lemma:binary_stochastic_connection_2}, 
        means that the optimal solution for both problems defined by
        \eqref{eqn:binary_optimization}and ~\eqref{eqn:stochastic_optimization},
        respectively, is attained only at extremal points of the domain of $\hyperparam$.
        Moreover, if we assume the value of $\obj$ is unique for all instances of
        $\design$, then there is a unique global optimum for the two problems.
      \end{proof}
        
      Further discussion of the connection of the two problems in the context of a simple 
      toy model is provided in the numerical experiments
      in~\Cref{sec:numerical_experiments}.
    }

    The stochastic formalism proposed
    in~\Cref{subsec:stochastic_OED} is beneficial for multiple reasons.
    First, this probabilistic formulation~\eqref{eqn:stochastic_optimization} 
    to solving the original optimization problem~\eqref{eqn:binary_optimization} 
    enables converting a binary design domain into a bounded continuous domain, 
    where the optimal solution of the two problems \eqref{eqn:stochastic_optimization}
    and~\eqref{eqn:binary_optimization} coincide.
    Second, the stochastic formulation~\eqref{eqn:stochastic_optimization} 
    enables utilizing efficient stochastic optimization
    algorithms to solve binary optimization problems, 
    which are not applicable for the deterministic binary optimization
    problem~\eqref{eqn:binary_optimization}.
    Third, the solution of the stochastic
    problem~\eqref{eqn:stochastic_optimization} is an optimal parameter
    $\hyperparam\opt$ that can be used for sampling binary designs
    $\design$ by sampling $\CondProb{\design}{\hyperparam\opt}$ 
    , even if only a suboptimal solution is found.
    Fourth, the stochastic formulation~\eqref{eqn:stochastic_optimization} 
    requires evaluating the derivative of $\stochobj$ with
    respect to the parameter of the probability distribution $\hyperparam$,
    instead of the design $\design$. 
    Thus,  we do not need to evaluate the
    derivative of $\obj$ with respect to the design $\design$.
    In OED for Bayesian inversion, as discussed in~\Cref{sec:background}, 
    evaluating the gradient (with respect to the design) is very expensive 
    since it requires many forward and backward solves of expensive
    simulation models.
  Fifth, because $\obj$ is not required to be differentiable with respect to
    the design $\design$, a nonsmooth penalty function $\Phi$ can be utilized in
    defining $\obj$, for example to enforce sparsity or budget constraints, 
    with the need to approximate this penalty with a smoother penalty approximation.
    Note that in this case, as explained by
    (\ref{eqn:binary_optimization} and \ref{eqn:stochastic_optimization}),
    the penalty is asserted on the design, rather than
    the distribution parameter $\hyperparam$.

  \subsection{Approximately solving the stochastic OED optimization problem}
    \label{subsec:stochastic_OED_solution}
    The objective in~\eqref{eqn:stochastic_optimization_expanded} and the
    gradient~\eqref{eqn:exact_gradient} amount to evaluating a large finite sum,
    which can be approximated by sampling.
    The simplest approach for approximately solving ~\eqref{eqn:stochastic_optimization_expanded} 
    is to follow a Monte Carlo (MC) sample-based approximation to the 
    objective $\stochobj$ and the gradient $\nabla_{\hyperparam}\stochobj$, 
    where the sample is drawn from an appropriate probability distribution.
    However, the objective~\eqref{eqn:stochastic_optimization} is 
    deterministic with respect to $\hyperparam$. 
    Nevertheless, one can add randomness by assuming a uniform random variable 
    (an index) over all terms of the objective and the associated gradient.
    This allows utilizing external and internal sampling-based stochastic 
    optimization algorithms; see, for
    example,~\cite{mak1999monte,kleywegt2002sample,dupacova1988asymptotic,king1993asymptotic,shapiro1991asymptotic,defazio2014saga,reddi2016fast,schmidt2017minimizing}. 
    These algorithms, however, are beyond the scope of this work.

    Here, instead, we focus on efficient algorithmic procedures 
    inspired by recent developments in reinforcement 
    learning~\cite{BertsekasTsitsiklis96,williams1992simple,sutton2000policy}. 
    Specifically, we view $\CondProb{\design}{\hyperparam}$ as a policy and 
    the design $\design$ as a state, and we seek an optimal policy to configure 
    the observational grid. 
    The reward is prescribed by the value of the original objective function, 
    and we seek the optimal design that minimizes the total expected reward.
    This is further explained below.

    An alternative form of the gradient~\eqref{eqn:exact_gradient} can be
    obtained by using the \textit{``kernel trick''}, which utilizes the fact that
    $\nabla_{\hyperparam}{\log(\CondProb{\design}{\hyperparam})} = \frac{1}{\CondProb{\design}{\hyperparam}} \nabla_{\hyperparam}{\CondProb{\design}{\hyperparam}}
    $, and thus 
    $\nabla_{\hyperparam}{\CondProb{\design}{\hyperparam}} ={\CondProb{\design}{\hyperparam}} \, \nabla_{\hyperparam}{\log(\CondProb{\design}{\hyperparam})} $.
    Using this identity, we can rewrite the gradient~\eqref{eqn:exact_gradient}
     as
    \begin{equation}\label{eqn:kernel_exact_gradient}
      \begin{aligned}
        \nabla_{\hyperparam}\, \stochobj(\hyperparam)
          = \sum_{k=1}^{2^{\Nsens}} 
          \Bigl(
              \obj(\design[k])\,
              \nabla_{\hyperparam}\, \log{\CondProb{\design[k]}{\hyperparam} }
              \Bigr)
            \CondProb{\design[k]}{\hyperparam}
          %
          =\Expect{\design\sim \CondProb{\design}{\hyperparam}}{ \obj(\design)\,
            \nabla_{\hyperparam}\, \log{\CondProb{\design}{\hyperparam} } }  \,.
      \end{aligned}
    \end{equation}

    We assume, without loss of generality, that $\hyperparam$ falls inside the
    open ball $(0, 1)^{\Nsens}$, and thus both the logarithm and the associated
    derivative of the log-probability are well defined.
    If any of the components of $\hyperparam$ attain their bound, then the
    distribution becomes degenerate in this direction, and the probability is
    set to either $0$ or $1$ and is thus taken out of the formulation since the
    gradient in that direction is set to $0$.
    This is equivalent to projecting $\hyperparam$ onto a lower-dimensional
    subspace. 

    The form of the gradient described by~\eqref{eqn:kernel_exact_gradient} is
    equivalent to~\eqref{eqn:exact_gradient}. However, it shows that the gradient
    can be written as an expectation of gradients.
    This enables us to approximate the gradient using MC sampling
    by following a stochastic optimization approach.
    %
    Specifically. given a sample $\design[j] \sim \CondProb{\design}{\hyperparam},\,
    j=1,2,\ldots,\Nens$, we can use the following MC approximation of the gradient,
    \begin{equation}\label{eqn:kernel_stochastic_gradient}
      \vec{g}:=
      \nabla_{\hyperparam}\, \Expect{\CondProb{\design}{\hyperparam}}{\obj(\design)}
        \approx
        \widehat{\vec{g}} :=
        \frac{1}{\Nens} \sum_{j=1}^{\Nens}
          \obj(\design[j]) \nabla_{\hyperparam} \,
          \log{ \CondProb{\design[j]}{\hyperparam } }  \,,
    \end{equation}
    where $\nabla_{\hyperparam} \log{ \CondProb{\design[j]}{\hyperparam }}$ is
    the score function of the multivariate Bernoulli distribution
    (see~\Cref{app:Bernoulli} for details).
    By combining \eqref{eqn:kernel_stochastic_gradient} with~\eqref{eqn:grad_log_Bernoulli},
    we obtain the following form of the stochastic gradient:
    \begin{equation}\label{eqn:kernel_stochastic_gradient_Bernoulli}
        \widehat{\vec{g}}
        = \frac{1}{\Nens} \sum_{j=1}^{\Nens}
            \obj(\design[j])
          \sum_{i=1}^{\Nsens} \left(
          \frac{\design_i[j]}{\hyperparam_i} + \frac{\design_i[j]-1}{1-\hyperparam_i}
        \right) \,\vec{e}_i \,.
    \end{equation}

    To solve the optimization problem~\eqref{eqn:stochastic_optimization}, one
    can start with an initial policy, in other words, an initial set of parameters
    $\hyperparam$, and iteratively follow an approximate descent direction until an
    approximately \textit{optimal policy} is obtained.
    
    Because the optimization problems described here require box constraints, we
    introduce the projection operator $\proj{}{}$, which maps an arbitrary point
    $\hyperparam$ onto the feasible region described by the box constraints.
    One can apply projection by truncation (see~\cite[Section 16.7]{nocedal2006numerical}):
    \begin{equation}\label{eqn:box_projection}
      \Proj{}{\hyperparam_i} := \min\{1, \max\{0,\hyperparam_i\} \}
      \equiv
      \begin{cases}
        0 \quad&\text{ if}\quad \hyperparam_i < 0 \,,  \\
        \hyperparam_i \quad&\text{ if}\quad \hyperparam_i \in [0, 1] \,,  \\
        1 \quad&\text{ if}\quad \hyperparam_i > 1 \,,
      \end{cases}
    \end{equation}
    where the projection $\Proj{}{\hyperparam}$ is applied componentwise to the parameters vector
    $\hyperparam$.
    Alternatively, the following metric projection operator can  be utilized:
    \begin{equation}\label{eqn:metric_projection}
      \Proj{}{\hyperparam} := \argmin_{\hyperparam^{\prime}\in[0,1]^{\Nsens} }
        \wnorm{\hyperparam^{\prime}-\hyperparam}{2} \,.
    \end{equation}

    Note that the projection operator $\proj$  is
    nonexpansive and thus is orthogonal. This guarantees that for any
    $\Proj{}{\hyperparam[1]},\,\Proj{}{\hyperparam[2]}\in\Rnum^{\Nsens} $, it
    follows that
    \begin{equation}\label{eqn:projected_bound}
      \norm{ \Proj{}{\hyperparam[1]} - \Proj{}{\hyperparam[2]} }
        \leq \norm{\hyperparam[1] - \hyperparam[2] } \,.
    \end{equation}

    A stochastic steepest-descent step to solve~\eqref{eqn:stochastic_optimization} 
    is approximated by the stochastic approximation of the
    gradient~\eqref{eqn:kernel_stochastic_gradient_Bernoulli}
    and is described as
    \begin{equation}\label{eqn:stoch_steep_step}
      \hyperparam^{(n+1)}
          = \Proj{}{\hyperparam^{(n)} - \eta^{(n)} \widehat{\vec{g}}^{(n)} } \,,
    \end{equation}
    where $\eta^{(n)}$ is the step size (learning rate) at the $n$th iteration.
    \Cref{eqn:stoch_steep_step} describes a stochastic gradient descent approach, 
    using the stochastic approximation~\eqref{eqn:kernel_stochastic_gradient_Bernoulli} 
    at each iteration in place of $g^{(n)}$,  where the sample is generated
    from $\CondProb{\design}{\hyperparam^{(n)}}$.
    One can choose a fixed step size or follow a decreasing schedule to
    guarantee convergence or can do a line search using the sampled design to
    approximate the objective function.
    Additionally, one can use approximate second-order information and create a
    quasi-Newton-like step.
    
    \autoref{alg:REINFORCE} describes a stochastic descent approach 
    for solving~\eqref{eqn:binary_optimization} by
    solving~\eqref{eqn:stochastic_optimization} followed by a sampling step. 
    Note that because $\design$ is sampled from a multivariate Bernoulli
    distribution with parameter $\hyperparam$,
    in~\autoref{algstep:REINFORCE:grad} 
    if $\hyperparam_i\in\{0, 1\}$, then $\design_i=\hyperparam_i$. Thus the
    corresponding term in the summation vanishes.

    \begin{algorithm}
      \caption{Stochastic optimization algorithm for binary OED.}
      \label{alg:REINFORCE}
      \begin{algorithmic}[1]
        \Require{Initial distribution parameter $\hyperparam^{(0)}$,
                  step size schedule $\eta^{(n)}$, 
                  $\Nens,\, m$}
        \Ensure{$\design\opt$}

        \State{Initialize $n = 0$}

        \While{Not Converged}
          \State {
            Update $n\leftarrow n+1$ 
          }
          
          \State{
            Sample $\{\design[i]; i=1,2,\ldots,\Nens \}\sim
              \CondProb{\design}{\hyperparam^{(n)}}$ 
          }
        
          \State{
            Calculate $ \widehat{g}^{(n)}=\frac{1}{\Nens} \sum_{j=1}^{\Nens}
              \left(\obj(\design[j])\right) \sum_{i=1}^{\Nsens} 
              \left( 
                \frac{\design_i[j]}{\hyperparam_i} 
                  + \frac{\design[j]_i-1}{1-\hyperparam_i}
              \right) \,\vec{e}_i $
          }\label{algstep:REINFORCE:grad}
        
          \State {
            Update $\hyperparam^{(n+1)} 
              = \Proj{}{\hyperparam^{(n)} - \eta^{(n)} \widehat{g}^{(n)} } $
          }\label{algstep:REINFORCE:proj}
          
        \EndWhile
        
        \State{Set 
          $\hyperparam\opt = \hyperparam^{(n)}$
        }

        \State\label{algstep:REINFORCE:sampling}{Sample 
          $\{\design[j];j=1,2,\ldots,m \} \sim 
            \CondProb{\design}{\hyperparam\opt}$,
          and calculate $\obj(\design[j])$
        }
   
        \State\Return{
            $\design\opt $: the design $\design$ with smallest 
              value of $\obj$ in the sample.
        }

      \end{algorithmic}
    \end{algorithm}
    
    \Cref{alg:REINFORCE} is a stochastic gradient descent method 
    with convergence guarantees in expectation only;
    see~\Cref{subsec:convergence_analysis}.
    In practice, the convergence test 
    could be replaced with a maximum number of iterations. Doing so,
    however, might not result in degenerate PMF.
    The output of the algorithm $\hyperparam\opt=\hyperparam^{(a)}$
    is used for sampling binary designs $\design$, by sampling
    $\CondProb{\design}{\hyperparam\opt}$, even if a suboptimal solution is found.
    Another potential convergence criterion is the magnitude of the projected gradient
    $\norm{ \Proj{}{g^{(k)} } } \leq \textrm{pgtol} $.
    We also note that~\Cref{alg:REINFORCE} is equivalent to the vanilla policy gradient
    REINFORCE algorithm~\cite{williams1992simple}.
    The remainder of~\Cref{sec:problem_formulation} is devoted 
    to addressing convergence guarantees, convergence analysis, 
    and improvements of~\Cref{alg:REINFORCE}.

  \subsection{Convergence analysis}
    \label{subsec:convergence_analysis}
    Here, we show that 
    the optimal solution of~\eqref{eqn:stochastic_optimization_expanded} is 
    a degenerate multivariate Bernoulli distribution and that the 
    convergence of~\Cref{alg:REINFORCE} to such an optimal distribution in
    expectation is guaranteed under mild conditions. 
    
    First, we study the properties of the exact objective defined 
    by~\eqref{eqn:stochastic_optimization_expanded} and the associated
    exact gradient~\eqref{eqn:exact_gradient}. We then address the stochastic
    approximation and the performance of~\Cref{alg:REINFORCE}.

    \subsubsection{Analysis of the exact stochastic optimization problem}
      \label{subsubsec:convergence_analysis_deterministic}
      Recall that the objective function utilized
      in~\eqref{eqn:stochastic_optimization_expanded} and the associated gradient take 
      the respective forms
      \begin{equation}
        \stochobj(\hyperparam) 
          =\sum_{k=1}^{2^{\Nsens}} \obj(\design[k])
          \CondProb{\design[k]}{\hyperparam} 
          \,, \quad
        \nabla_{\hyperparam}\,\stochobj(\hyperparam) \equiv \vec{g}
          =\sum_{k=1}^{2^{\Nsens}} \obj(\design[k])
            \nabla_{\hyperparam}\, \CondProb{\design[k]}{\hyperparam} \,,
      \end{equation}
      where $\hyperparam \in \Omega_\hyperparam:=[0, 1]^{\Nsens}$, 
      $\CondProb{\design}{\hyperparam}$ is the multivariate Bernoulli
      distribution~\eqref{eqn:joint_Bernoulli_pmf_prod},
      and we use the indexing scheme~\eqref{eqn:binary_index}.
      Note that we are viewing the objective $\stochobj$ explicitly as a
      function of the PMF parameter $\hyperparam$, because all possible
      combinations of binary designs $\design[k];\, k=1,2,\ldots,2^{\Nsens}$ are
      present in the expectation. 
      
      Here, we show that the exact gradient of $\stochobj$ is bounded
      (\Cref{lemma:exact_gradient_bound}) and that the Hessian of $\stochobj$ 
      is bounded; hence, by following a steepest-descent approach, a locally optimal 
      design is obtained. 
      This will set the ground for an analysis of~\Cref{alg:REINFORCE}.

      \begin{lemma}\label{lemma:exact_gradient_bound}
        Let $\design \in \Omega_{\design} = \{0, 1\}^{\Nsens}$, and let 
        $ C = \max\limits_{\design\in \Omega_{\design}}{\{ \left|\obj(\design)\right| \}}$.
        Then the following bounds of the gradient of the stochastic
        objective $\stochobj$ hold,
        \begin{subequations}\label{eqn:exact_gradient_bounds}
        \begin{eqnarray} 
            \norm{ \nabla_{\hyperparam}\,\stochobj(\hyperparam) } 
            &\leq &
            \sqrt{\Nsens \left| \Omega_{\design} \right| } \, C 
            \,, \quad \hyperparam \in \Omega_\hyperparam \,,
              \label{eqn:exact_gradient_norm_bound}  \\
            \norm{ 
              \nabla_{\hyperparam}\,\stochobj(\hyperparam[1]) - 
              \nabla_{\hyperparam}\,\stochobj(\hyperparam[2]) 
            } &\leq& 2 \sqrt{ \Nsens \left| \Omega_{\design} \right| } \, C  
            \,, \quad \hyperparam[1], \hyperparam[2] \in \Omega_\hyperparam \,,
              \label{eqn:exact_gradient_Lipschitz}
        \end{eqnarray}
        \end{subequations}
        where $\left| \Omega_{\design} \right| = 2^{\Nsens}$ is the cardinality
        of the design domain.
        Moreover, the Hessian is bounded  by
        \begin{equation}\label{eqn:Hessian_entries_bound}
          \abs{ \delll{\stochobj}{\hyperparam_i}{\hyperparam_j} } 
            \leq \sqrt{\left| \Omega_{\design} \right|} \, C  \,.
        \end{equation}
      \end{lemma}
      \begin{proof}
        The first bound is obtained as follows: 
        \begin{equation}\label{eqn:exact_gradient_norm_bound_1}
          \begin{aligned}
            \sqnorm{ \nabla_{\hyperparam}\,\stochobj(\hyperparam) } 
            &= \sqnorm{ \sum_{k=1}^{2^{\Nsens}} \obj(\design[k])
              \nabla_{\hyperparam}\, \CondProb{\design[k]}{\hyperparam}
              } \\ 
            &\leq \sum_{k=1}^{2^{\Nsens}} 
             \sqnorm{ \obj(\design[k])
              \nabla_{\hyperparam}\, \CondProb{\design[k]}{\hyperparam} }  \\
            &= \sum_{k=1}^{2^{\Nsens}} 
             \left( \obj(\design[k]) \right)^2
              \sqnorm{ \nabla_{\hyperparam}\, \CondProb{\design[k]}{\hyperparam} } 
            \leq C^2 \sum_{k=1}^{2^{\Nsens}} 
              \sqnorm{ \nabla_{\hyperparam}\, \CondProb{\design[k]}{\hyperparam} }  \,.  
          \end{aligned}
        \end{equation}

        In~\Cref{app:Bernoulli}, we derive a bound on 
        $\nabla_{\hyperparam}\, \CondProb{\design[k]}{\hyperparam}$, 
        the gradient of log probability of the multivariate Bernoulli
        distribution~\eqref{eqn:joint_Bernoulli_pmf_prod}.
        Specifically,~\Cref{lemma:Bernoulli_PMF_gradient_norm_bound} shows
        that
        $
        \norm{ \nabla_{\hyperparam}\, \CondProb{\design}{\hyperparam} }
        \leq
          \sqrt{\Nsens} \,
            \max\limits_{j=1,\ldots,\Nsens}
              \, \min\limits_{\substack{k=1,\ldots,\Nsens\\k\neq j}}
            \CondProb{\design_k}{\hyperparam_k}
        $. Hence it follows from~\eqref{eqn:exact_gradient_norm_bound_1} that
        \begin{equation}
          \begin{aligned}
            \sqnorm{ \nabla_{\hyperparam}\,\stochobj(\hyperparam) } 
            %
            %
            &\leq \Nsens\, C^2 \sum_{i=1}^{2^{\Nsens}}
              \max\limits_{j=1,\ldots,\Nsens}
              \, \min\limits_{\substack{k=1,\ldots,\Nsens\\k\neq j}}
            \CondProb{\design_k[i]}{\hyperparam_k} \\
            &\leq \Nsens 2^{\Nsens} C^2 
              \max\limits_{i=1,\ldots,2^{\Nsens}}
              \max\limits_{j=1,\ldots,\Nsens}
              \min\limits_{\substack{k=1, \ldots,\Nsens\\k\neq j}}
            \CondProb{\design_k[i]}{\hyperparam_k} 
            \,.
          \end{aligned}
        \end{equation}

        Because $\CondProb{\design}{\hyperparam}\leq 1$, it follows immediately
        that 
        $ 
        \sqnorm{ \nabla_{\hyperparam}\,\stochobj(\hyperparam) } 
          \leq \Nsens\, 2^{\Nsens} \, C^2
        $, and the bound~\eqref{eqn:exact_gradient_norm_bound} follows by taking 
        the square root on both sides.
        The second bound~\eqref{eqn:exact_gradient_Lipschitz} follows
        immediately by noting that
        \begin{equation}
          \norm{ 
            \nabla_{\hyperparam}\,\stochobj(\hyperparam[1]) - 
            \nabla_{\hyperparam}\,\stochobj(\hyperparam[2]) 
          }
          \leq
          \norm{\nabla_{\hyperparam}\,\stochobj(\hyperparam[1]) }
          + \norm{\nabla_{\hyperparam}\,\stochobj(\hyperparam[2]) } 
          \leq
          2 \sqrt{ \Nsens \left| \Omega_{\design} \right| } C  \,.  
        \end{equation}

      The second-order derivative of the objective $\stochobj$ is
      \begin{equation}\label{eqn:eqn:stochastic_optimization_Hessian}
        \mathrm{Hess}_{\hyperparam} \stochobj 
          = \nabla_{\hyperparam}\nabla_{\hyperparam} \stochobj
          = \sum_{\design} \obj(\design)
             \nabla_{\hyperparam} \nabla_{\hyperparam}\, 
               \CondProb{\design}{\hyperparam}  
          = \sum_{i=k}^{2^{\Nsens}} \obj(\design[k])
             \nabla_{\hyperparam} \nabla_{\hyperparam}\, 
               \CondProb{\design[k]}{\hyperparam}  
         \,.
      \end{equation}

      \Cref{eqn:eqn:stochastic_optimization_Hessian} shows that the 
      $(i,j)$th entry of the Hessian~\eqref{eqn:eqn:stochastic_optimization_Hessian} 
      is
      $
      \sum_{k=1}^{2^{\Nsens}} \obj(\design[k])
        \delll{\CondProb{\design}{\hyperparam}}{\hyperparam_i}{\hyperparam_j}
      $. The second-order derivative of the Bernoulli distribution is discussed
      in detail in~\Cref{app:Bernoulli}, and it follows 
      from~\eqref{eqn:Bernoulli_PMF_Hessian} that
      $ \delll{\CondProb{\design}{\hyperparam}}{\hyperparam_i}{\hyperparam_j}
        \leq1$. 
      This means that the entries of the Hessian
      ~\eqref{eqn:eqn:stochastic_optimization_Hessian} are bounded as follows,
      \begin{equation}
        \left({ \delll{\stochobj}{\hyperparam_i}{\hyperparam_j} }\right)^2 
          = \left({ \sum_{i=k}^{2^{\Nsens}} \obj(\design[k])
              \delll{\CondProb{\design}{\hyperparam}}{\hyperparam_i}{\hyperparam_j}
            }\right)^2
          \leq
          C^2 \sum_{k=1}^{2^{\Nsens}}\left({
              \delll{\CondProb{\design[k]}{\hyperparam}}{\hyperparam_i}{\hyperparam_j}
            }\right)^2
          \leq C^2 2^{\Nsens} \,,
      \end{equation}
      and the bound in~\eqref{eqn:Hessian_entries_bound} follows by taking
      square root of both sides.
      $\qquad$  
      \end{proof}
      
      %
      
      \Cref{lemma:exact_gradient_bound} is especially interesting because 
      it shows that the gradient of $\stochobj(\hyperparam)$ is bounded. 
      On the contrary, $\obj(\design)$ does not have bounded gradients for
      $\ell_0$ regularization.
      Additionally, it follows from~\Cref{lemma:exact_gradient_bound} that 
      the entries of the Hessian are bounded,  implying that 
      the Hessian is bounded. Thus the objective $\stochobj$ is Lipschitz smooth.
     Finding the Lipschitz constant, however, depends on the value of the
      function $\obj$ evaluated at all possible values of $\design$, or at least
      the maximum value.
      This is impossible to know a priori, without any prior knowledge about the
      problem in hand.
      Moreover, the bound given above indicates that the Lipschitz constant
      decreases for higher dimensions, since it is exponentially proportional to the 
      cardinality of the design space $\Omega_{\design}$.
      However, one can estimate the Hessian matrix by using an ensemble of
      realizations 
      and use it to estimate the Lipschitz constant. This can be
      helpful for choosing a proper step size for a steepest-descent
      algorithm.
      Alternatively, one can use a decreasing step-size sequence
      $\{\eta^(k)\}_{k=0}^{\infty}$ such that $\lim_{k\rightarrow
      \infty}\eta^(k)=0$ and $\sum_{k}\eta^(k)=\infty$, which guarantees 
      convergence to a local optimum; see~\cite[Proposition 4.1]{bertsekas1996neuro}.
      %

    \subsubsection{Analysis of stochastic steepest-descent algorithm}
      \label{subsubsec:convergence_analysis_stochastic}
      Here we show that $\widehat{\vec{g}(\hyperparam)}$ is an unbiased estimator of the 
      true gradient $\vec{g}(\hyperparam)$. 
      This fact, along with~\Cref{lemma:exact_gradient_bound} 
      and the fact that $\stochobj$ is a convex combination,
      guarantees convergence, in expectation, of~\Cref{alg:REINFORCE} to a
      locally optimal policy. 
      At each iteration $n$ of ~\Cref{alg:REINFORCE},
      the gradient is approximated with a sample from the respective conditional
      distribution $\CondProb{\design}{\hyperparam^{(n)}}$.
      First, we note that 
      \begin{equation}
        \begin{aligned}
          \norm{\widehat{\vec{g}}}
            &= \norm{
              \frac{1}{\Nens} \sum_{j=1}^{\Nens}
              \obj(\design[j]) \nabla_{\hyperparam} \log{ \CondProb{\design[j]}{\hyperparam} }
            } 
            %
            %
            %
            \leq \frac{1}{\Nens}
              \sum_{j=1}^{\Nens} | \obj(\design[j]) |
              \norm{ \nabla_{\hyperparam} \log{ \CondProb{\design[j]}{\hyperparam} }
            }  \,,
        \end{aligned}
      \end{equation}
      which shows that the magnitude of $\widehat{g}$ is bounded. Moreover, we will
      show next 
      that $\widehat{\vec{g}}(\hyperparam)$ 
      is an unbiased estimator of $\vec{g}(\hyperparam)$ and that the variance
      of this estimator is bounded. 

      \begin{lemma}\label{lemma:g_stats}
        The stochastic estimator $\widehat{\vec{g}}$ defined
        by~\eqref{eqn:kernel_stochastic_gradient} is unbiased, with sampling
        total variance $\brVar{\widehat{\vec{g}}}$, such that
        \begin{equation}\label{eqn:stoch_grad_unbiased}
          \Expect{}{\widehat{\vec{g}}} = \vec{g} =
          \nabla_{\hyperparam}\stochobj(\hyperparam) \,; \quad
            \brVar{\widehat{\vec{g}}}
              = \frac{1}{\Nens}
              \brVar{\obj(\design)\nabla_{\hyperparam}{\log{\CondProb{\design}{\hyperparam}}}}
              \,,
        \end{equation}
        where the total variance operator evaluates 
        the trace of the variance-covariance matrix of the random vector.  
        Moreover, the total variance of $\widehat{\vec{g}}$ is bounded, and there exist some positive 
        constants $K_1,\, K_2$, such that
        %
        \begin{equation}\label{eqn:stoch_grad_norm_bound}
         \Expect{}{\widehat{\vec{g}}\tran \widehat{\vec{g}} }
          = \expect{}{\sqnorm{\widehat{\vec{g}} } }
            \leq K_1 + K_2 \sqnorm{ \vec{g} }
            = K_1 + K_2 \vec{g}\tran \vec{g} \,.
        \end{equation}
      \end{lemma}
      \begin{proof}
        For any realization of the random design 
        $\design\sim \CondProb{\design}{\hyperparam}$, it holds that 
        \begin{equation}\label{eqn:unbiased_estimate}
          \begin{aligned}
            \Expect{}{
              \obj(\design) \nabla_{\hyperparam}
              \log{ \CondProb{\design}{\hyperparam } }
            }
            &= \sum_{\design}
              \obj(\design) \nabla_{\hyperparam}
              \log{ \CondProb{\design}{\hyperparam } }
              \CondProb{\design}{\hyperparam } 
            = \sum_{\design}
              \obj(\design) \nabla_{\hyperparam}
              \CondProb{\design}{\hyperparam }  \\
            &= \nabla_{\hyperparam} \sum_{\design}   \obj(\design)
              \CondProb{\design}{\hyperparam }  
            = \nabla_{\hyperparam}
              \Expect{}{\obj(\design) }
            = \nabla_{\hyperparam} \stochobj(\hyperparam) \,.
          \end{aligned}
        \end{equation}

        Thus, unbiasedness of the estimator $\widehat{\vec{g}}$ follows because
        \begin{equation}
          \begin{aligned}
            \Expect{}{\widehat{\vec{g}}}
            &= \Expect{}{
              \frac{1}{\Nens} \sum_{j=1}^{\Nens}
              \obj(\design[j]) \nabla_{\hyperparam} \,
              \log{ \CondProb{\design[j]}{\hyperparam } }
            } \\
            &= \frac{1}{\Nens} \sum_{j=1}^{\Nens}
              \Expect{}{
                \obj(\design[j]) \nabla_{\hyperparam} \,
              \log{ \CondProb{\design[j]}{\hyperparam } }
            }
            \stackrel{(\ref{eqn:unbiased_estimate})}{=} 
            \frac{1}{\Nens} \sum_{j=1}^{\Nens}
              \nabla_{\hyperparam} \stochobj(\hyperparam)
            =  \nabla_{\hyperparam} \stochobj(\hyperparam) \,.
          \end{aligned}
        \end{equation}

        The total variance $\brVar{\widehat{\vec{g}}}$ follows by definition of
        $\widehat{\vec{g}}$ as
        \begin{equation}\label{eqn:total_var}
          \brVar{\widehat{\vec{g}}}
          =\brVar{\frac{1}{\Nens}\sum_{j=1}^{\Nens}
   \obj(\design[j])\nabla_{\hyperparam}{\log{\CondProb{\design[j]}{\hyperparam}} }}
          =\frac{1}{\Nens} \brVar{          \obj(\design)\nabla_{\hyperparam}{\log{\CondProb{\design}{\hyperparam}} }}
            \,.
        \end{equation}

        To prove~\eqref{eqn:stoch_grad_norm_bound}, we
        rewrite~\eqref{eqn:total_var} as follows, 
        \begin{equation}
          \begin{aligned}
            \Expect{}{
              \widehat{\vec{g}} \tran \widehat{\vec{g}}
            }
            &= \brVar{\widehat{\vec{g}}} 
              + \Expect{}{\widehat{\vec{g}}} \tran \Expect{}{\widehat{\vec{g}}}
            = \frac{1}{\Nens} \brVar{          \obj(\design)\nabla_{\hyperparam}{\log{\CondProb{\design}{\hyperparam}} }}
              + \vec{g}\tran \vec{g}
              \\
            &\leq \frac{1}{\Nens} \brVar{
              C \nabla_{\hyperparam}{\log{\CondProb{\design}{\hyperparam}} }}
              + \vec{g}\tran \vec{g} \\
            &= \frac{C^2}{\Nens} \brVar{
              \nabla_{\hyperparam}{\log{\CondProb{\design}{\hyperparam}} }}
              + \vec{g}\tran \vec{g}
            \leq \frac{C^2\, \Nsens}{\Nens} 
              \left(\frac{1}{\min\limits_i \hyperparam}
                + \frac{1}{1-\max\limits_i \hyperparam}
                \right)
              + \vec{g}\tran \vec{g} \,,
          \end{aligned}
        \end{equation}
        where the last inequality follows
        by~\eqref{eqn:grad_log_Bernoulli_norm_bound},
        and, as before, $C = \max\limits_{\design\in \Omega_{\design}}{\{
          \left|\obj(\design)\right| \}}$.
        From~\Cref{lemma:exact_gradient_bound}, the gradient is bounded. By
        setting $K_1=\frac{C^2 \Nsens}{\Nens} \left(  \frac{1}{\min \hyperparam_i} +
        \frac{1}{1-\max \hyperparam_i} \right)$ and $K_2=1$, 
        one can guarantee that $
          \Expect{}{\widehat{\vec{g}}\tran \widehat{\vec{g}} }
            \leq K_1 + K_2 \vec{g}\tran \vec{g}  \,. \qquad$
      \end{proof}
      
    The significance of~\Cref{lemma:g_stats} 
    is that \eqref{eqn:stoch_grad_norm_bound} guarantees that Assumption (d)
    of~\cite[Assumptions 4.2]{bertsekas1996neuro} is satisfied.
    This, along with the fact that $\widehat{\vec{g}}$ is unbiased, and given
    the boundedeness of entries of the Hessian 
    (\Cref{lemma:exact_gradient_bound}), and by complying with the
    step-size requirement, guarantees convergence of~\Cref{alg:REINFORCE} 
    to a locally optimal policy $\hyperparam\opt$ .

    The sample-based approximation of the gradient described
    by~\eqref{eqn:kernel_stochastic_gradient_Bernoulli} exhibits high
    variance, however, and thus requires a prohibitively large number of samples in order to achieve
    acceptable convergence behavior to a locally optimal policy.
    Alternatively, one can use 
    importance sampling~\cite{arouna2004adaptative} 
    or antithetic variates~\cite{l1994efficiency} 
    or can add a
    \text{baseline} to the objective, in order to reduce the variability of the
    estimator. This issue is discussed 
    next (\Cref{subsec:variance_reduction_baseline}).

  \subsection{Variance reduction: introducing the baseline}
  \label{subsec:variance_reduction_baseline}
    %
    The formulation of the stochastic estimator $\widehat{\vec{g}}$ defined
    by~\eqref{eqn:kernel_stochastic_gradient}
    provides limited control over its variability.
    The implication of~\Cref{lemma:g_stats} is that, while the gradient estimator
    is unbiased, large samples are required to reduce its variability.
    Instead of increasing the sample size, one can reduce the estimator
    variability by introducing a \textit{baseline} $\baseline$ to the objective
    function.
    Specifically, it follows from~\eqref{eqn:stoch_grad_unbiased} that 
    $\Var{(\widehat{g}})=\Oh{\Nens^{-1}}$; thus,
    to reduce the variability of the estimator,
    one would need to increase the sample size.
    In general, MC estimators are known to suffer from high variance;  thus
    this estimator, while unbiased, will be impractical especially if 
    $\Nsens$ is large and the sample size $\Nens$ is small.
    The objective function $\stochobj$ can be replaced with the following
    baseline version,
    \begin{equation}\label{eqn:baseline_stoch_obj}
      \stochobj^{\rm b}(\hyperparam)
        := \Expect{\design\sim \CondProb{\design}{\hyperparam}}{\obj(\design) - \baseline} \,,
    \end{equation}
    where
    $\baseline$ is a baseline assumed to be independent from the
    parameter $\hyperparam$.
    Since $\baseline$  is independent from $\stochobj$ and by linearity of the
    expectation, it follows that
    $ \stochobj^{\rm b}(\hyperparam)
      = \Expect{\design\sim \CondProb{\design}{\hyperparam}}{\obj(\design) } - \baseline
    $, and thus
    $ \argmin_{\hyperparam}{\stochobj^{\rm b}} = \argmin_{\hyperparam}{\stochobj}$,
    and
    $ \nabla_{\hyperparam} \stochobj^{\rm b}
      = \nabla_{\hyperparam} \stochobj -  \nabla_{\hyperparam} \baseline
      = \nabla_{\hyperparam} \stochobj  \,.
    $
    By applying the kernel trick again, we can write the gradient
    of~\eqref{eqn:baseline_stoch_obj} as
    \begin{equation}\label{eqn:baseline_stoch_obj_grad}
      \nabla_{\hyperparam} \stochobj^{\rm b}(\hyperparam)
        = \Expect{\design\sim \CondProb{\design}{\hyperparam}}{
          \left(\obj(\design) - \baseline\right) \nabla_{\hyperparam}
            \log{ \CondProb{\design}{\hyperparam } }
        } \,,
    \end{equation}
    which can be approximated by the sample estimator
    \begin{equation}\label{eqn:kernel_stochastic_gradient_baseline}
      \begin{aligned}
        \widehat{\vec{g}}^{\rm b}
        &:= \frac{1}{\Nens} \sum_{j=1}^{\Nens}
          \left(\obj(\design[j]) -\baseline\right) \nabla_{\hyperparam} \,
          \log{ \CondProb{\design[j]}{\hyperparam } }
        = \widehat{\vec{g}} -
        \frac{\baseline}{\Nens} \sum_{j=1}^{\Nens}
          \nabla_{\hyperparam} \,
          \log{ \CondProb{\design[j]}{\hyperparam } }  \,.
      \end{aligned}
    \end{equation}

    Note that~\eqref{eqn:baseline_stoch_obj_grad} is also an unbiased
    estimator since 
    $\Expect{}{\widehat{\vec{g}}^{\rm b}}=\Expect{}{\widehat{\vec{g}}}=\vec{g}$.
    The variance of such an estimator, however, can be
    controlled by choosing an adequate baseline $\baseline$.
    In~\Cref{subsec:optimal_baseline}, we provide guidance for choosing
    the baseline. In what follows, however, we keep the baseline
    as a user-defined parameter, opening the door for other choices.
    For example, as shown in the numerical experiments,
    $\baseline=\frac{1}{2}\left( \obj(0) + \obj(1) \right) $ can be an 
    utilized as a constant baseline, which avoids the additional overhead of
    calculating~\eqref{eqn:optimal_baseline} at each iteration. This empirical
    choice, however, is suboptimal and is not guaranteed to provide acceptable
    results in all settings.

  \subsubsection{On the choice of the baseline}
  \label{subsec:optimal_baseline}
    We define the optimal baseline to be the one that minimizes the variability in the
    gradient estimator. 
    We provide the following results that will help us obtain an
    optimal baseline.
    \begin{lemma}
      Let
      $ \vec{d} = \frac{1}{\Nens}
        \sum_{j=1}^{\Nens} \nabla_{\hyperparam}
        \log{ \CondProb{\design[j]}{\hyperparam } }
      $.
      Then the following identities hold:
      \begin{subequations}
        \begin{align}
          \Expect{}{\vec{d}}
            &= 0 \,,  \\
          \brVar{\vec{d}}
            &= \frac{1}{\Nens}
            \sum_{i=1}^{\Nsens}\frac{1}  {\hyperparam_i-\hyperparam_i^2} \,.  
        \end{align}
      \end{subequations}
    \end{lemma}
    \begin{proof}
      The first identity follows from
      \begin{equation}
        \Expect{}{\vec{d}}
          = \Expect{}{
            \frac{1}{\Nens}
            \sum_{j=1}^{\Nens} \nabla_{\hyperparam}
            \log{ \CondProb{\design[j]}{\hyperparam } }
          }
          = \frac{1}{\Nens} \sum_{j=1}^{\Nens} \Expect{}{
              \nabla_{\hyperparam}
            \log{ \CondProb{\design[j]}{\hyperparam } }
          }
          = 0 \,,
      \end{equation}
      where the last equality follows from~\eqref{eqn:grad_logvar_dist}.
      Because $\Expect{}{\vec{d}}=0$, and by utilizing~\Cref{lemma:grad_logvar_dist},
      the second identity follows as
      \begin{equation}
        \begin{aligned}
          \brVar{\vec{d}}
          &= \Expect{}{\vec{d}\tran\vec{d}} -
            \Expect{}{\vec{d}}\tran \Expect{}{\vec{d}}
          = \Expect{}{\vec{d}\tran\vec{d}}
          = \frac{1}{\Nens^2} \sum_{j=1}^{\Nens}
            \brVar{\nabla_{\hyperparam}\log{ \CondProb{\design[j]}{\hyperparam }
            } }  \\
          &= \frac{1}{\Nens^2}
            \sum_{j=1}^{\Nens}
              \sum_{i=1}^{\Nsens}\frac{1}  {\hyperparam_i-\hyperparam_i^2}
          = \frac{1}{\Nens}
            \sum_{i=1}^{\Nsens}\frac{1}  {\hyperparam_i-\hyperparam_i^2}
          \,.
        \end{aligned}
      \end{equation}

      $\qquad$
    \end{proof}

    \begin{lemma}\label{lemma:baseline_estimate_variance}
      The ensemble estimator $\widehat{\vec{g}}^{\rm b} $ described
      by~\eqref{eqn:kernel_stochastic_gradient_baseline} is unbiased, with
      sampling total variance $\brVar{\widehat{\vec{g}}^{\rm b} }$, such that
      \begin{equation}
        \Expect{}{\widehat{\vec{g}}^{\rm b}} = \vec{g} =
        \nabla_{\hyperparam}\stochobj(\hyperparam) \,; \quad
        \brVar{\widehat{\vec{g}}^{\rm b} } = \brVar{\widehat{\vec{g}}}
            - 2 \baseline \Expect{}{\widehat{\vec{g}}\tran \vec{d} }
            + \frac{\baseline^2}{\Nens}
              \sum_{i=1}^{\Nsens}\frac{1}  {\hyperparam_i-\hyperparam_i^2} \,.
      \end{equation}
    \end{lemma}
    \begin{proof}
      The estimator is unbiased because 
      $ \nabla_{\hyperparam} \stochobj^{\rm b}(\hyperparam)
        = \nabla_{\hyperparam} \stochobj(\hyperparam) \,,
      $
      and
      \begin{equation}
        \begin{aligned}
          \Expect{}{\widehat{\vec{g}}^{\rm b} }
            &= \Expect{}{
              \frac{1}{\Nens} \sum_{j=1}^{\Nens}
              \obj(\design[j]-\baseline) \nabla_{\hyperparam}
              \log{ \CondProb{\design[j]}{\hyperparam } }
            }  \\
            &= \frac{1}{\Nens} \sum_{j=1}^{\Nens}
              \Expect{}{\obj(\design[j]-\baseline) \nabla_{\hyperparam}
              \log{ \CondProb{\design[j]}{\hyperparam } }
            }  \\
            &= \frac{1}{\Nens} \sum_{j=1}^{\Nens}
              \Expect{}{\obj(\design[j]) \nabla_{\hyperparam}
              \log{ \CondProb{\design[j]}{\hyperparam } }
            }
            - \frac{1}{\Nens} \sum_{j=1}^{\Nens}
              \Expect{}{\baseline \nabla_{\hyperparam}
              \log{ \CondProb{\design[j]}{\hyperparam } }
            }  \\
            & = \nabla_{\hyperparam} \stochobj(\hyperparam)
            - \frac{\baseline}{\Nens} \sum_{j=1}^{\Nens}
              \Expect{}{ \nabla_{\hyperparam}
              \log{ \CondProb{\design[j]}{\hyperparam } }
            } \\
            & = \nabla_{\hyperparam} \stochobj(\hyperparam)  \,,
        \end{aligned}
      \end{equation}
      where the last step follows by~\Cref{lemma:grad_logvar_dist}.
      The total variance of the estimator $\widehat{\vec{g}}^{\rm b}$ is
      given by 
      \begin{equation}\label{eqn:grad_estimate_var_1}
        \brVar{\widehat{\vec{g}}^{\rm b}}
        = \Expect{}{\left(\widehat{\vec{g}}^{\rm b}\right) \tran \widehat{\vec{g}}^{\rm b} }
        - \Expect{}{ \widehat{\vec{g}}^{\rm b}}\tran \Expect{}{ \widehat{\vec{g}}^{\rm b}}
        = \Expect{}{\left(\widehat{\vec{g}}^{\rm b}\right) \tran \widehat{\vec{g}}^{\rm b} }
        - \vec{g} \tran \vec{g}  \,,
      \end{equation}
      where
      $\vec{g}=\nabla_{\hyperparam}\stochobj^{\rm
      b}=\nabla_{\hyperparam}\stochobj$.
      The first term follows as
      \begin{equation}\label{eqn:grad_estimate_var_2}
        \begin{aligned}
          \Expect{}{\left(\widehat{\vec{g}}^{\rm b}\right) \tran \widehat{\vec{g}}^{\rm b} }
          &= \Expect{}{
            \left(\widehat{\vec{g}}-\baseline\vec{d}\right)\tran
            \left(\widehat{\vec{g}}-\baseline\vec{d} \right)
            }  \\
          &= \Expect{}{ \widehat{\vec{g}}\tran \widehat{\vec{g}} }
            - 2 \Expect{}{\widehat{\vec{g}}\tran \baseline\vec{d} }
            + \Expect{}{\baseline\vec{d}\tran \baseline\vec{d}}  \\
          &= \brVar{\widehat{\vec{g}}}
            + \Expect{}{\widehat{\vec{g}}}\tran \Expect{}{\widehat{\vec{g}}}
            - 2\baseline \Expect{}{\widehat{\vec{g}}\tran \vec{d} }
            + \baseline^2 \brVar{\vec{d}} + \baseline^2\Expect{}{\vec{d}}\tran
            \Expect{}{\vec{d}} \\
          &= \brVar{\widehat{\vec{g}}}
            + \vec{g}\tran \vec{g}
            - 2 \baseline\Expect{}{\widehat{\vec{g}}\tran \vec{d} }
            + \baseline^2 \brVar{\vec{d}}   \\
          &= \brVar{\widehat{\vec{g}}}
            + \vec{g}\tran \vec{g}
            - 2 \baseline\Expect{}{\widehat{\vec{g}}\tran \vec{d} }
            + \frac{\baseline^2}{\Nens}
              \sum_{i=1}^{\Nsens}\frac{1}  {\hyperparam_i-\hyperparam_i^2}  \,.
        \end{aligned}
      \end{equation}

      From~(\ref{eqn:grad_estimate_var_1},~\ref{eqn:grad_estimate_var_2}), 
      the total variance follows as
      \begin{equation}\label{eqn:baseline_estimator_variance}
        \brVar{\widehat{\vec{g}}^{\rm b}}
        = \brVar{\widehat{\vec{g}}}
            - 2 \baseline \Expect{}{\widehat{\vec{g}}\tran \vec{d} }
            + \frac{\baseline^2}{\Nens}
              \sum_{i=1}^{\Nsens}\frac{1}  {\hyperparam_i-\hyperparam_i^2}
         \,.
      \end{equation}

    \end{proof}

    The variance of $\widehat{\vec{g}}^{\rm b}$ is described
    by~\Cref{lemma:baseline_estimate_variance}.
    We can view~\eqref{eqn:baseline_estimator_variance} as a quadratic
    expression in $\baseline$ and minimize it over $\baseline$. 
    Because $\sum_{i=1}^{\Nsens}\frac{1}  {\hyperparam_i-\hyperparam_i^2} > 0$,
    the quadratic is convex, and the min in $\baseline$ is obtained
    by equating the derivative of the estimator
    variance~\eqref{eqn:baseline_estimator_variance} to zero, which yields
    the following optimal baseline:
    \begin{equation}\label{eqn:optimal_baseline}
      \baseline^{\rm opt}
        = \frac{\Nens }{\sum_{i=1}^{\Nsens}\frac{1}{\hyperparam_i-\hyperparam_i^2}}
          \, \Expect{}{\widehat{\vec{g}}\tran \vec{d} } \,.
    \end{equation}

    The expectation in~\eqref{eqn:optimal_baseline}, however, depends on the
    value of the function  $\stochobj$ and can be estimated by an ensemble of
    realizations $\widehat{\vec{g}}[j],\,\vec{d}[j],\,j=1,2,\ldots,m$,
    \begin{equation} 
      \Expect{}{ \widehat{\vec{g}}\tran \vec{d} }
      \approx
      \frac{1}{b_m} \sum_{e=1}^{b_m} \widehat{\vec{g}}[e]\tran \vec{d}[e] \,,
    \end{equation}
    where $\vec{g}[e]$ and $\vec{d}[e]$ are realizations of $\vec{g}$ and
    $\vec{d}$, respectively.
    Thus, we propose 
    to estimate the optimal baseline $\baseline\opt$ as follows. 
    Given  a realization $\hyperparam$ of the hyperparameter,
    a set of $b_m$ batches each of size $\Nens$ are
    sampled from $\CondProb{\design}{\hyperparam}$, resulting in the
    multivariate Bernoulli samples
    $\{\design[e,j]; e=1,2,\ldots,b_m;\, j=1,2,\ldots,\Nens \}$.
    The following function is then used to
    estimate $\baseline\opt$:
    \begin{equation}\label{eqn:optimal_baseline_estimate}
      \baseline\opt
      \approx
      \widehat{\baseline}\opt
        :=
        \frac{\sum\limits_{e=1}^{b_m} 
          \left(
            \sum\limits_{j=1}^{\Nens} \obj(\design[e,j]) \nabla_{\hyperparam}
            \log{ \CondProb{\design[e, j]}{\hyperparam } } 
          \right)
          \tran\! 
          \left(
            \sum\limits_{j=1}^{\Nens} \nabla_{\hyperparam}
            \log{ \CondProb{\design[e, j]}{\hyperparam } } 
          \right)
          }
        { \Nens\, b_m\, \sum\limits_{i=1}^{\Nsens}\frac{1}{\hyperparam_i-\hyperparam_i^2}}
      \,.
    \end{equation}
    %
    

    \subsubsection{Complete algorithm statement}
    \label{subsubsec:full_algorithm}
    We conclude this section with an algorithmic description of the stochastic
    steepest-descent algorithm with the optimal baseline suggested here.
    \Cref{alg:REINFORCE_baseline} is a modification of~\Cref{alg:REINFORCE}, where
    we  added only the baseline~\eqref{eqn:optimal_baseline_estimate}.

    \begin{algorithm}
      \caption{Stochastic optimization for binary OED with the optimal baseline.}
      \label{alg:REINFORCE_baseline}
      \begin{algorithmic}[1] 
      
        \Require{Initial distribution parameter $\hyperparam^{(0)}$,
                  step size schedule $\eta^{(n)}$, 
                  sample sizes $\Nens,\, m$,
                  baseline batch size $b_m$}
        \Ensure{$\design\opt$}

        \State{initialize $n = 0$}

        \While{Not Converged}
          \State {
            Update $n\leftarrow n+1$ 
          }
          
          \State{
            Sample $\{\design[j]; j=1,2,\ldots,\Nens \}\sim
              \CondProb{\design}{\hyperparam^{(n)}}$ 
          }
        
          \State{
            Calculate  $\baseline$ = \Call{OptimalBaseline}{$\hyperparam^{(n)}$, 
              $\Nens$, $b_m$} 
          }
          
          \State\label{algstep:REINFORCE_baseline:grad}{
            Calculate $ \vec{g}^{(n)}=\frac{1}{\Nens} \sum_{j=1}^{\Nens}
              \left(\obj(\design[j]-\baseline)\right) \sum_{i=1}^{\Nsens} 
              \left( 
                \frac{\design_i[j]}{\hyperparam_i} 
                  + \frac{\design[j]_i-1}{1-\hyperparam_i}
              \right) \,\vec{e}_i $
          }
        
          \State \label{algstep:REINFORCE_baseline:proj}{
            Update $\hyperparam^{(n+1)} 
              = \Proj{}{\hyperparam^{(n)} - \eta^{(n)} g^{(n)} } $
          }
          
        \EndWhile
        
        \State{
          Set $\hyperparam\opt = \hyperparam^{(n)}$
        }

        \State\label{algstep:REINFORCE_baseline:sampling} {
          Sample $\{\design[j];j=1,2,\ldots,m \} \sim 
            \CondProb{\design}{\hyperparam\opt}$,
          and calculate $\obj(\design[j])$
        }
   
        \Return{
            $\design\opt $: the design $\design$ with smallest 
              value of $\obj$ in the sample.
        }
        
        \vspace{5pt}

        \Function{OptimalBaseline}{$\theta$, $\Nens$, $b_m$} 
          
          \State{Initialize $\baseline \gets 0$}

          \For {$e$ $\gets 1$ to $b_m$}
            
            \For {$j$ $\gets 1$ to $\Nens$}
              \State{
                Sample $\design[j] \sim \CondProb{\design}{\hyperparam}$ 
              }
              \State\label{algstep:REINFORCE_baseline:logprob}{
                Calculate $\vec{r}[j] = \sum_{i=1}^{\Nsens} 
                  \left( 
                    \frac{\design_i[j]}{\hyperparam_i} 
                      + \frac{\design[j]_i-1}{1-\hyperparam_i}
                  \right) \,\vec{e}_i $
              }
            \EndFor

            \State{
              Calculate $\vec{d}[e] = \frac{1}{\Nens} \sum_{j=1}^{\Nens}
              \vec{r}[j] $
            }
            
            \State{
              Calculate $ \vec{g}[e] = \frac{1}{\Nens} \sum_{j=1}^{\Nens}
                \obj(\design[j]) \, \vec{r}[j] $
            }
            \State{
              Update $\baseline \gets \baseline + \left(\vec{g}[e] \right)\tran \vec{d}[e]$
            }
          \EndFor

          \State{
            Update 
              $\baseline \gets \baseline 
                \times 
                \frac{ \Nens }{b_m \, 
                \sum_{i=1}^{\Nsens} \frac{1}{\hyperparam_i- \hyperparam_i^2}}
              $
          }
          
          \State \Return $\baseline$
        \EndFunction

      \end{algorithmic}
    \end{algorithm}

    Note that in both~\Cref{alg:REINFORCE} and~\Cref{alg:REINFORCE_baseline},
    the value of $\obj$ is evaluated repeatedly at instances of the binary design
    $\design$. With the algorithm proceeding, it becomes more likely to revisit
    previously sampled designs. One should keep track of the sampled designs
    and the corresponding value of $\obj$, for example, by utilizing the indexing
    scheme~\eqref{eqn:binary_index}, to prevent redundant computations.
  We remark  that as noted in~\autoref{alg:REINFORCE}, 
    if $\hyperparam_i\in\{0, 1\}$
    in~\autoref{algstep:REINFORCE_baseline:grad} 
    or~\autoref{algstep:REINFORCE_baseline:logprob} 
    of~\autoref{alg:REINFORCE_baseline}, 
    then $\design_i=\hyperparam_i$. Thus the
    corresponding term in the summation vanishes.

  \subsection{Computational considerations}
  \label{subsec:computational_cost}
    Here, we discuss the computational cost of the proposed algorithms and of
    standard OED approaches in terms
    of the number of forward $\F$ and adjoint $\F\adj$ model evaluations.
    We assume $\obj$ is set to the A-optimality criterion, that is, the trace
    of the posterior covariance of the inversion parameter. 
    This discussion extends easily to other OED optimality criteria.

    Standard OED approaches require solving the relaxed OED 
    problem~\eqref{eqn:relaxed_optimization}, which requires
    evaluating the gradient of the objective $\obj$, namely, the optimality
    criterion, with respect to the relaxed design, in addition to evaluating the
    objective itself $\obj$ for line-search optimization.
    Formulating the gradient requires one Hessian solve and a forward 
    integration of the model $\F$ for each entry of the gradient. 
    The Hessian, being the inverse of the
    posterior covariance, is a function of the relaxed design; see, for example,~\cite{attia2018goal}
    for details.
    Hessian solves can be done by using a preconditioned conjugate gradient
    (CG) method. 
    Each application of the Hessian requires a forward and an adjoint model
    evaluation.  If the prior covariance is employed as a preconditioner,
    and assuming $r \ll \Nstate$ is the numerical rank of the prior preconditioned 
    data misfit Hessian
    (see~\cite{bui2013computational,IsaacPetraStadlerEtAl15}), then
    the cost of one Hessian solve is $\mathcal{O}(r)$ CG iterations, 
    that is, $\mathcal{O}(2 r)$ evaluations of the forward model $\F$.
    To summarize, the cost of evaluating the optimality criterion $\obj$ for a
    given design $\design$ is  $\Oh{2\, r\, \Nstate}$ forward model solves.
    Moreover, the cost of evaluating the gradient of $\obj$ with respect to the
    design is $\Oh{2\, r\,\Nsens \Nstate}$ model solves.

    In contrast, with the proposed algorithms, we do not need to evaluate the gradient of
    $\obj$ with respect to the design.
    At each iteration of~\Cref{alg:REINFORCE}, the function $\obj$ is evaluated
    for each sampled design to evaluate the stochastic gradient. 
    Assuming the size of the sample used to formulate the stochastic
    gradient $\widehat{\vec{g}}$ is $\Nens$, then the cost of each iteration is
    $\Oh{2\, r\, \Nens \Nstate }$.
    The cost of evaluating the gradient of the multivariate Bernoulli
    distribution~\eqref{eqn:Bernoulli_PMF_gradient} is negligible compared with
    solving the forward model $\F$.
    
    Note that unlike the case with the relaxed OED formulation,
    in the proposed framework the design space is binary by definition; 
    and as we will show later, as the optimization algorithm proceeds, 
    it reuses previously sampled designs.
    Moreover, the value of $\obj$ can be evaluated independently, and thus the
    stochastic gradient approximation is embarrassingly parallel. 
    
\section{Numerical Experiments}
  \label{sec:numerical_experiments}
  We start this section with a small illustrative model
  to clarify the approach proposed and
  provide additional insight.
  Next, we present numerical experiments using an advection-diffusion model.
  
  \subsection{Results for a two-dimensional problem}
  \label{subsec:toy_models_results}
  Here we discuss an idealized  
  problem following the
  definition of the linear forward and inverse problem described
  in~\Cref{sec:background}.
  Python code for this set of experiments is available 
  from~\cite{attia2020DOERL}.
  We define the forward operator $\F$ as a short wide matrix 
  that projects model space into observation space. 
  Moreover, we specify prior and observation
  covariance matrices and formulate the posterior covariance matrix
  $\Cparampostmat$ and the objective $\obj$ accordingly.
  The forward operator and prior and observation noise covariances are 
  \begin{equation} \label{eqn:oneD_Toy}
    \F :=  \begin{bmatrix}0.5 & 0.5 & 0 & 0  \\ 0 & 0 & 0.5 & 0.5 \end{bmatrix};
    \quad
    \Cparamprior := \diag{4, 1, 0.25, 1}
    \,; 
    \quad
    \Cobsnoise:= \diag{0.25, 1}
    \,,
  \end{equation}
  which result in the following form of the objective 
  $\obj= \Trace{\Cparampost(\design)}$:
  \begin{equation}\label{eqn:oneD_Toy_objective}
    \obj(\design) 
      = \Trace{
          \begin{bmatrix}
            \design_1+0.25 & \design_1 & 0 & 0  \\ 
            \design_1 & \design_1+1 & 0 & 0  \\ 
            0 & 0 & 0.25\design_2+4 & 0.25\design_2  \\ 
            0 & 0 & 0.25\design_2 & 0.25\design_2+1 
          \end{bmatrix}  
        }
      =2\, \design_1 + 0.5\, \design_2 + 6.25
      \,.
  \end{equation}

    \Cref{fig:REINFORCE_2dToy_optimiazation} (left) shows the surface of the two
    objective functions $\obj$ and $\stochobj$, respectively. 
    The objective
    functions are evaluated on a regular grid of $15$ values equally spaced 
    in each direction.
    In the same plot we also display the
    progress of~\Cref{alg:REINFORCE_baseline} with various choices
    of the baseline $\baseline$. Specifically, we first set $\baseline$ to $0$
    (this corresponds to applying~\Cref{alg:REINFORCE}). 
    Next, we set the baseline to an empirically chosen value
    \begin{equation}\label{eqn:empirical_baseline}
      \baseline = \frac{\obj(\vec{0}) + \obj(\vec{1})}{2}\,,
    \end{equation}
    where $\vec{0}$ corresponds to turning all sensors off and $\vec{1}$
    corresponds to activating all sensors. 
    This gives an empirical estimate of
    the average value of the deterministic objective function $\obj$ and thus,
    in principle, may scale the gradient properly.
    We utilize the optimal baseline $\baseline\opt$ described
    in~\Cref{subsec:optimal_baseline}.
    In all cases, we set the learning rate to $0.25$.
    
    The value of the objective function $\stochobj$ evaluated at each iteration
    of the optimizer is shown
    in~\Cref{fig:REINFORCE_2dToy_optimiazation} (right).
    \begin{figure}[ht] \centering
      \includegraphics[align=c,width=0.60\linewidth]{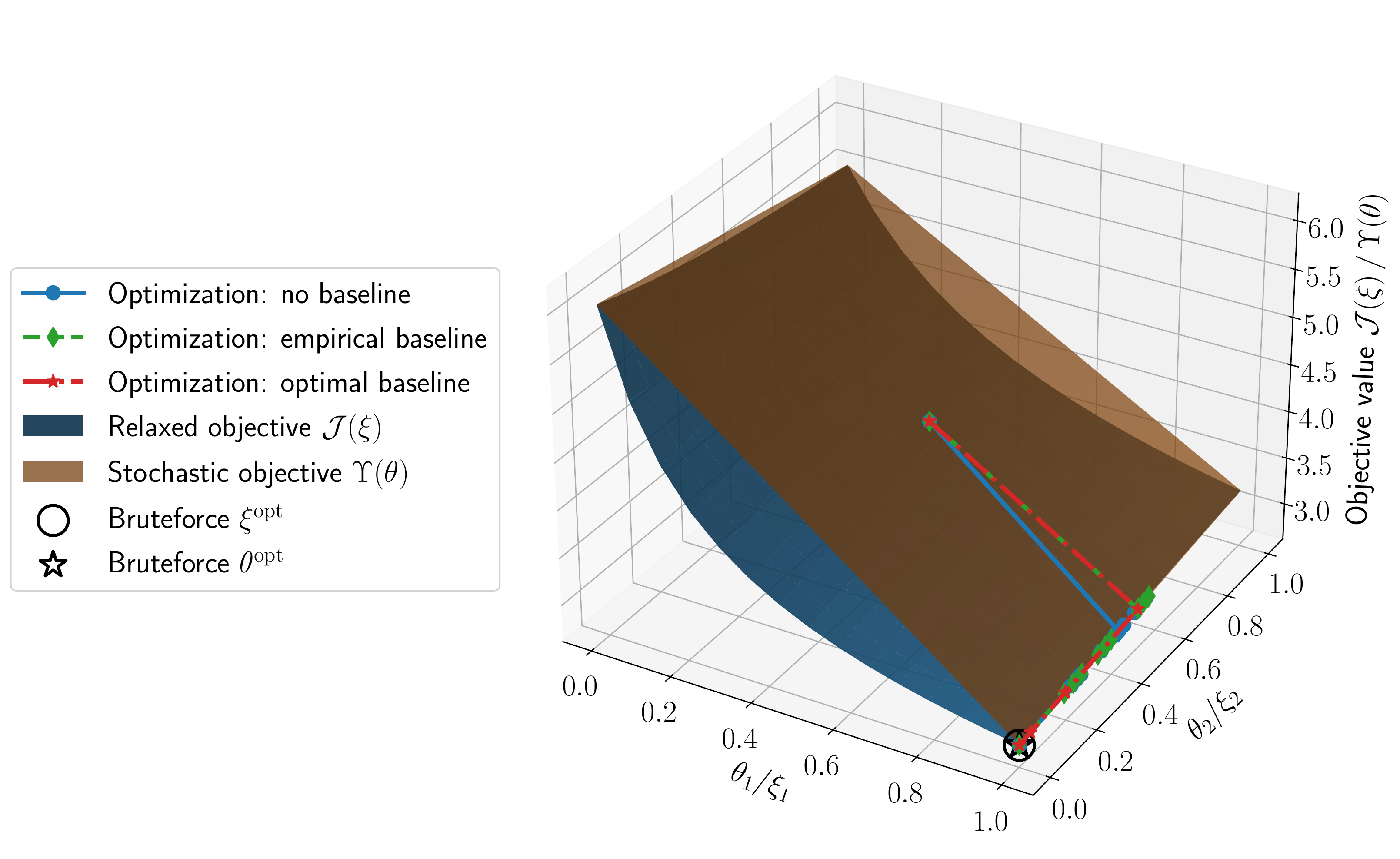}
      \hfill
      \includegraphics[align=c,width=0.35\linewidth]{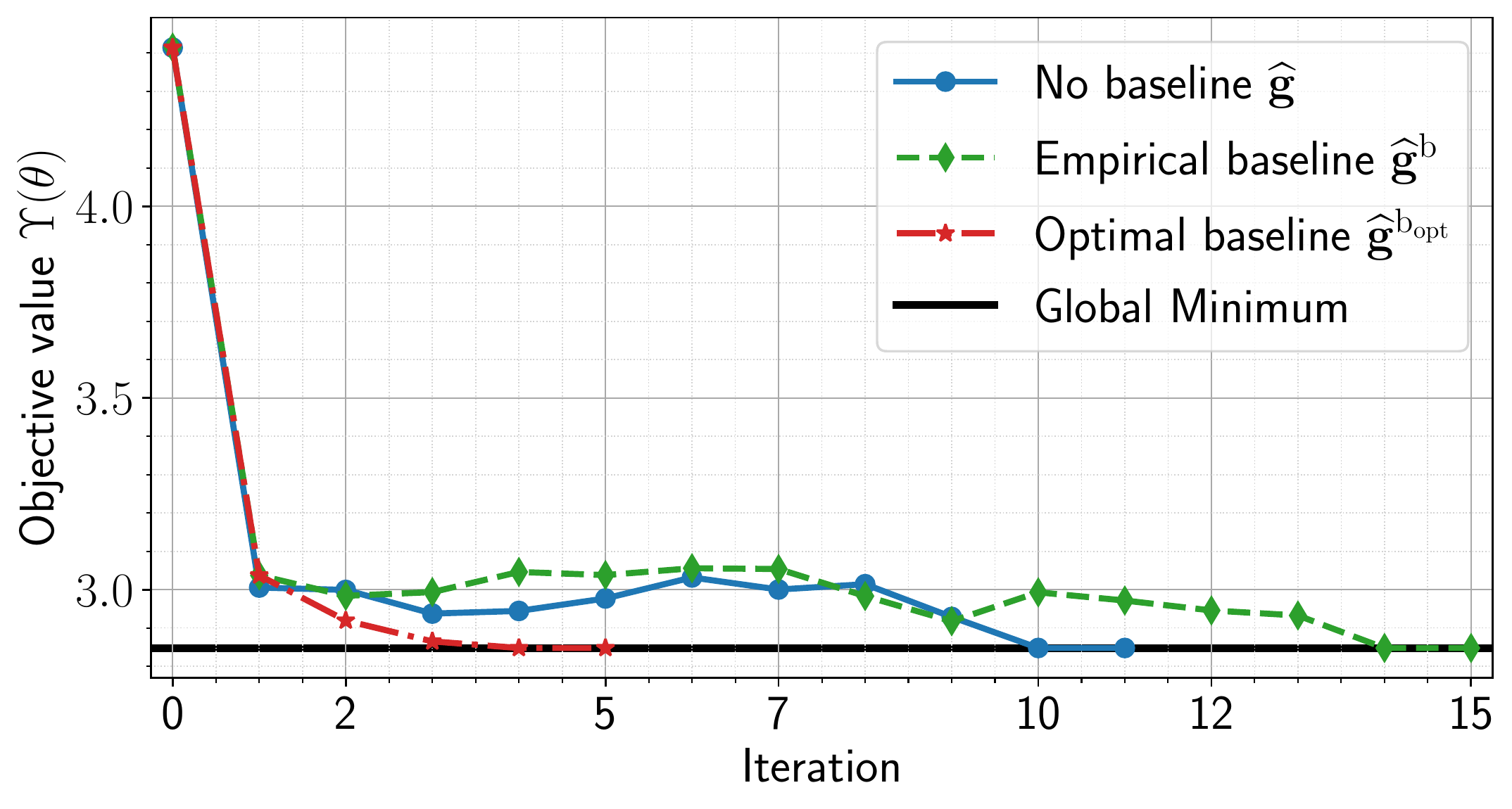} 
      \caption{
        Left: surface plot of the objective function $\obj$ of the relaxed OED
        problem and the objective function $\stochobj$ of the corresponding
        stochastic OED problem.  In each direction
        $15$ equally spaced points are taken  to create the
        surface plots.
        Iterations of the optimization algorithm are shown on the surface plot
        for various choices of the baseline $\baseline$.
        Right: value of the objective function $\stochobj$ evaluated at each
        iteration of the algorithm until convergence.
        Brute-force results are obtained by searching over all $4$ possible
        values of the binary design $\design\in\Omega_\design$.
    The initial parameter $\hyperparam^{(0)}$ of the optimizer is set to
        $(0.5,0.5)\tran$, and the algorithm terminates when the magnitude of the
        projected gradient (pgtol) is lower than $10^{-8}$.
      }
      \label{fig:REINFORCE_2dToy_optimiazation}
    \end{figure}
    We note that, as explained in~\Cref{subsec:benefits_of_stochastic_OED}, the
    values of $\obj(\design)$ and $\hyperparam$ coincide at the extremal points of the
    domain $[0,1]^{\Nsens}$.  
    Moreover, unlike the surface of the stochastic objective $\stochobj$, 
    the surface of the original objective function $\obj$, evaluated at 
    the relaxed design, flattens out for values of $\hyperparam_1$ greater than
    $0.5$. This behavior  makes applying a sparsification procedure challenging when associated 
    with traditional OED approaches.

    While the performance varies slightly based on the choice of the baseline $\baseline$,
    we note that, in general, the optimizer initially moves quickly toward a 
    lower-dimensional space corresponding to lower values of the objective function 
    $\stochobj$ and then moves slowly toward a local optimum. 
    We also note that since both
    candidate parameter values $(0,1)\tran$ and $(1, 1)\tran$ 
    have similar objective values, the optimizer
    moves slowly between them, since the value of the gradient in this direction
    is close to zero. If the optimizer is terminated before convergence (say
    after the first iteration where $\hyperparam_1$ is set to $1$ here), the
    returned value of $\hyperparam_2$ is in the interval $(0, 1)$, which allows
    sampling estimates of $\design_1\opt$ from $\{0, 1\}$, and the decision can
    be made based on the value of $\obj$ or other decisions, such as 
    budget constraints.
    Alternatively, one could modify~\eqref{eqn:oneD_Toy_objective} by adding a
    regularization term to enforce desired constraints. This will be explained
    further in~\Cref{subsec:advection_diffusion_results}.
    
    \Cref{fig:REINFORCE_2dToy_gradients} shows the gradient of the stochastic
    objective function $\stochobj$ evaluated (or approximated) at various 
    choices of $\hyperparam$.
    The top-left panel show results using the exact formulation of the 
    gradient~\eqref{eqn:exact_gradient}. The top-right panel shows the gradient
    evaluated  using~\eqref{eqn:kernel_stochastic_gradient}.
    The lower two panels show evaluations of the gradient
    using~\eqref{eqn:kernel_stochastic_gradient_baseline} with an empirical choice
    baseline~\eqref{eqn:empirical_baseline} and the estimate of the
    optimal baseline~\eqref{eqn:optimal_baseline_estimate}, respectively.
    These results show that the stochastic approximations of the gradient
    utilized in~\Cref{alg:REINFORCE} and~\Cref{alg:REINFORCE_baseline},
    better approximate the true gradient, given any realization of the
    parameter $\hyperparam$.
    However, the estimates with the baseline (both empirical choice and optimal estimate)
    exhibit much lower variability than the gradient evaluated without the baseline.
    This is also reflected by the performance of the optimization results
    in~\Cref{fig:REINFORCE_2dToy_optimiazation}.
    \begin{figure}[ht] \centering
      \includegraphics[width=0.60\linewidth]{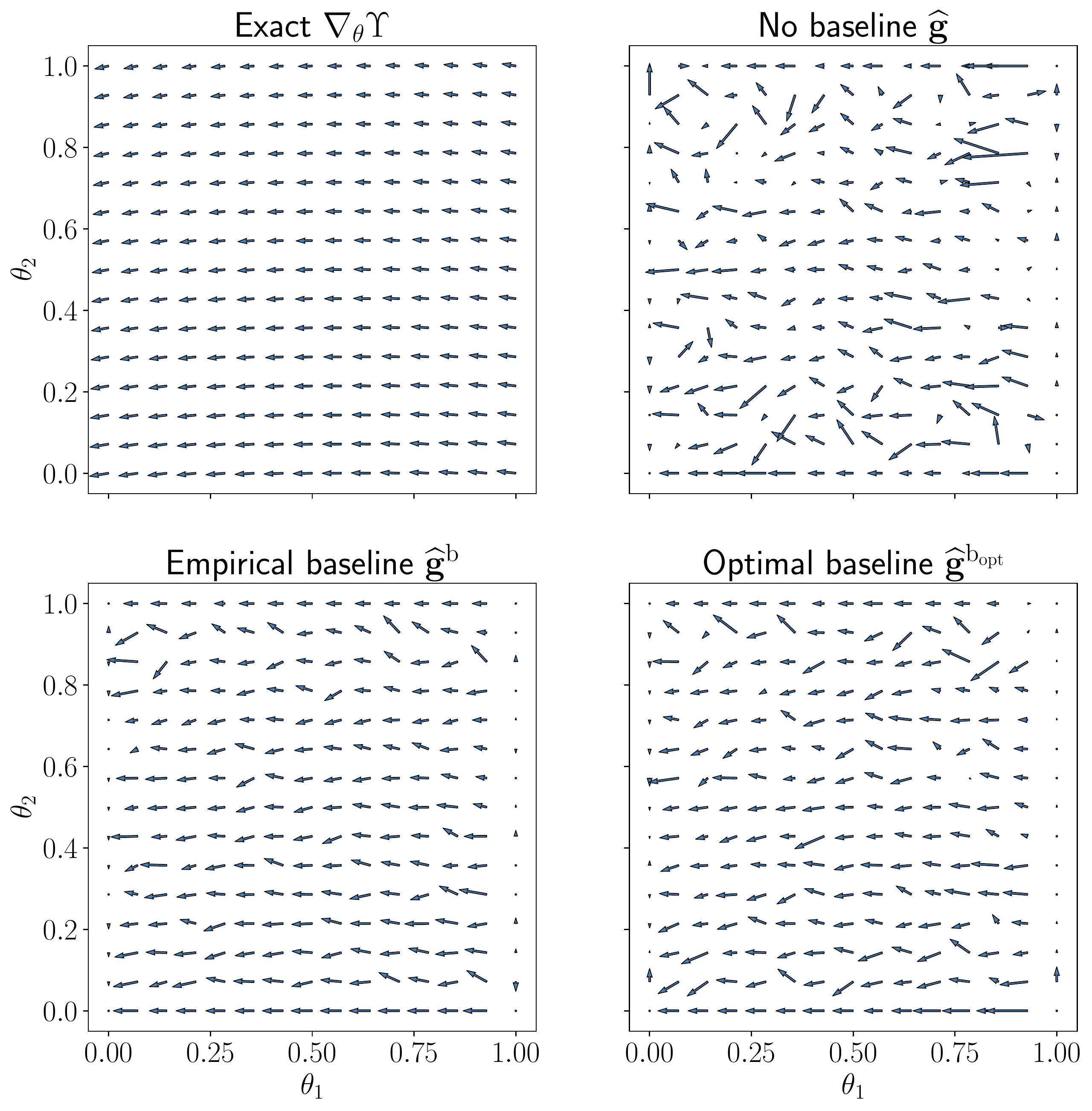}
      \caption{ 
        Evaluation of the gradient of the objective
        $\stochobj(\hyperparam)
          :=\Expect{\xi\sim\CondProb{\xi}{\hyperparam} }{\mathcal{J(\xi)}}
        $, where $\obj$ is defined by~\eqref{eqn:oneD_Toy}. 
        Gradients are evaluated (or approximated) at $15$ equally spaced points
        in each direction.
        Top-left: gradient is evaluated exactly
        using~\eqref{eqn:exact_gradient}.
        Top-right: gradient is approximated using~\eqref{eqn:kernel_stochastic_gradient}.
        Bottom-left: gradient is approximated
        using~\eqref{eqn:kernel_stochastic_gradient_baseline}, with $\baseline$ set
        to~\eqref{eqn:empirical_baseline}.
        Bottom-right: gradient is approximated
        using~\eqref{eqn:kernel_stochastic_gradient_baseline}, with $\baseline$
        evaluated using~\eqref{eqn:optimal_baseline_estimate}.
      }
      \label{fig:REINFORCE_2dToy_gradients}
    \end{figure}

  \subsection{Experimental setup for an advection-diffusion problem}
  \label{subsec:advection_diffusion_results}
    In this subsection we demonstrate the effectiveness of our proposed approach
    using an advection-diffusion model simulation that has been used extensively
    in the literature; see, for
    example,~\cite{PetraStadler11,attia2018goal,attia2020optimal} and references
    therein.

    The advection-diffusion model simulates the spatiotemporal evolution of a
    contaminant field $u=\xcont(\vec{x}, t)$ in a closed domain $\domain$.
    Given a set of candidate locations to deploy sensors to measure the
    contaminant concentration, we seek the optimal subset of sensors 
    that once deployed would enable inferring the initial distribution of the
    contaminant with minimum uncertainty.
    To this end, we seek the optimal subset of candidate sensors that minimized
    the A-optimality criterion, that is, the trace of the posterior covariance
    matrix.

    We carry out numerical experiments in two settings with varying
    complexities. Specifically, we start with a setup where only $14$ candidate
    sensor locations are considered inside the domain $\domain$.
    The number of possible combinations of active sensors in this case is
    $2^{14}=16,384$. 
    Despite being large, this allows us to carry out a brute-force search. 
    The purpose of the brute-force search here is to study the
    behavior of the proposed methodology and its capability in exploring the
    design space and utilizing any constraints properly, while seeking the optimal design. 
    In particular, we can compare the quality of our solution with the global
    minimum in this case.


    \paragraph{Model setup: advection-diffusion}
    In both sets of experiments, we use the same model setup.
    Specifically, the contaminant field $u = \xcont(\mathbf{x}, t)$ 
    is governed by the advection-diffusion equation 
      \begin{equation}\label{eqn:advection_diffusion}
        \begin{aligned}
          \xcont_t - \kappa \Delta \xcont + \vec{v} \cdot \nabla \xcont &= 0     
            \quad \text{in } \domain \times [0,T],   \\
          \xcont(x,\,0) &= \theta \quad \text{in } \domain,  \\
          \kappa \nabla \xcont \cdot \vec{n} &= 0  
            \quad \text{on } \partial \domain \times [0,T],
        \end{aligned}
    \end{equation}
    where $\kappa>0$ is the diffusivity, $T$ is the simulation final time, and
    $\vec{v}$ is the velocity field. 
    The spatial domain here is $\domain=[0, 1]^2$, with two rectangular regions
    inside the domain simulating two buildings where the flow is not
    allowed to enter.
    Here, $\partial \domain$ refers to the boundary of the domain, which includes 
    both the external boundary and the walls of the two buildings.
    The velocity field $\vec{v}$ is assumed to be known and is obtained by solving 
    a steady Navier--Stokes equation, with the side walls driving the flow;
    see~\cite{PetraStadler11} for further details.
    
    To create a synthetic simulation, we use the initial distribution of
    contaminant shown in~\Cref{fig:AD_Setup} (left) as the ground truth.
    \begin{figure}[ht]
      \centering
      \includegraphics[width=0.35\linewidth]{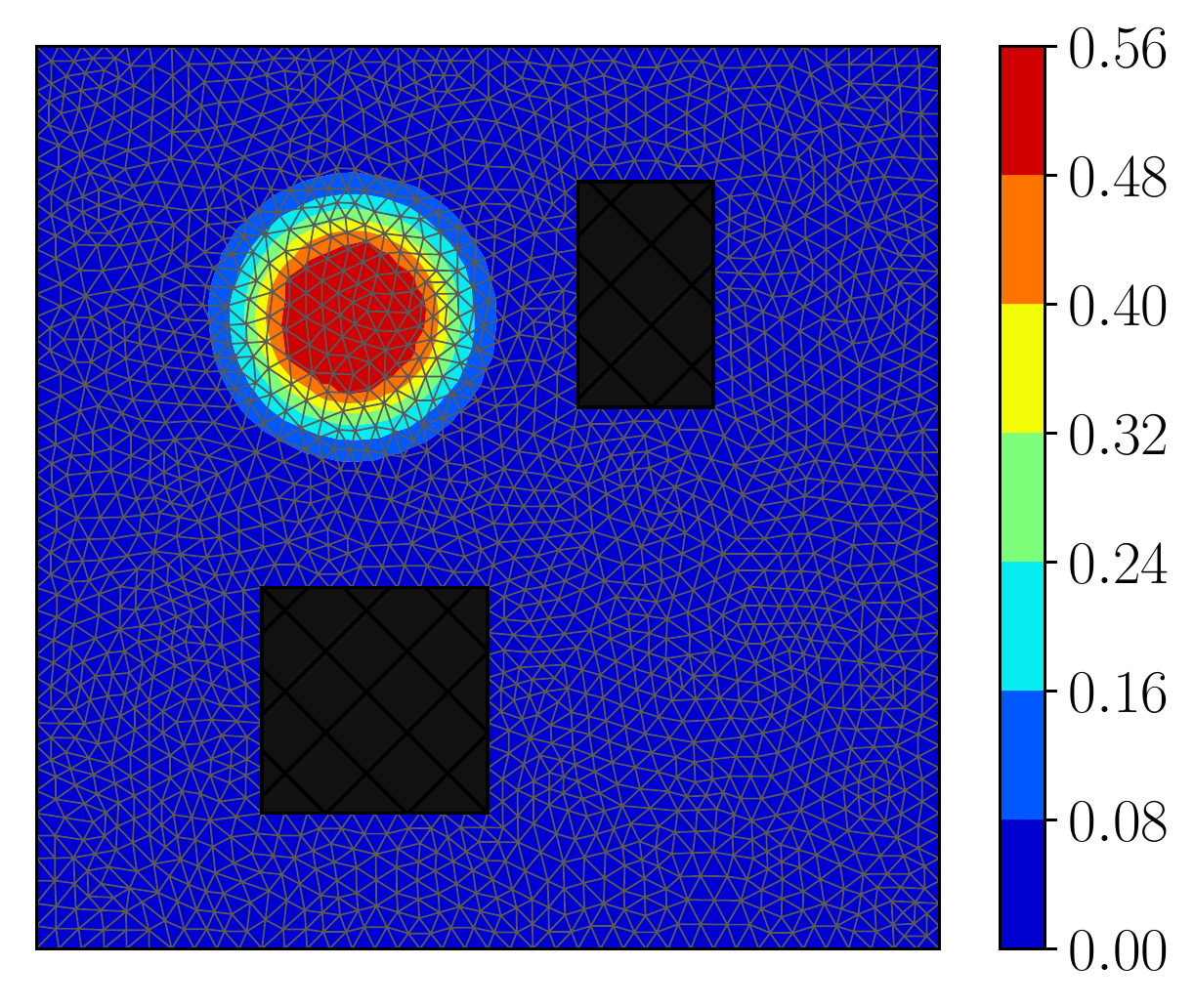}
      \quad
      \includegraphics[width=0.30\linewidth]{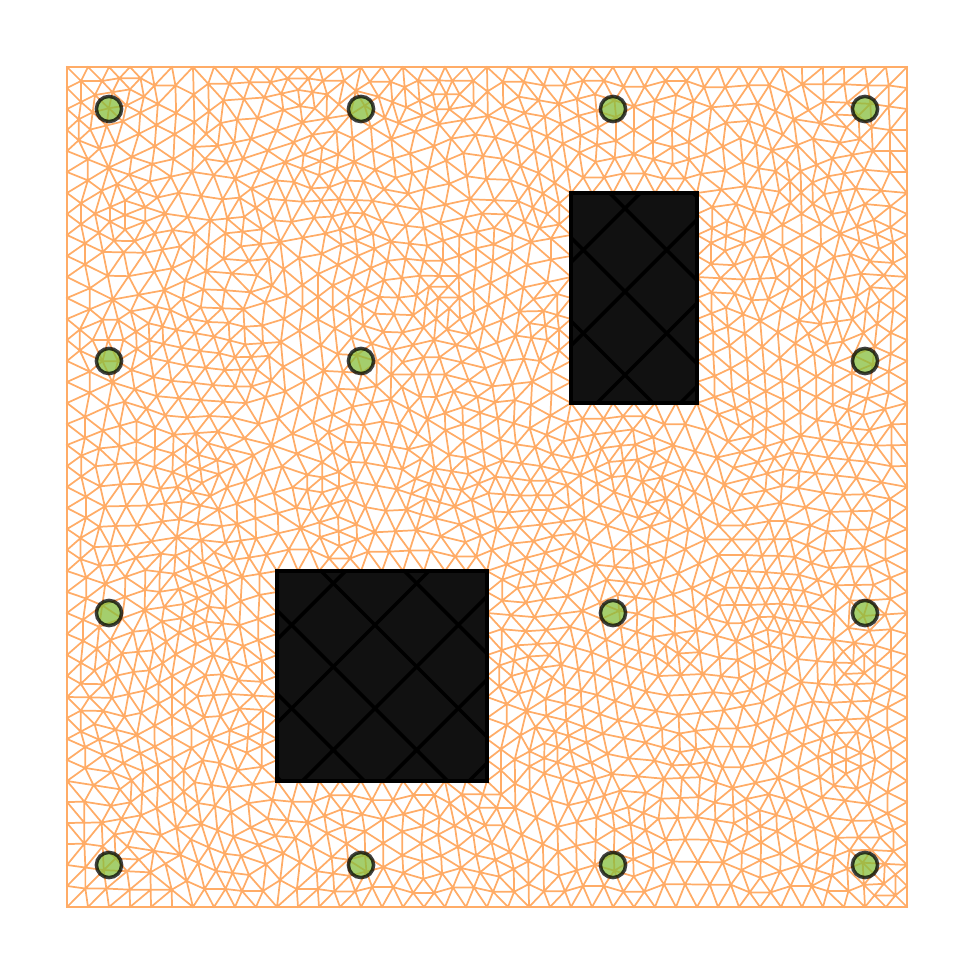}
      \caption{Advection-diffusion model domain, discretization, candidate
        sensor locations, and the true model parameter, in other words,
        the true initial condition. 
        Left: The physical domain $\domain$ including outer boundary 
        and the two buildings, the model grid discretization, and the
        true model parameter. 
        Right: Candidate observational sensor locations.
        }
      \label{fig:AD_Setup} 
    \end{figure}
    %

    \paragraph{Observational setup}
      We consider a set of uniformly distributed candidate sensor locations
      (spatial observational gridpoints).
      Specifically, we consider $\Nsens=14$ candidate sensor locations as described
      by~\Cref{fig:AD_Setup} (right), and we assume that the sensor locations do not
      change over time.
      An observation vector $\obs$ represents the concentration of the
      contaminant at the sensor locations, at a set of predefined time instances 
      $\{ t_1,\,  t_2,\, \ldots,\, t_{\nobstimes} \} \subset [0, T]$.
      The observation times are set to $t_1\!+\! s \Delta t$, with initial observation
      time $t_1\!=\!1$;  $\Delta t\!=\!0.2$ is the model simulation timestep;
      and $s=0, 1, \ldots, 20$. The result is $\nobstimes=16$ observation time
      instances, over the simulation window $[0,\, T\!=\!4]$.
      The dimension of the observation space is thus
      $\Nobs = \Nsens \times \nobstimes$. 
      
      The observation error distribution is $\GM{\vec{0}}{\Cobsnoise}$, with
      $\Cobsnoise\in \Rnum^{\Nobs\times\Nobs}$ describing spatiotemporal
      correlations of observational errors.
      We assume that observation errors are time-invariant and are calculated as
      follows.
      For simplicity, we assume that observation errors are uncorrelated, 
      with fixed standard deviation; that is, the observation error covariance 
      matrix takes the form $\Cobsnoise=\sigma_{\rm obs}^2 \mat{I}$, where
      $\mat{I}\in\Rnum^{\Nobs\times\Nobs}$ is the identity matrix.
      Here, we set the observation error variances to 
      $\sigma_{\rm obs}\!=\! 2.482\!\times\! 10^{-2}$.
      This specific value is obtained by considering a noise level 
      of $5\%$ of the maximum value of the contaminant concentration captured 
      at all observation points, by running a simulation over $[0, T]$, using
      the ground truth of the model parameter; see~\Cref{fig:AD_Setup} (left).

    \paragraph{Forward operator, adjoint operator, and the prior}
      The forward operator $\F$ maps the model parameter $\iparam$, here the
      model initial condition, to the observation space.
      Specifically, $\F$ represents a forward simulation over the interval $[0,
      T]$ followed by applying an observation operator (here, a restriction
      operator), to extract concentrations at sensor locations at observation
      time instances.
      The forward operator here is linear, and the adjoint is defined by using the 
      Euclidean inner product weighted by the finite-element mass matrix
      $\mat{M}$ as $\F\adj := \mat{M}\inv\mat{F}\tran$; 
      see~\cite{Bui-ThanhGhattasMartinEtAl13} for further details.

      The prior distribution of the parameter $\iparam$ is
      modeled by a Gaussian distribution $\GM{\iparb}{\Cparampriormat}$, 
      where $\Cparampriormat$ is a discretization of $\mathcal{A}^{-2}$,
      with $\mathcal{A}$ being a Laplacian (following~\cite{Bui-ThanhGhattasMartinEtAl13}).

    \subsection{The OED optimization problem}
      Now we define the design space and formulate the OED optimization problem.
      To find the best subset of candidate sensor locations, we assign a binary
      design variable $\design_i$ to each candidate sensor location $x_i$, where
      $i=1,2,\ldots,\Nsens$, and hence $\design\in\{0,1\}^{\Nsens}$.
      We aim to find a binary A-optimal design, that is, the minimizer 
      of the trace of the posterior covariance matrix. Moreover, 
      to promote sparsity of the design, we employ an $\ell_0$  penalty
      term $\Phi$.
      We thus define the objective function $\obj$ for this problem as
      \begin{equation}\label{eqn:AD_obj}
        \obj(\design) = \Trace{
          \left( \mat{M}\inv\mat{F}\tran 
            \Cobsnoise^{-1/2} \diag{\design} \Cobsnoise^{-1/2}
            \F
            + \Cparampriormat^{-1} \right)\inv
          }
            + \regpenalty \Phi(\design) \,, 
      \end{equation}
      where $\regpenalty$ is the user-defined penalty parameter.
      This parameter controls the level of sparsity that we desire to impose on
      the design.
      Specifically, we set $\Phi(\design):=\wnorm{\design}{0}$ to impose sparsity.
      On the other hand, if we have a specific budget $\budget$, it would be
      more reasonable to define the penalty function as
      $ 
      \Phi(\design):=\regpenalty \abs{\wnorm{\design}{0}-\budget}
        = \regpenalty \abs{ \sum_{i=1}^{\Nsens}{\design_i} -\budget}
      $.
      We will discuss these two cases in the following and in 
      the numerical experiments.
      The stochastic optimization
      problem~\eqref{eqn:stochastic_optimization} is formulated given the
      definition of $\obj$ in ~\eqref{eqn:AD_obj} as 
      \begin{equation}\label{eqn:AD_stochastic_optimization}
        \hyperparam\opt 
        = \argmin_{\hyperparam \in [0, 1]^{\Nsens} } 
          \Expect{\design\sim\CondProb{\design}{\hyperparam}}{
          \Trace{
          \left( \mat{M}\inv\mat{F}\tran 
            \Cobsnoise^{-1/2} \diag{\design} \Cobsnoise^{-1/2}
            \F
            + \Cparampriormat^{-1} \right)\inv
          }
            + \regpenalty \Phi(\design) 
          } 
       \,,
      \end{equation}
      where $\CondProb{\design}{\hyperparam}$ is the multivariate Bernoulli
      distribution with PMF given by~\eqref{eqn:joint_Bernoulli_pmf_prod}.

    \subsection{Numerical results with advection-diffusion model}
      \label{subsec:AD_Results_set1}
      The main goal of this set of experiments is to study the behavior of the
      proposed~\Cref{alg:REINFORCE_baseline} compared with  the global solution
      of~\eqref{eqn:AD_stochastic_optimization}.
      
      Solution by enumeration (brute-force) is carried out for all
      $2^{14}=16,384$ possible designs, and the corresponding value of $\obj$ is recorded 
      to identify the global solution of~\eqref{eqn:AD_stochastic_optimization}.
      In addition, we run~\Cref{alg:REINFORCE_baseline} 
      with the maximum number of iterations set to $20$. 
      We choose this tight number to test the performance of the 
      stochastic optimization algorithm upon early termination.
      We choose the learning rate $\eta = 0.25$ and set 
      the gradient tolerance \textsc{pgtol} to $10^{-8}$.
      Each sensor is equipped with an initial probability $0.5$. This is employed
      by choosing the initial parameter $\hyperparam^{(0)}$ of the stochastic
      optimization algorithm to 
      $\hyperparam^{(0)}=(0.5, 0.5, \ldots, 0.5)\tran$.
      In all experiments, we set the batch size for estimating the
      stochastic gradient to $32$ and the
      number of epochs for the optimal baseline to $10$. 
      
      The optimization algorithm returns samples from the multivariate Bernoulli
      distribution associated with the parameter $\hyperparam$ at the final
      step. Then, it picks $\design\opt$ as the sampled design associated with the
      smallest value of $\obj$.
      We assume that the optimization procedure samples $10$ designs
      upon termination, from the final distribution.
      Note that all samples will be identical if the probability distribution is 
      degenerate. 
      In the numerical results discussed next, we show not only the
      final optimal design returned by the optimization procedure but also the
      sampled designs. 

      \subsubsection{Results without penalty term}
      \label{subsubsec:AD_Results_NoPenalty}
        We start with numerical results obtained by setting the penalty
        parameter $\regpenalty=0$.
        \Cref{fig:AD_Set1_REINFORCE_NoPenalty_NoBaseline} shows the results of the
        brute-force search, along with results returned by~\Cref{alg:REINFORCE}.
        Specifically, in~\Cref{fig:AD_Set1_REINFORCE_NoPenalty_NoBaseline} (left),
        the value of $\obj:=\Trace{\Cparampost(\design)}$ is evaluated at each
        possible binary design $\design$ and is shown on the y-axis. 
        Candidate binary designs are grouped on the x-axis by the number of entries 
        set to $1$, that is, the number of active sensors. 
        In this setup, we have access to the value of $\obj$ corresponding to
        all possible designs, and thus we can in fact evaluate
        $\stochobj(\hyperparam) \equiv \Expect{\design\sim
        \CondProb{\design}{\hyperparam}}{\obj(\design)}$ exactly, for any choice
        of the parameter $\hyperparam$. Of course, such an action is impossible in
        practice; however, we are interested in understanding the behavior of
        the optimization algorithm. 
        \Cref{fig:AD_Set1_REINFORCE_NoPenalty_NoBaseline} (right) shows the value
        of $\stochobj$ evaluated at the $k$th step
        of~\Cref{alg:REINFORCE}. 
        %
        \begin{figure}[ht] \centering
          \includegraphics[align=c,width=0.55\linewidth]{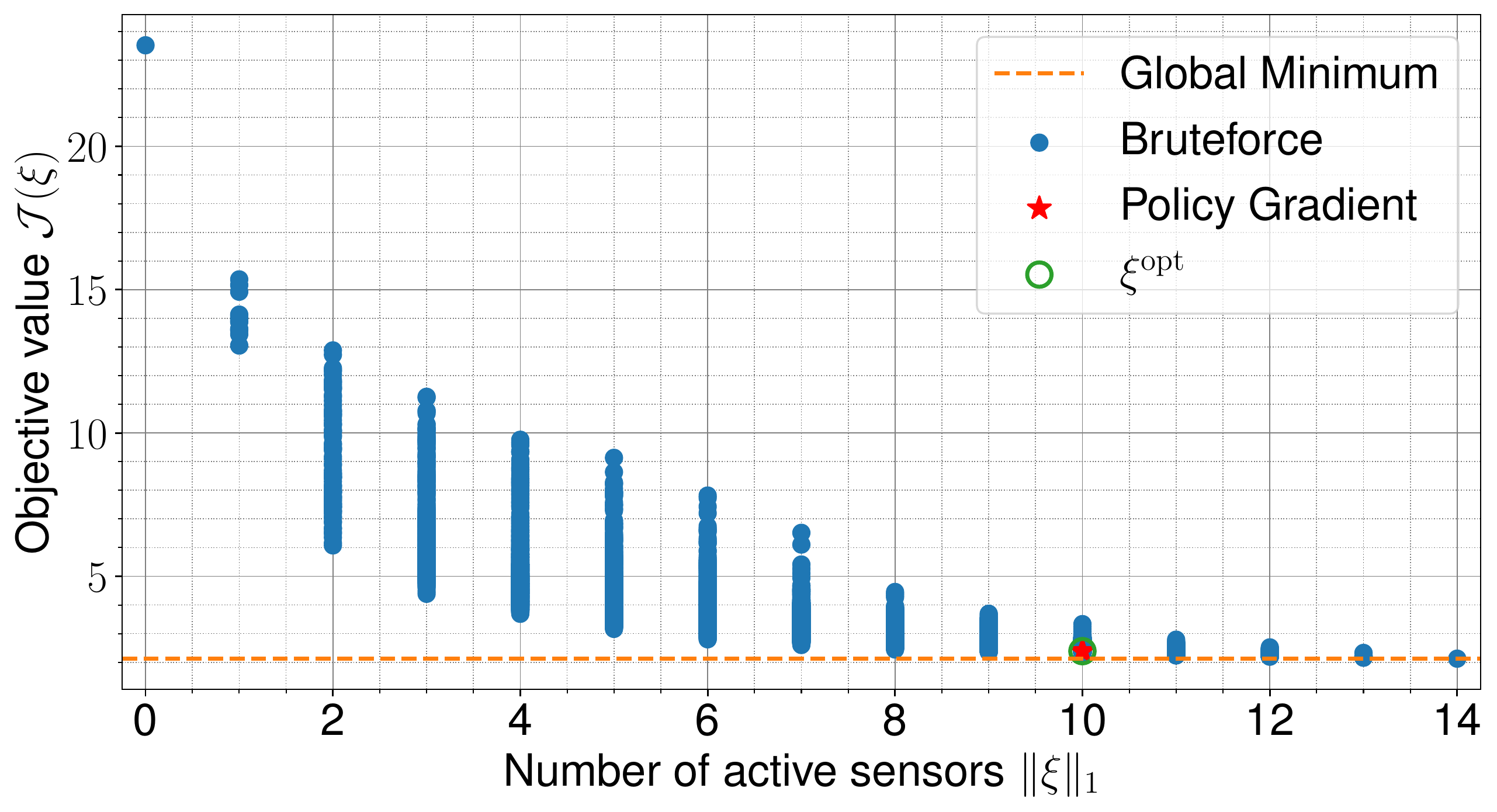}
          \quad
          \includegraphics[align=c,width=0.40\linewidth]{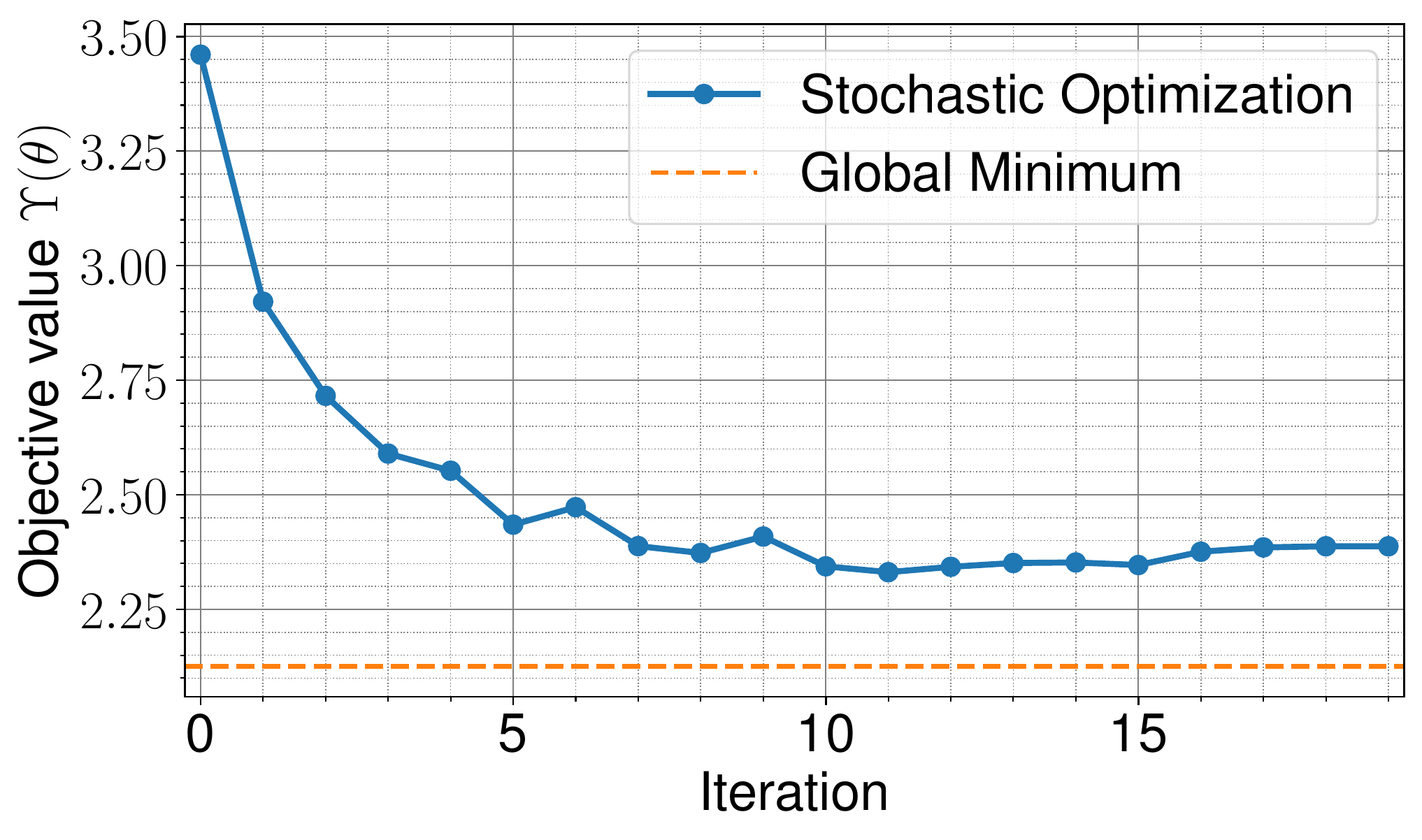}

          \caption{Results of the policy gradient~\Cref{alg:REINFORCE} compared with brute-force
            search of all candidate binary designs.
            No penalty is used here; that is, we set the penalty parameters
            $\regpenalty=0$. 
            Left:
            candidate designs are grouped by the number of active sensors, on the
            x-axis, with the corresponding value of $\obj$ displayed on the y-axis.
            Brute-force results are shown as blue dots.
            The results of~\autoref{algstep:REINFORCE:sampling}
            of~\autoref{alg:REINFORCE} with $m=10$ are shown as red
            stars, and the optimal solution returned from the algorithm is shown as a
            green circle.
            Right: The value of the stochastic objective
            $\stochobj(\hyperparam)$ evaluated at the $\hyperparam^{(k)}$ at
            each iteration $k$ of~\Cref{alg:REINFORCE}. 
          }
          \label{fig:AD_Set1_REINFORCE_NoPenalty_NoBaseline}
        \end{figure}
        
        In this setup, without any constraints on the number of sensors,
        the global optimal minimum is attained by
        $\design\opt=\vec{1}\in\Rnum^{\Nsens}$, that is, by activating all
        sensors.
        However, we note that increasing the number of sensors, say more 
        than $8$, would add little to information gain from data. 
        The reason is the similarity of the values 
        of $\obj$ for all designs with more than $8$ active sensors.
        The designs sampled from the final distribution obtained
        by~\Cref{alg:REINFORCE} are marked as red stars, which in this case are
        identical, showing that the final probability distribution is
        degenerate.
        The algorithm moves quickly toward a local minimum, but it fails to
        explore the space near the global optimum.
        This action is expected because of sampling error and the high
        variability of the estimator.
        As discussed in~\ref{subsec:variance_reduction_baseline}, improvements
        could be achieved by incorporating baseline in the objective function.
        
        In~\Cref{fig:AD_Set1_REINFORCE_NoPenalty_Baseline}, we show results
        obtained by introducing baseline $\baseline$ to the stochastic gradient
        estimator, as described by~\Cref{alg:REINFORCE_baseline}.
        We show results with both the heuristic baseline 
        estimate~\eqref{eqn:empirical_baseline}~
        (\Cref{fig:AD_Set1_REINFORCE_NoPenalty_Baseline} (top)) and 
        the optimal baseline estimate~\eqref{eqn:optimal_baseline}~
        (\Cref{fig:AD_Set1_REINFORCE_NoPenalty_Baseline} (bottom)).
        Both~\Cref{alg:REINFORCE} and~\Cref{alg:REINFORCE_baseline} result in
        probability distributions (defined by $\hyperparam$) associated with
        small values.
        However,~\Cref{alg:REINFORCE_baseline} with the optimal
        baseline~\eqref{eqn:optimal_baseline} outperforms
        both~\Cref{alg:REINFORCE}, and~\Cref{alg:REINFORCE_baseline} with
        the heuristic baseline~\eqref{eqn:empirical_baseline} 
        and generates designs with significantly smaller objective values.
        Specifically, as shown in~\Cref{fig:AD_Set1_REINFORCE_NoPenalty_Baseline} (bottom), 
        the objective value $\obj$ evaluated at the designs generated 
        by~\Cref{alg:REINFORCE_baseline} are all similar and fall within $1\%$ 
        of the global optimum.
        \begin{figure}[ht] \centering
          \includegraphics[align=c,width=0.55\linewidth]{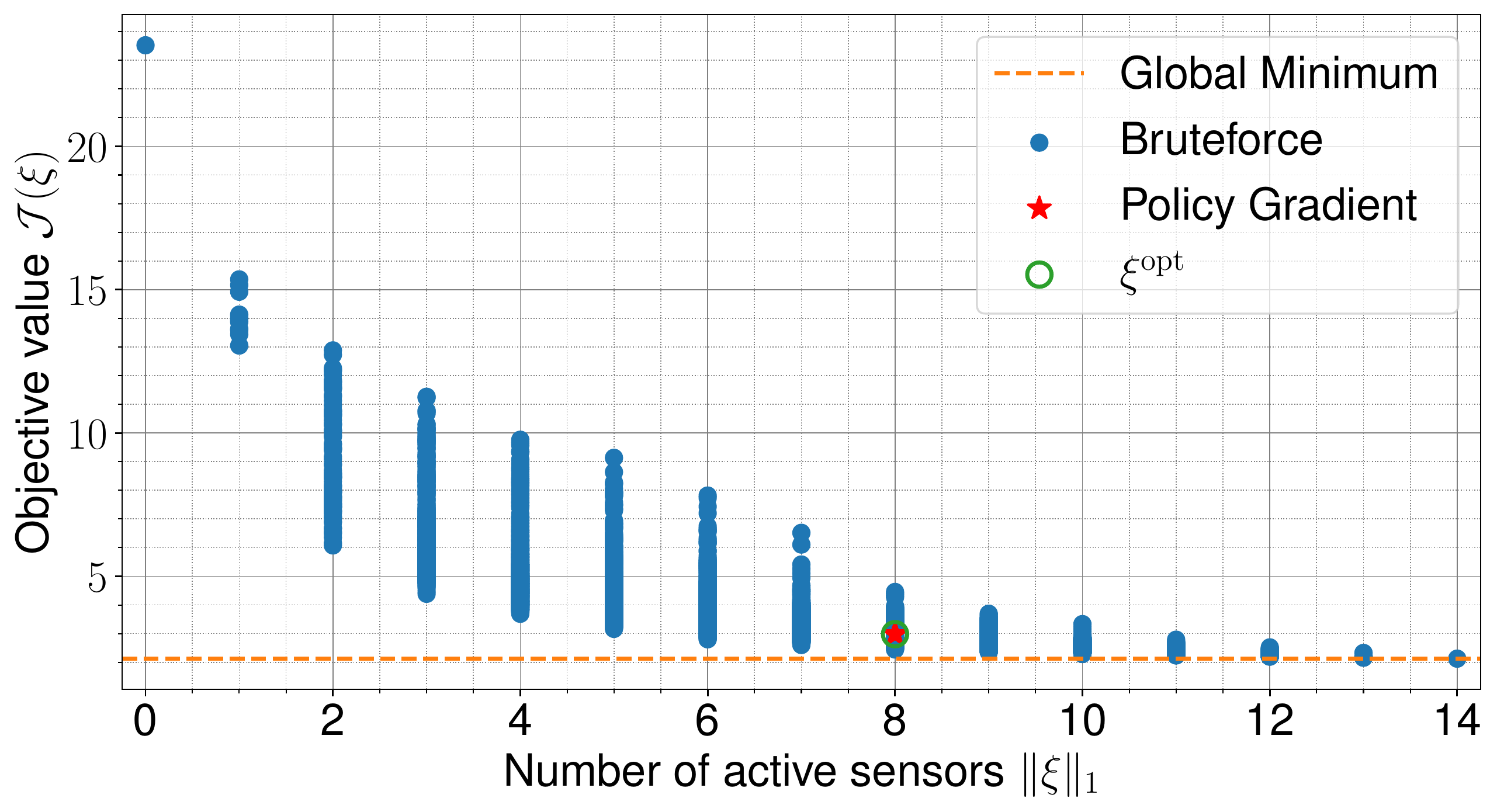}
          \quad
          \includegraphics[align=c,width=0.40\linewidth]{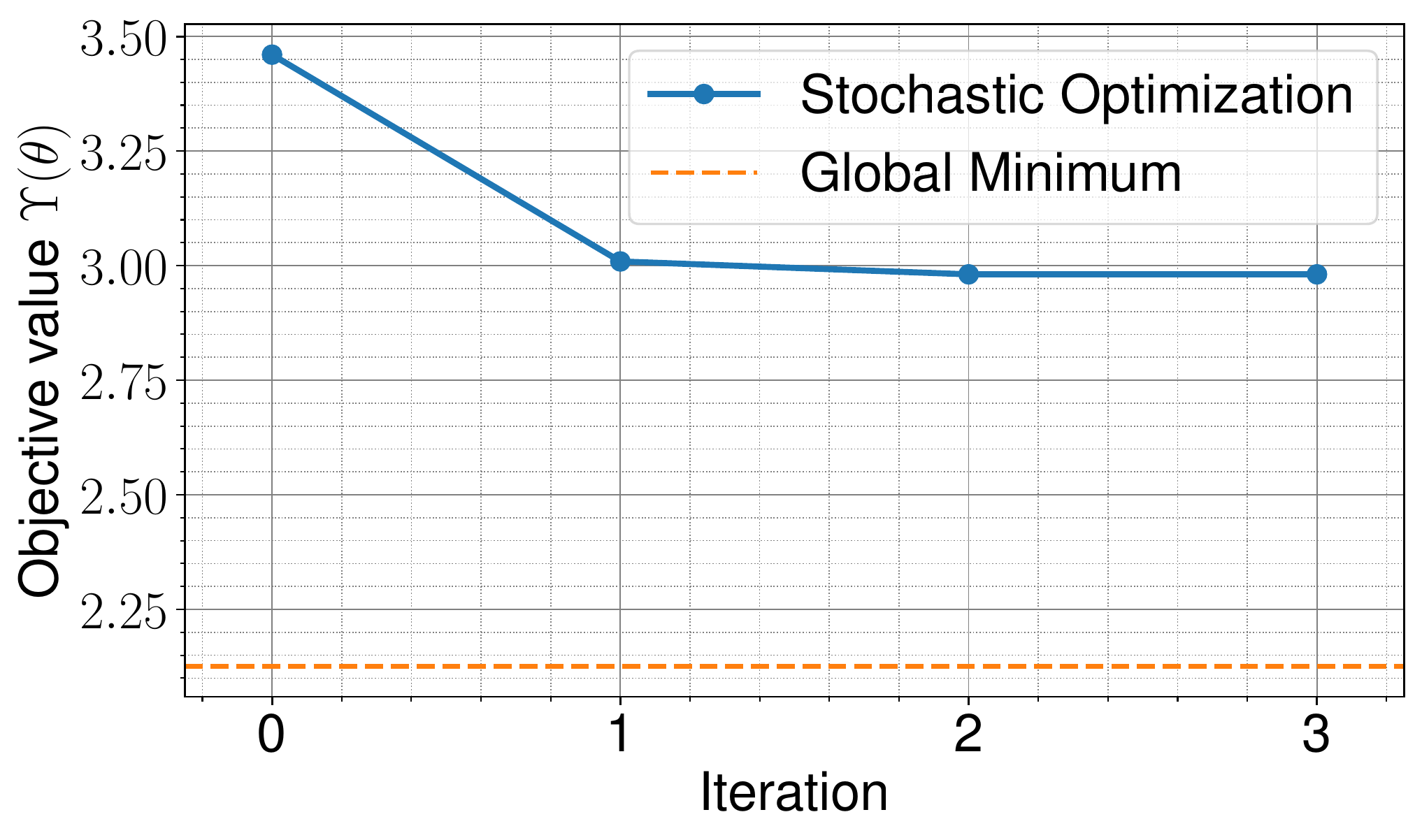}
          \includegraphics[align=c,width=0.55\linewidth]{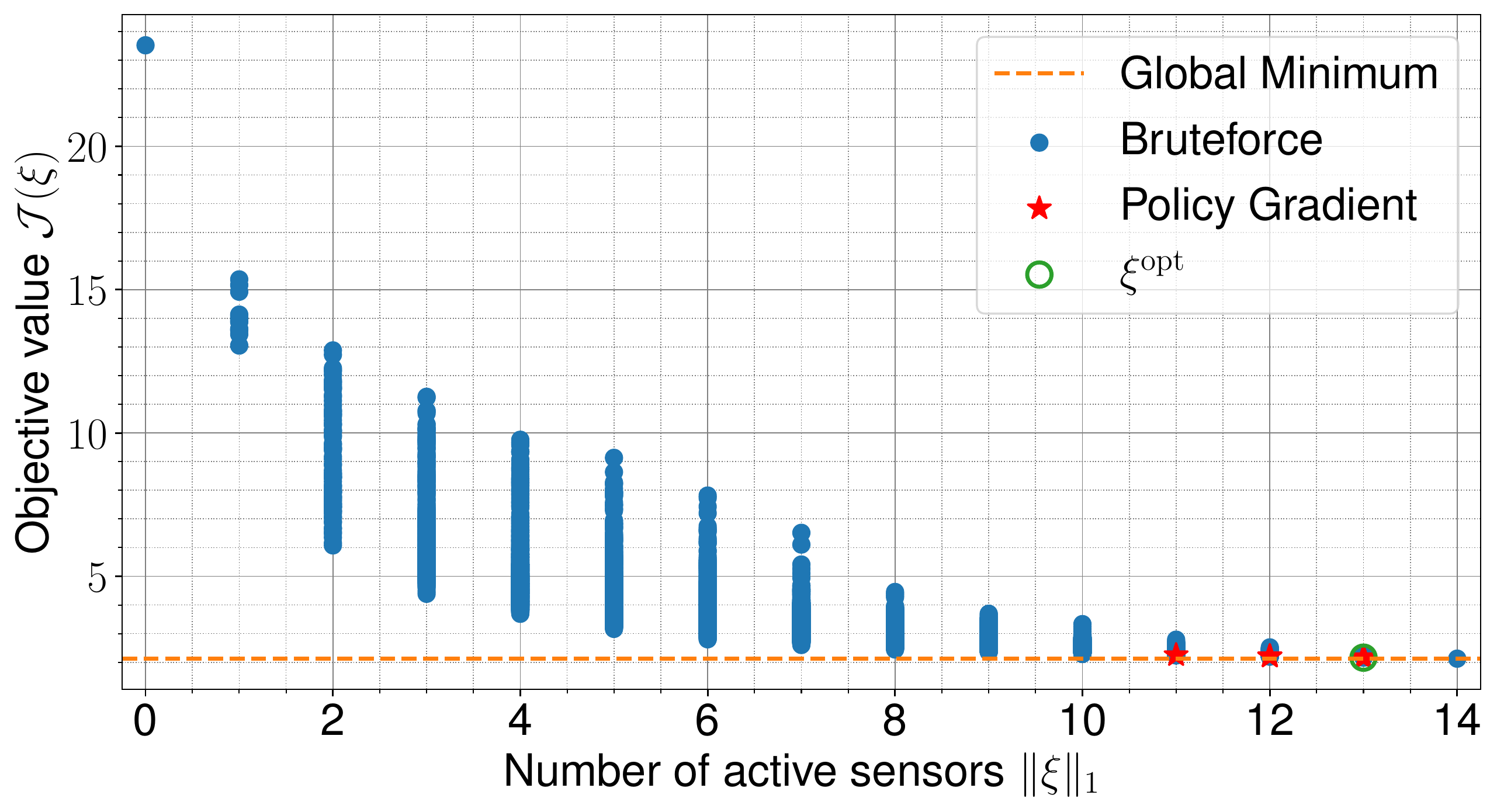}
          \quad
          \includegraphics[align=c,width=0.40\linewidth]{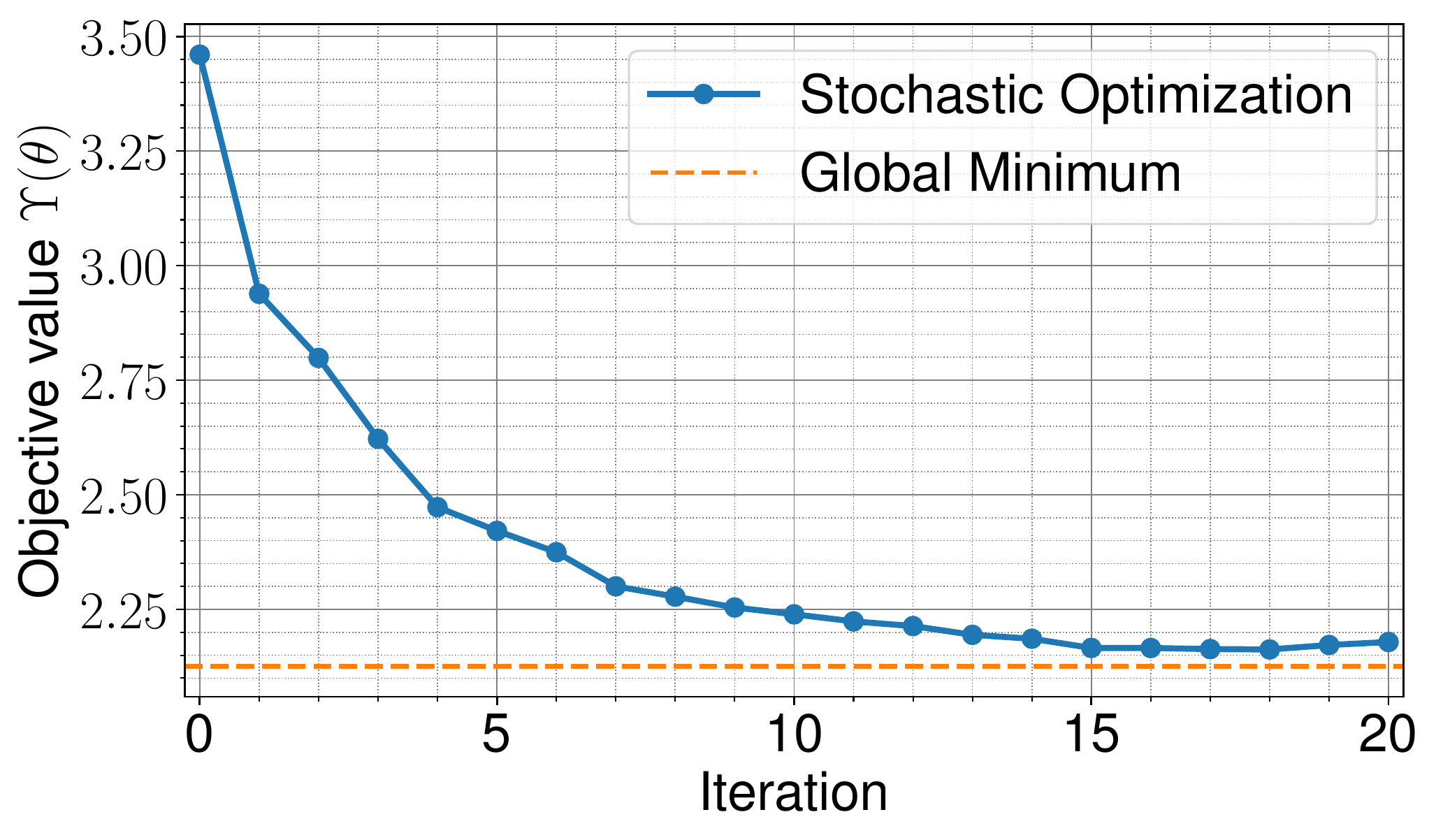}

          \caption{Same as~\Cref{fig:AD_Set1_REINFORCE_NoPenalty_NoBaseline}.
            Here, we show results of~\Cref{alg:REINFORCE_baseline}, that is,
            stochastic optimization with the baseline.
            The top panels show results with the heuristic
            baseline estimate~\eqref{eqn:empirical_baseline}.
            The bottom panels show results with the 
            optimal baseline estimate~\eqref{eqn:optimal_baseline}. 
          }
          \label{fig:AD_Set1_REINFORCE_NoPenalty_Baseline}
        \end{figure}

        Evaluating the optimal baseline estimate, however, requires additional
        evaluations of $\obj$.
        We monitor the number of additional evaluations
        of $\obj$ carried out at each iteration of the optimizer.
        Note that we keep track of the values of $\obj$ for each sampled design
        $\design$ during the course of the algorithm. By doing so, we avoid any
        computational redundancy due to recalculating the objective function
        multiple times for the same design.
        \Cref{fig:AD_Set1_REINFORCE_NoPenalty_function_evaluations} shows the
        number of new function evaluations carried out at each step of the
        optimization algorithm.
        \begin{figure}[ht] \centering
          \includegraphics[width=0.32\linewidth]{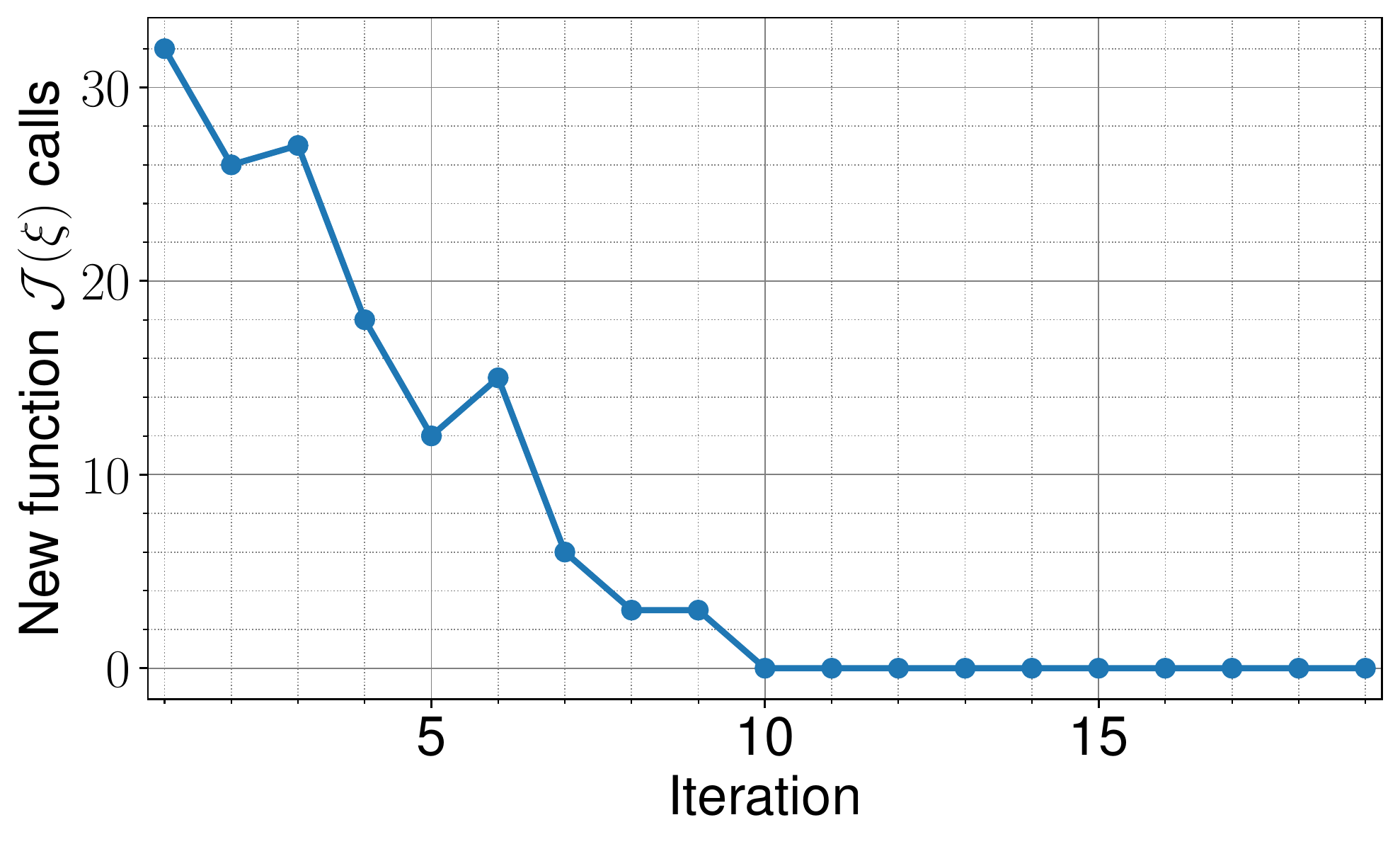}
          \includegraphics[width=0.32\linewidth]{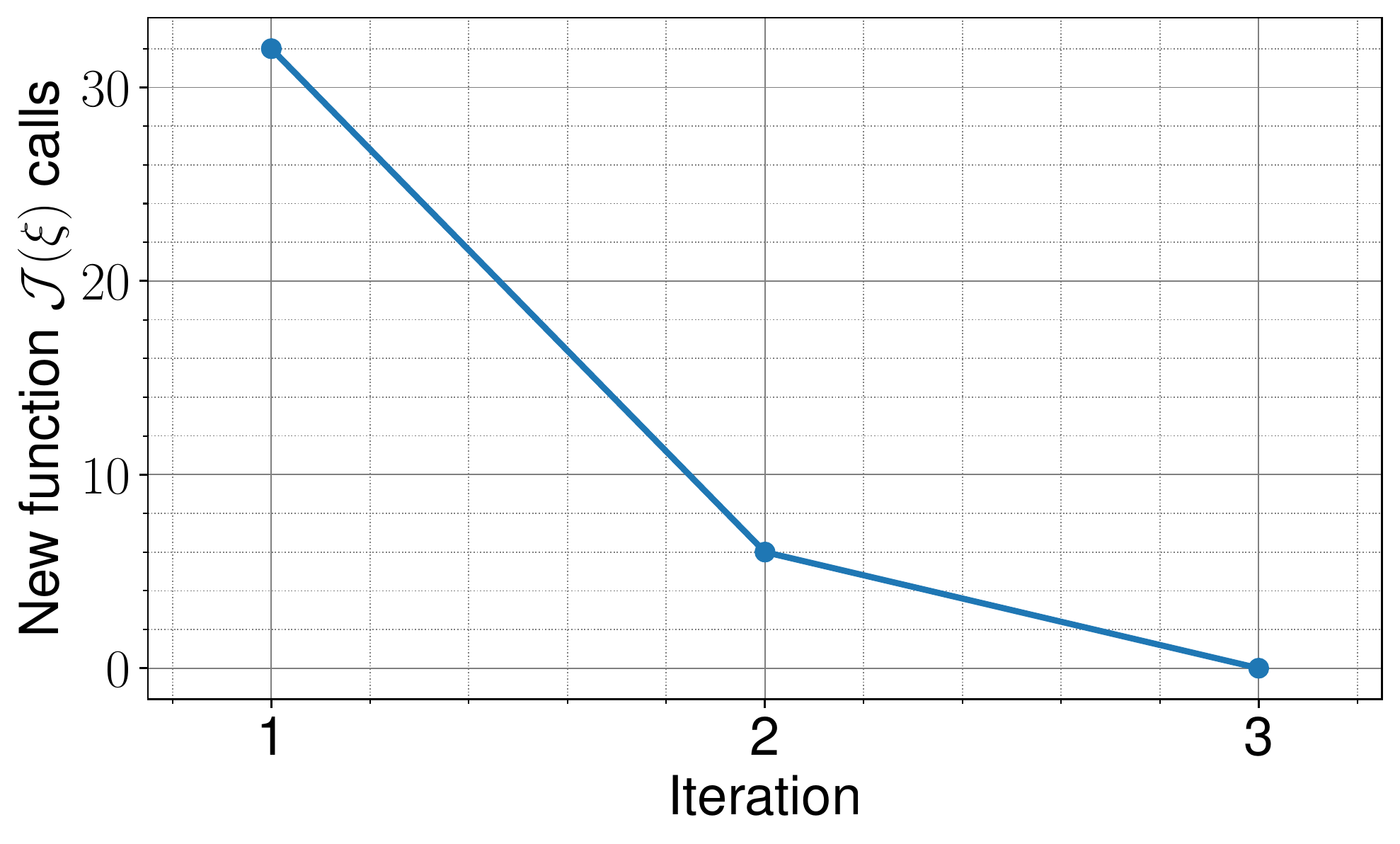}
          \includegraphics[width=0.32\linewidth]{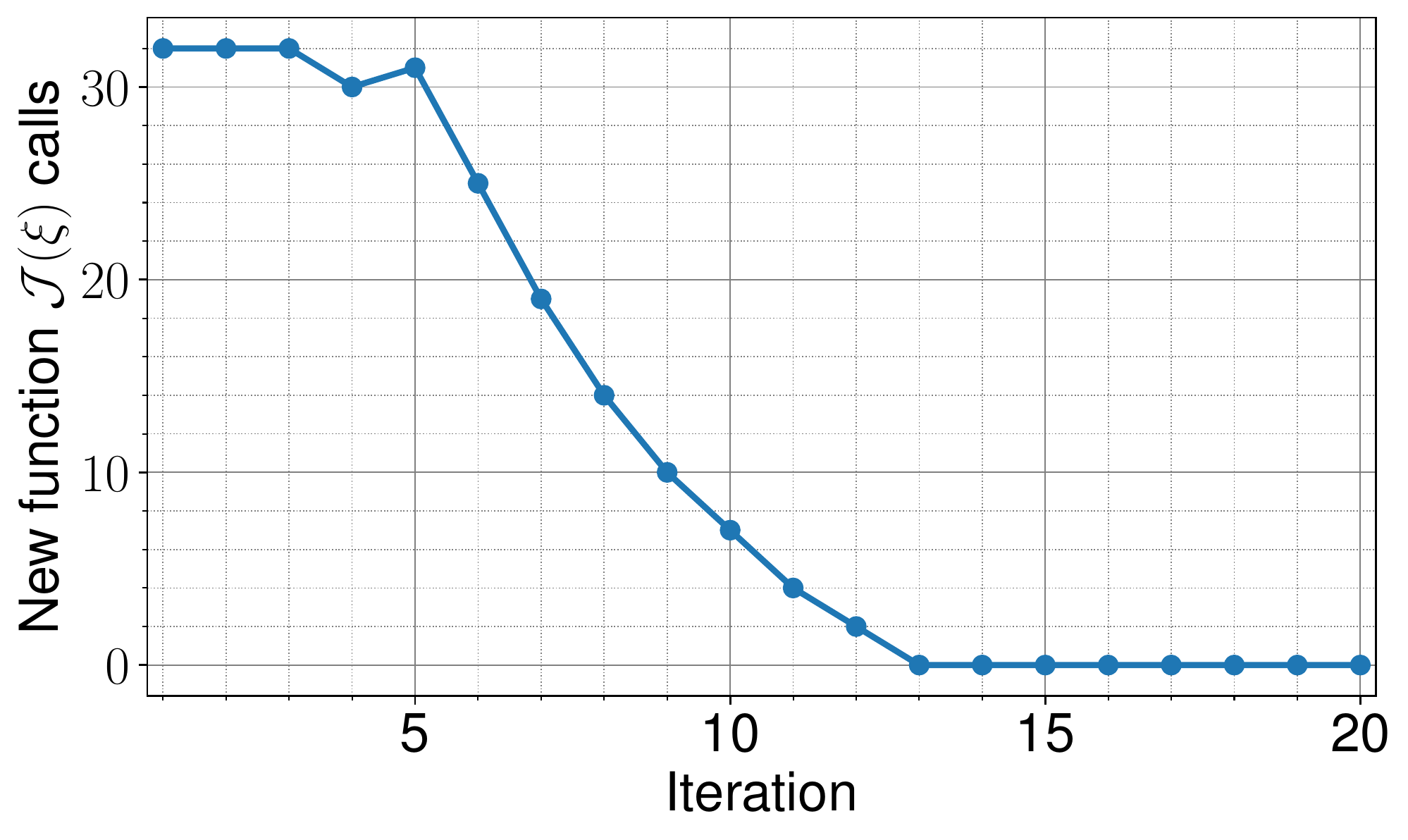}

          \caption{Number of new function evaluations carried out by the
          stochastic optimization algorithm.
          Left: results of~\Cref{alg:REINFORCE}. 
          Middle: results of~\Cref{alg:REINFORCE_baseline} with 
          the heuristic baseline estimate~\eqref{eqn:empirical_baseline}.
          Right: results of~\Cref{alg:REINFORCE_baseline} with the 
            optimal baseline estimate~\eqref{eqn:optimal_baseline}.
          }
          \label{fig:AD_Set1_REINFORCE_NoPenalty_function_evaluations}
        \end{figure}
        
        \Cref{fig:AD_Set1_REINFORCE_NoPenalty_function_evaluations} (left) 
        suggests that, by using the heuristic baseline~\eqref{eqn:empirical_baseline}, 
        the optimization algorithm converges quickly to a suboptimal probability 
        space and does not require many additional function evaluations. 
        A smaller step size $\eta$, in this case, might result in better performance.
        Conversely, comparing results in both
        \Cref{fig:AD_Set1_REINFORCE_NoPenalty_function_evaluations} (left) and
        \Cref{fig:AD_Set1_REINFORCE_NoPenalty_function_evaluations} (right),
        we notice that the computational cost, explained by the number of function
        evaluations, is not significantly different, especially after the first
        few iterations.
      
      \subsubsection{Results with sparsity constraint}
      \label{subsubsec:AD_Results_SprsityConstraint}
        To study the behavior of the optimization procedures in the presence 
        $\ell_0$ sparsity constraints, we set the penalty function  to 
        $\Phi(\design):=\wnorm{\design}{0}$
        and  the regularization penalty parameter to $\regpenalty=1.0$.
        Here we do not concern ourselves with the choice of $\regpenalty$, and we
        leave it for the user to tune based on the application at hand and the
        required level of sparsity.
        Results are shown in~\Cref{fig:AD_Set1_REINFORCE_Penalty_1.0} and
        ~\Cref{fig:AD_Set1_REINFORCE_function_evaluations_Penalty_1.0}, respectively. 
        For clarity we omit results obtained by the heuristic baseline.
        \begin{figure}[ht] \centering
          \includegraphics[align=c,width=0.55\linewidth]{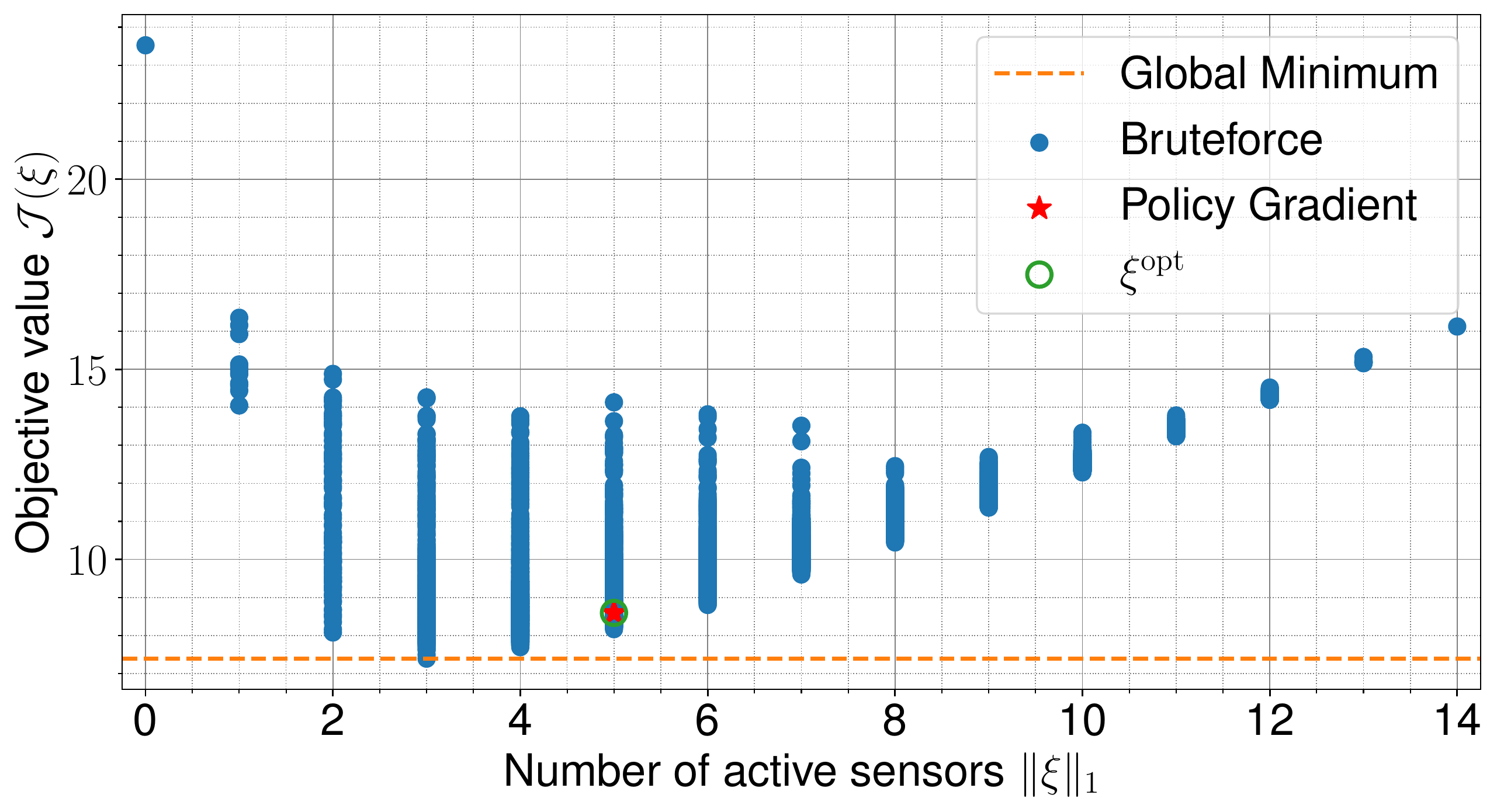}
          \quad
          \includegraphics[align=c,width=0.40\linewidth]{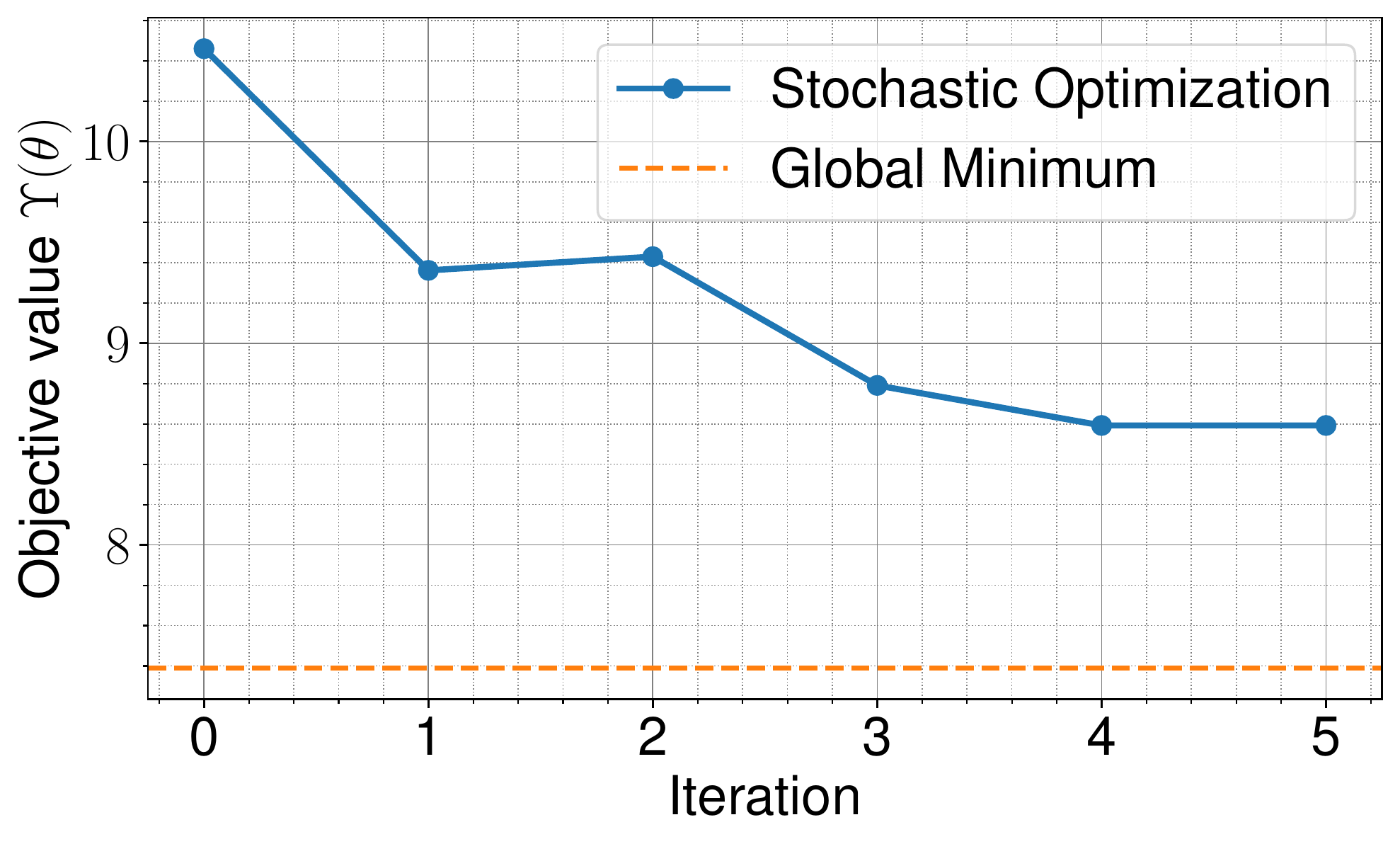}
          %
          %
          \includegraphics[align=c,width=0.55\linewidth]{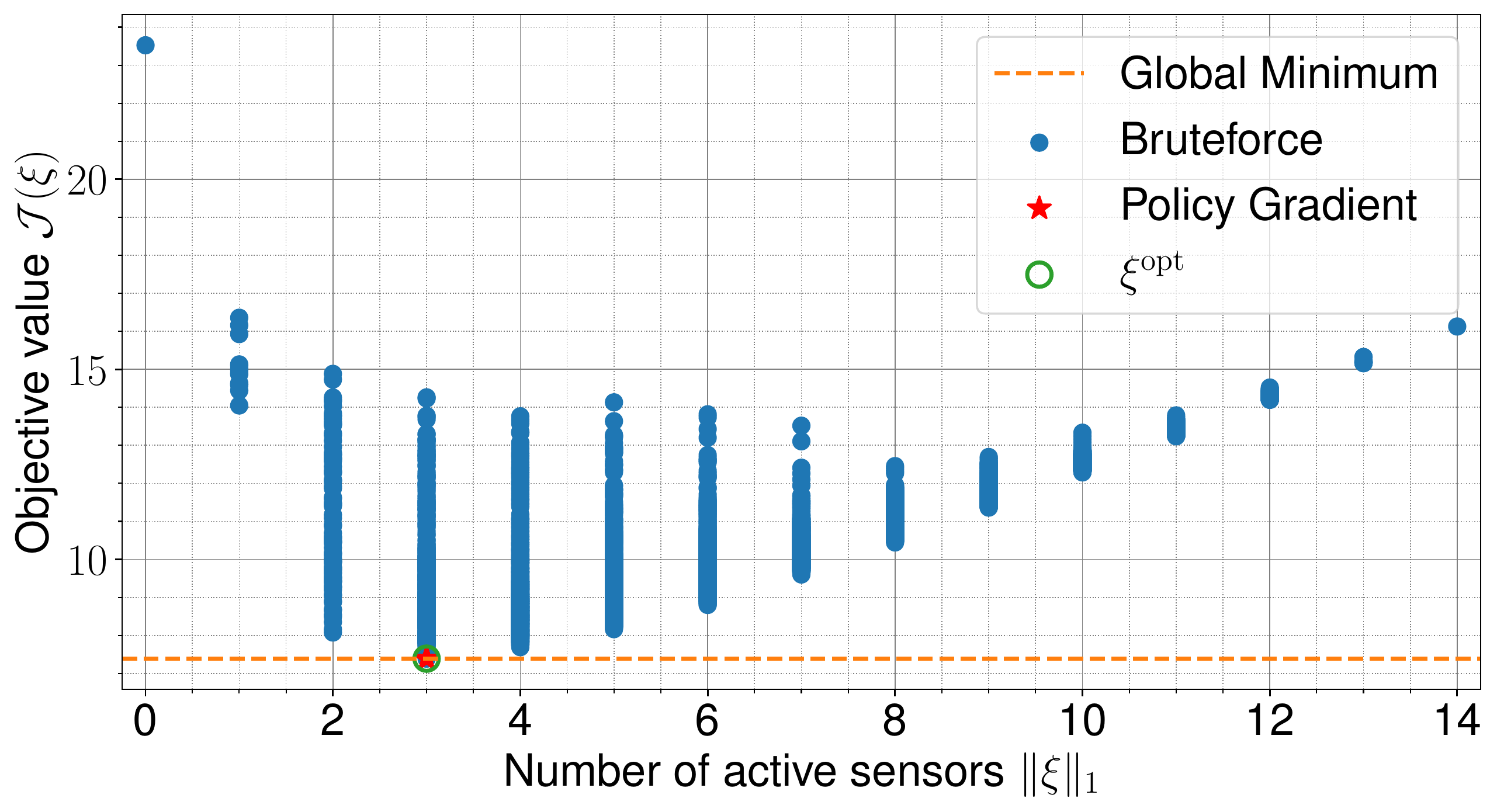}
          \quad
          \includegraphics[align=c,width=0.40\linewidth]{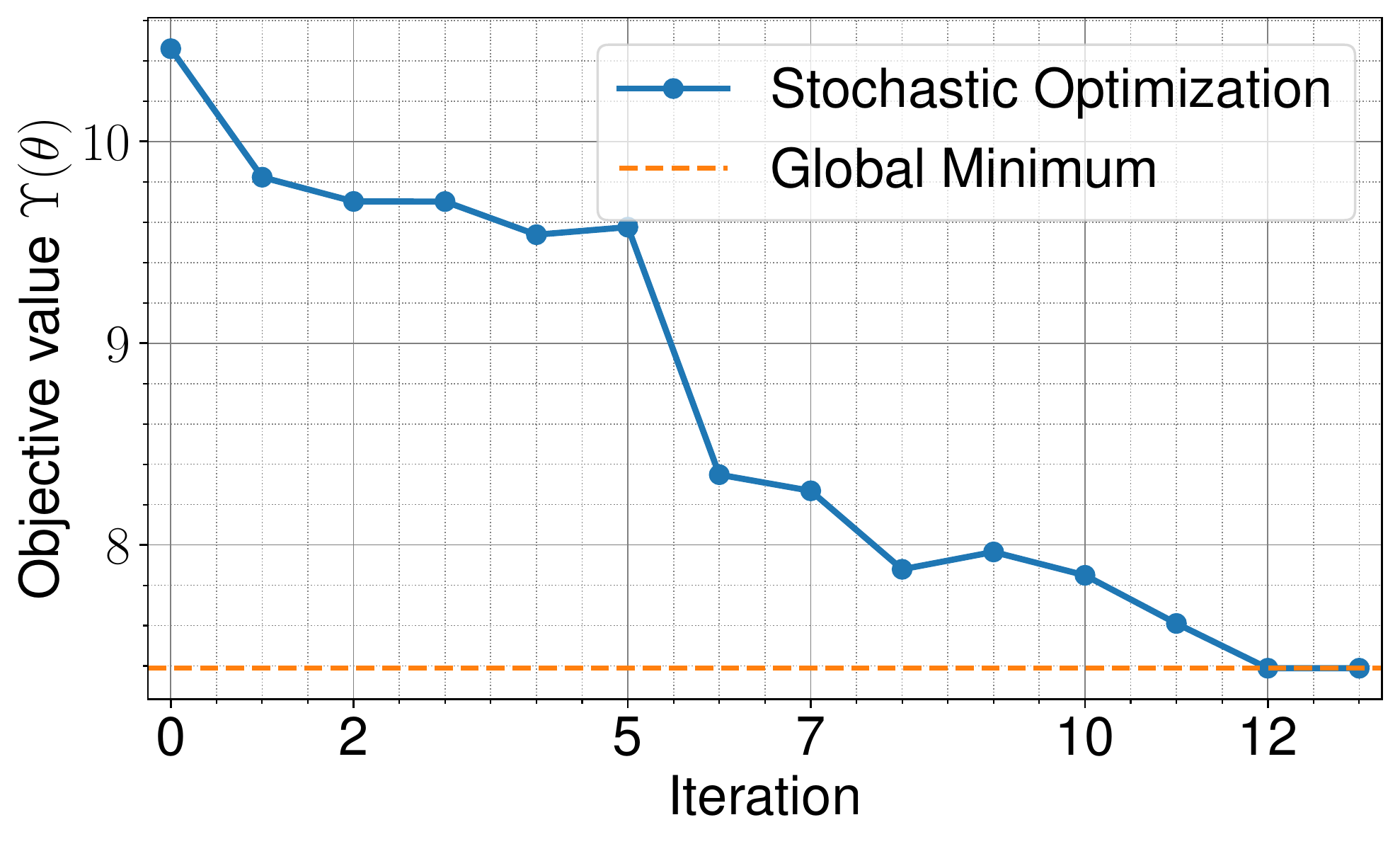}

          \caption{
            Results of the policy gradient procedures
            (~\Cref{alg:REINFORCE},~\Cref{alg:REINFORCE_baseline}), compared with the brute-force
            search of all candidate binary designs.
            Here, we set the sparsity penalty parameter to $\regpenalty=1.0$ and use
            a sparsity constraint, defined by $\Phi(\design):=\wnorm{\design}{0}$.
            Top: results of ~\Cref{alg:REINFORCE}.
            Bottom: ~\Cref{alg:REINFORCE_baseline}) with the optimal baseline
            estimate~\eqref{eqn:optimal_baseline}.
          }
          \label{fig:AD_Set1_REINFORCE_Penalty_1.0}
        \end{figure}
        
        As suggested by~\Cref{fig:AD_Set1_REINFORCE_Penalty_1.0} (left), there is
        a unique global optimum design with only $3$ active sensors.
        Both~\Cref{alg:REINFORCE}, and~\Cref{alg:REINFORCE_baseline} 
        result in degenerate probability distributions. The global optimal
        design, however, is attained by utilizing the optimal baseline estimate 
        as shown in~\Cref{fig:AD_Set1_REINFORCE_Penalty_1.0} (bottom).
        The computational cost of both algorithms, explained by the number of
        objective function evaluations,  is shown
        in~\Cref{fig:AD_Set1_REINFORCE_function_evaluations_Penalty_1.0}.
        \begin{figure}[ht] \centering
          \includegraphics[width=0.32\linewidth]{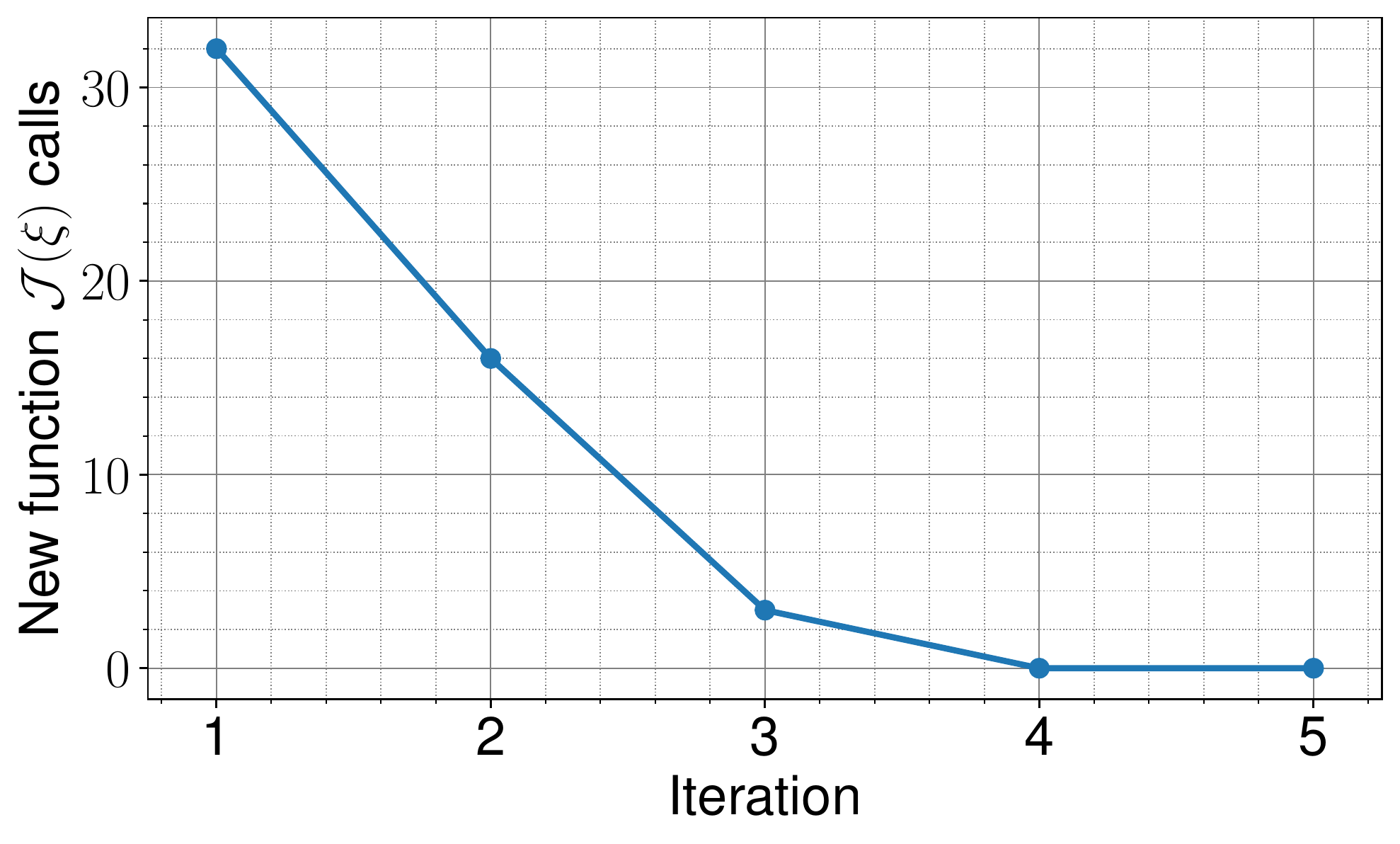}
          \quad
          \includegraphics[width=0.32\linewidth]{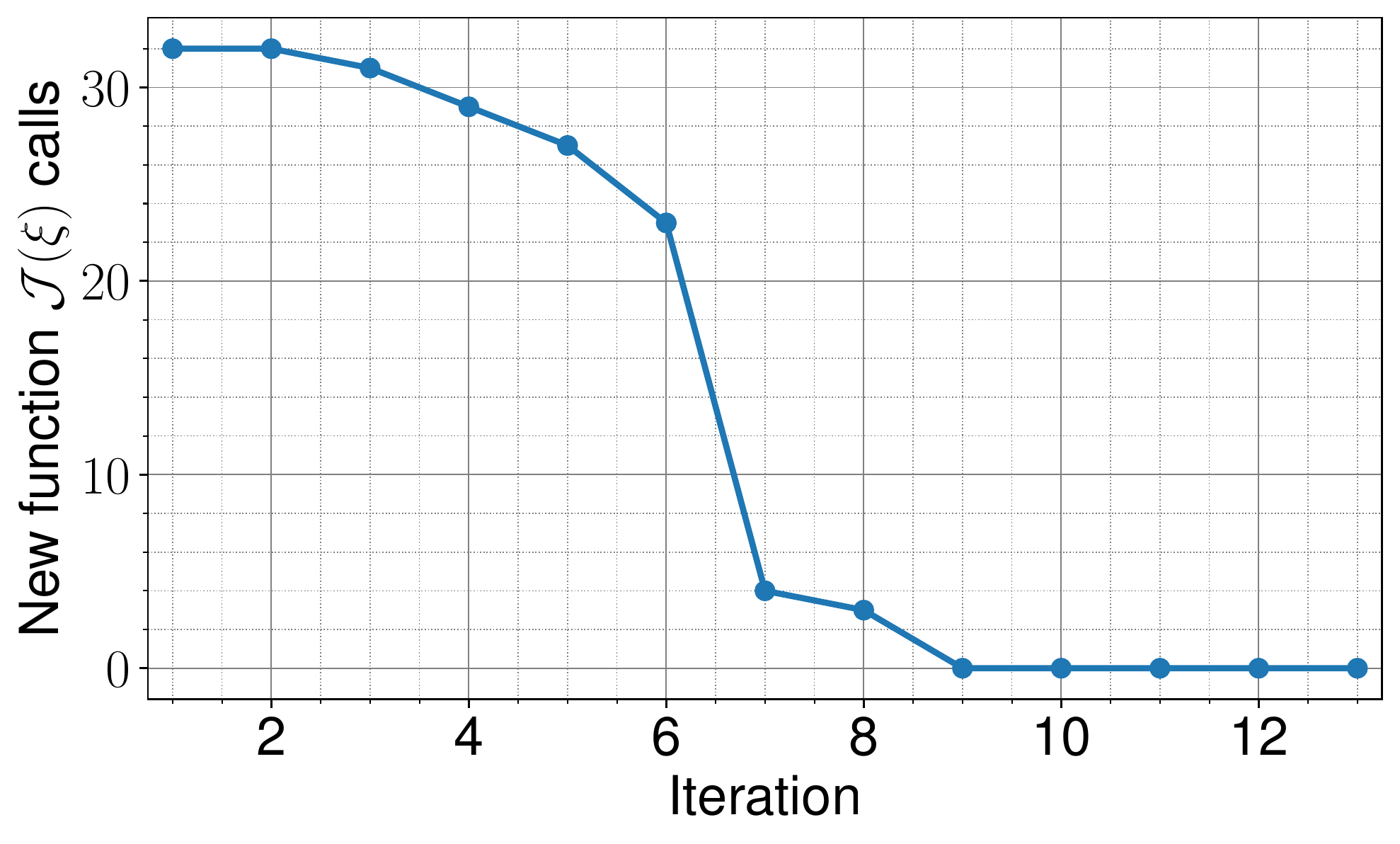}

          \caption{Similar to~\Cref{fig:AD_Set1_REINFORCE_NoPenalty_function_evaluations}.
            Here, we set the sparsity penalty parameter to $\regpenalty=1.0$ and use
            sparsity constraint, defined by $\Phi(\design):=\wnorm{\design}{0}$.
          }
          \label{fig:AD_Set1_REINFORCE_function_evaluations_Penalty_1.0}
        \end{figure}

      \subsubsection{Results with fixed-budget constraint}
      \label{subsubsec:AD_Results_BudgetConstraint}
        To study the behavior of the optimization algorithms in the presence of 
        an exact budget constraint $\wnorm{\design}{0}=\budget$, we carry out 
        the same procedure, with the penalty function set to
        $\Phi(\design):=\abs{\wnorm{\design}{0} - \budget}$
        and  the regularization penalty parameter set to $\regpenalty=1.0$,
        and we set the budget to $\budget=8$ sensors.
        Results are shown in~\Cref{fig:AD_Set1_REINFORCE_Budget_8_Penalty_1.0} and \Cref{fig:AD_Set1_REINFORCE_function_evaluations_Budget_8_Penalty_1.0},
        respectively. For clarity we omit results obtained by the 
        heuristic baseline.
        \begin{figure}[ht] \centering
          \includegraphics[align=c,width=0.55\linewidth]{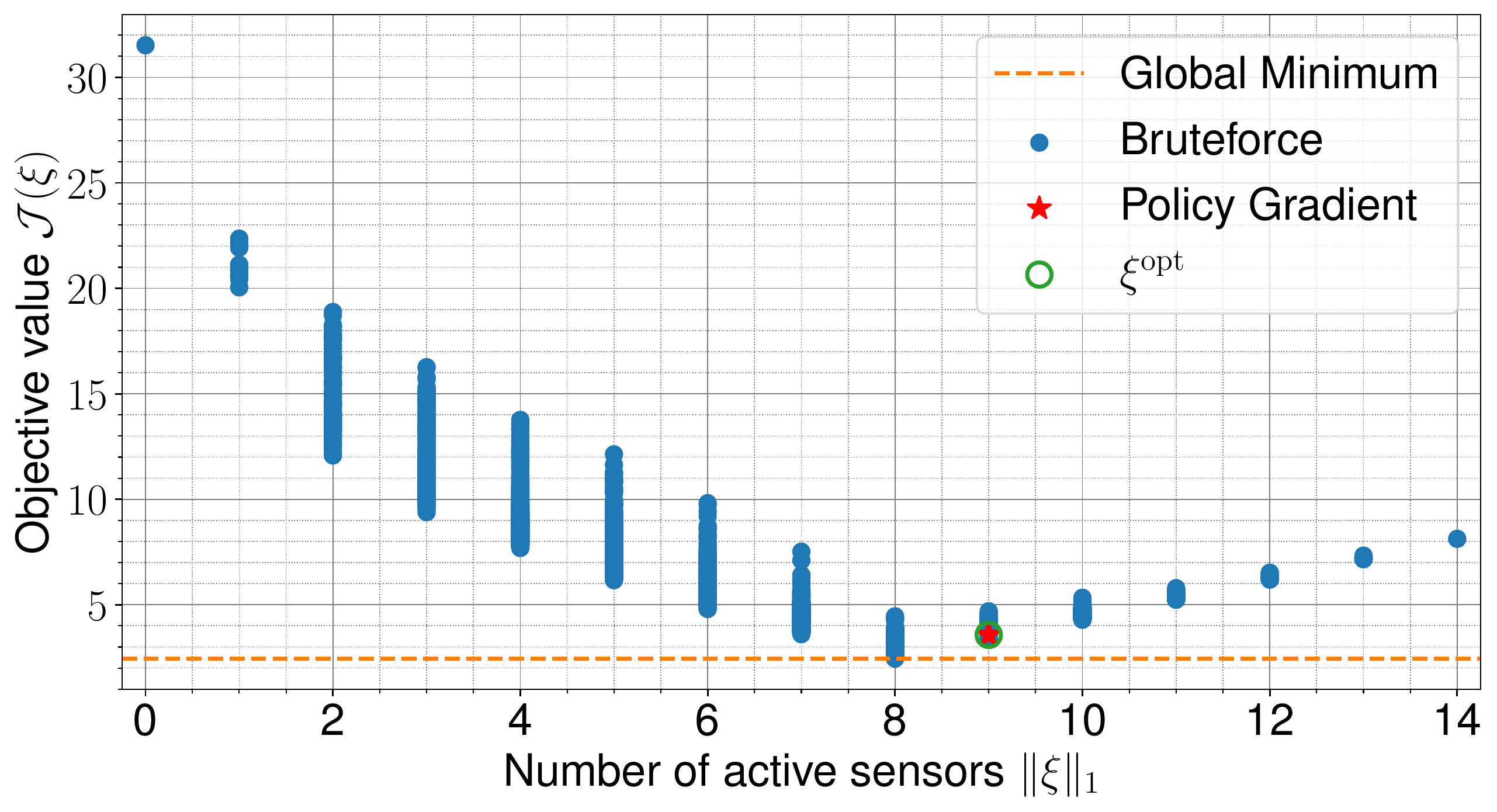}
          \quad
          \includegraphics[align=c,width=0.40\linewidth]{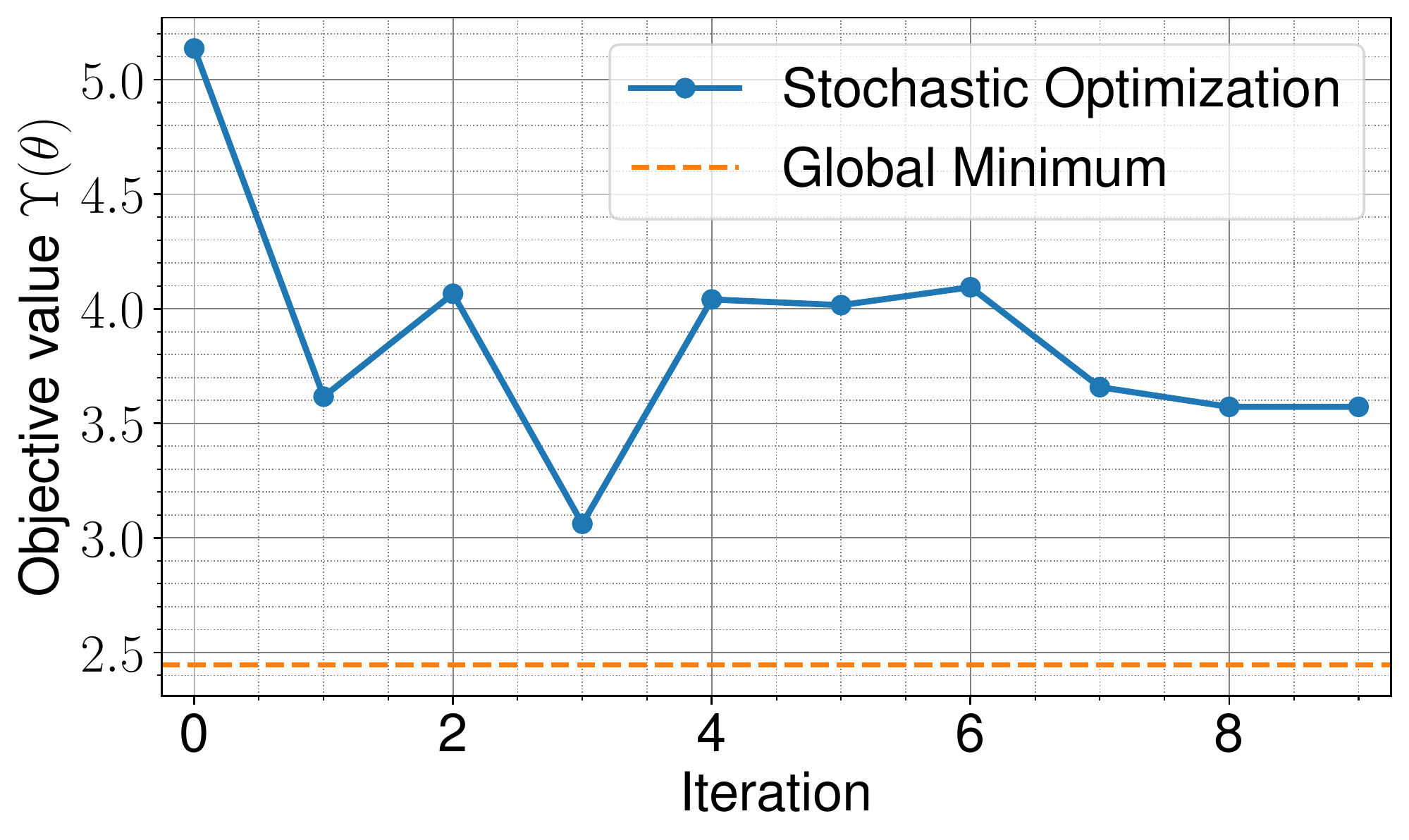}
          %
          %
          \includegraphics[align=c,width=0.55\linewidth]{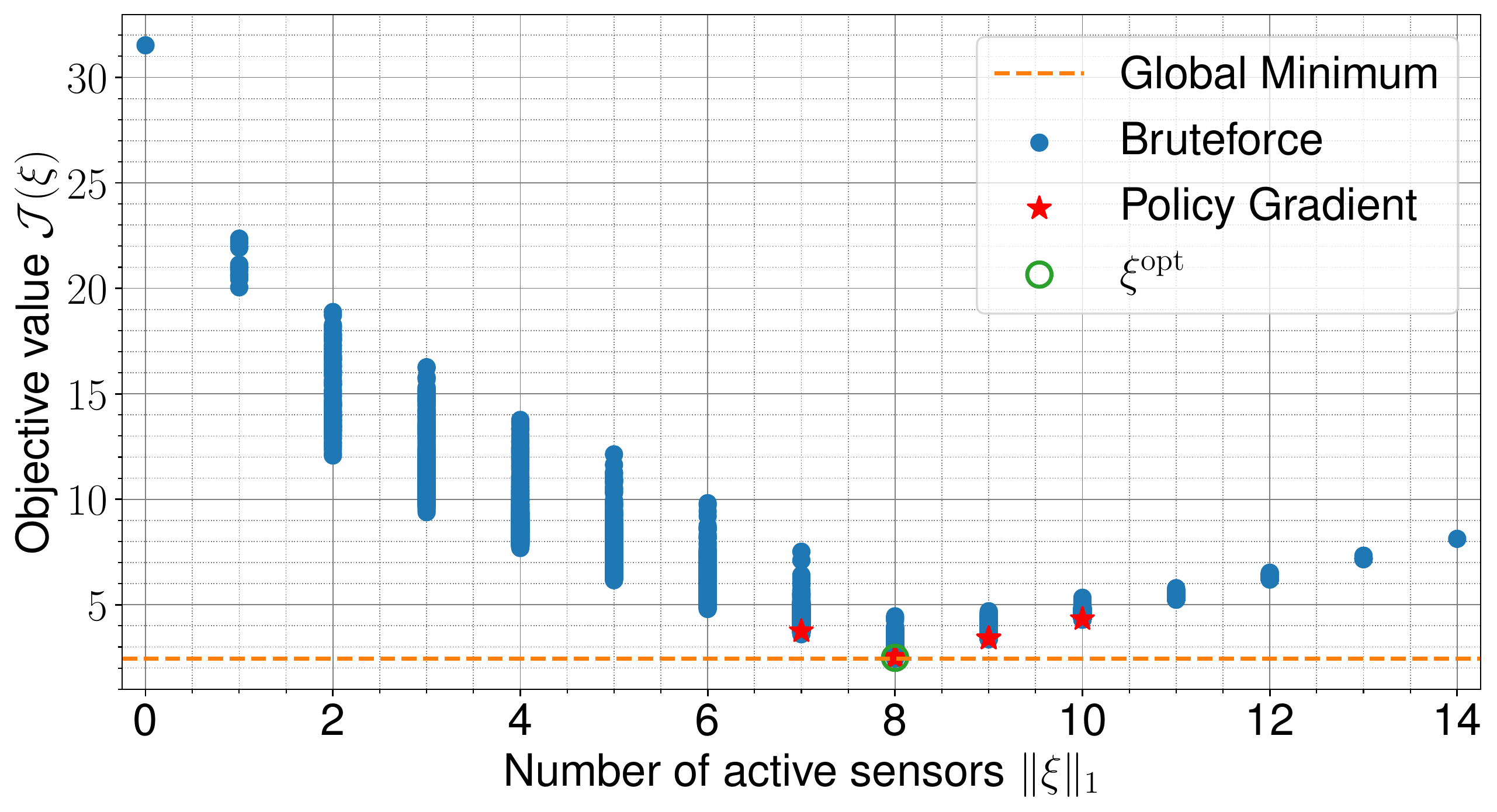}
          \quad
          \includegraphics[align=c,width=0.40\linewidth]{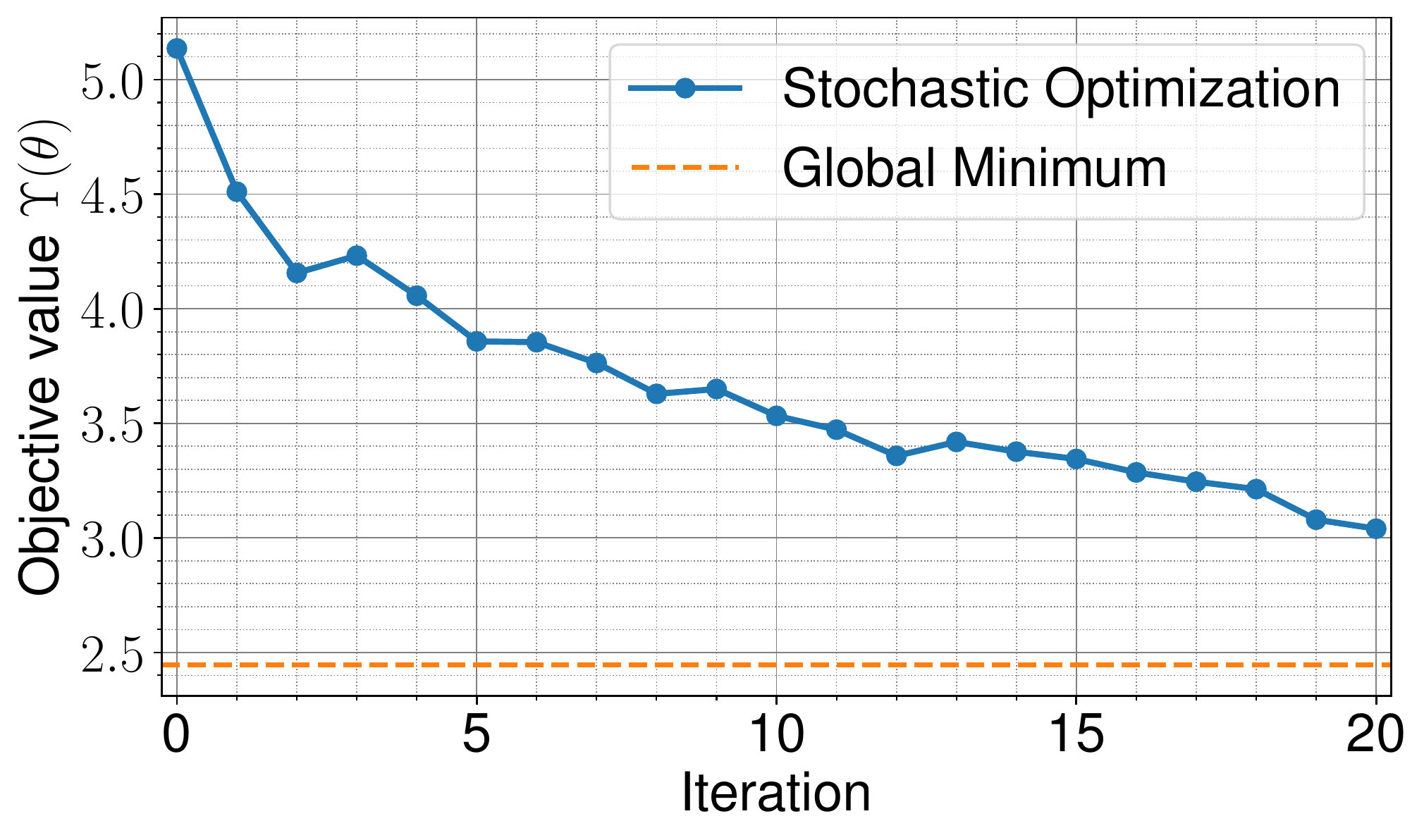}

          \caption{
            Results of the policy gradient procedures
            (~\Cref{alg:REINFORCE} and \Cref{alg:REINFORCE_baseline}), compared with the brute-force
            search of all candidate binary designs.
            Here, we set the sparsity penalty parameter to $\regpenalty=1.0$ and use
            budget constraint, defined by
            $\Phi(\design):=\wnorm{\design-\budget}{0}$, where $\budget=8$.
            Top: results of ~\Cref{alg:REINFORCE}.
            Bottom: ~\Cref{alg:REINFORCE_baseline}) with the optimal baseline
            estimate~\eqref{eqn:optimal_baseline}.
          }
          \label{fig:AD_Set1_REINFORCE_Budget_8_Penalty_1.0}
        \end{figure}
        
        \begin{figure}[ht] \centering
          \includegraphics[width=0.32\linewidth]{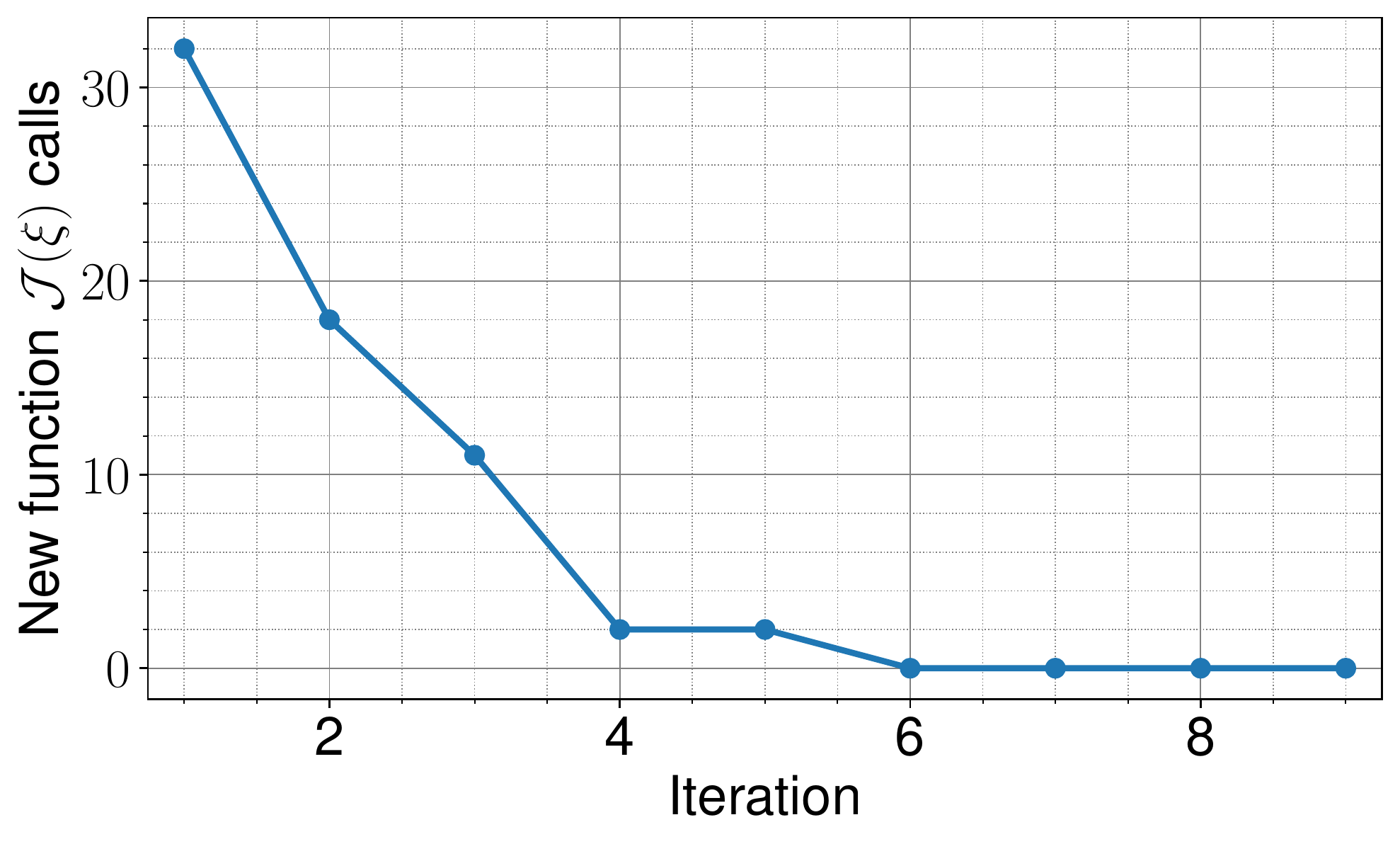}
          \quad 
          \includegraphics[width=0.32\linewidth]{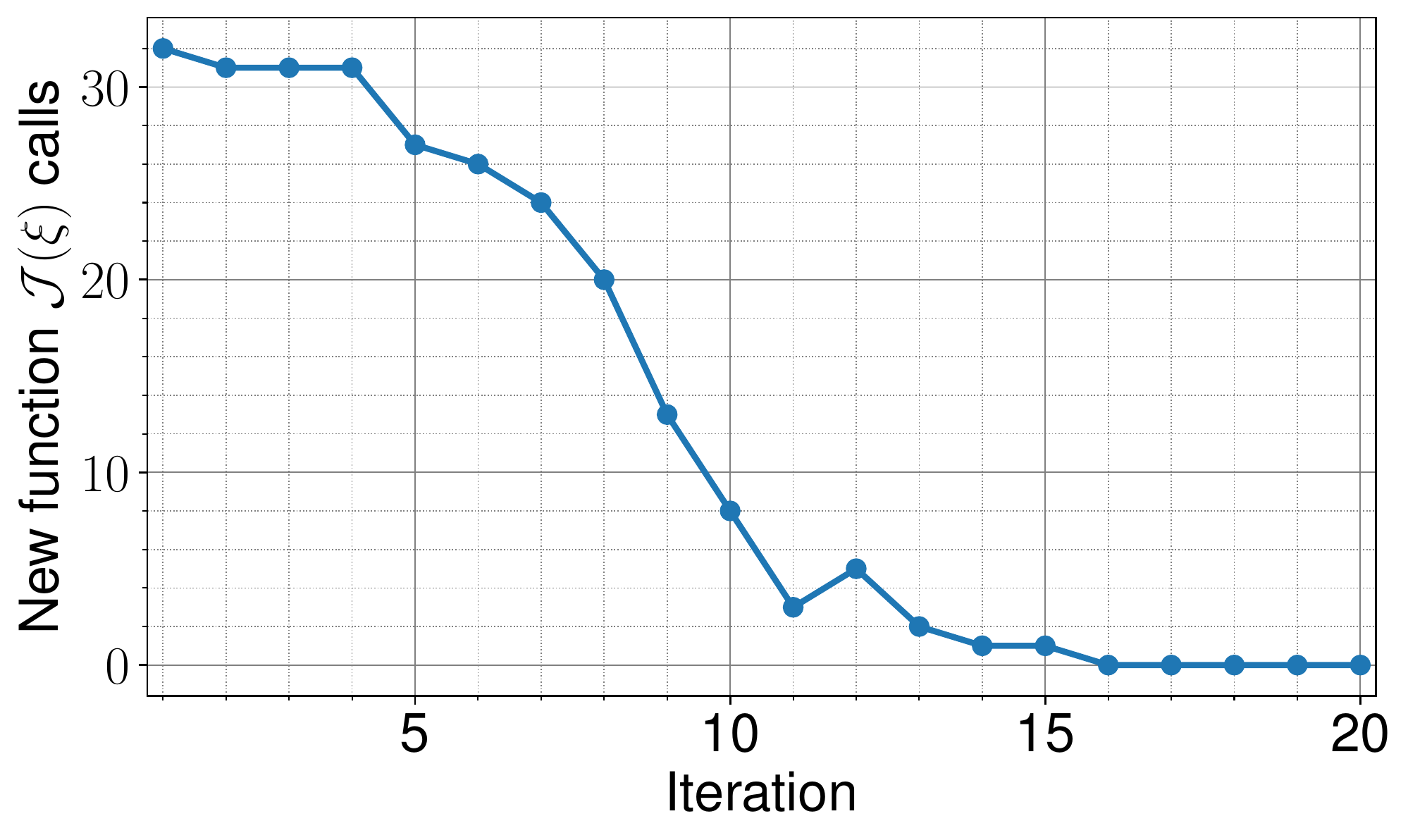}

          \caption{Similar to~\Cref{fig:AD_Set1_REINFORCE_NoPenalty_function_evaluations}.
            Here we set the sparsity penalty parameter to $\regpenalty=1.0$ and use
             the penalty constraint defined by
            $\Phi(\design):=\abs{\wnorm{\design}{0}-\budget}$, where $\budget=8$.
          }
          \label{fig:AD_Set1_REINFORCE_function_evaluations_Budget_8_Penalty_1.0}
        \end{figure}
        
        We note that the performance of both~\Cref{alg:REINFORCE},
        and~\Cref{alg:REINFORCE_baseline} is consistent with and without
        sparsity constraints. Moreover, by incorporating the optimal baseline
        estimate~\eqref{eqn:optimal_baseline} in~\Cref{alg:REINFORCE_baseline},
        at slight additional computational cost, the global optimum design is
        more likely to be discovered by the optimization algorithm.

      \subsubsection{Results with various learning rates}
      \label{subsubsec:AD_Results_LearningRate_Comparison}
        We conclude this section of experiments with results obtained by different
        learning rates.
        We use the setup in~\Cref{subsubsec:AD_Results_BudgetConstraint}; 
        that is, we assume an exact budget of $\budget=8$
        sensors and enforce it by setting the penalty function
        $\Phi(\design):=\wnorm{\design}{0}$ and the penalty parameter
        $\regpenalty=1$.
        \Cref{fig:AD_Set1_learning_rates} shows results obtained by varying the
        learning rate $\eta$ in the optimization algorithm.
        Specifically, we show results obtained from~\Cref{alg:REINFORCE_baseline}, with the
        optimal baseline estimate~\eqref{eqn:optimal_baseline} and note that 
        similar behavior was observed for the other settings used earlier in
        the paper.
        \begin{figure}[ht] \centering
          \includegraphics[width=0.45\linewidth]{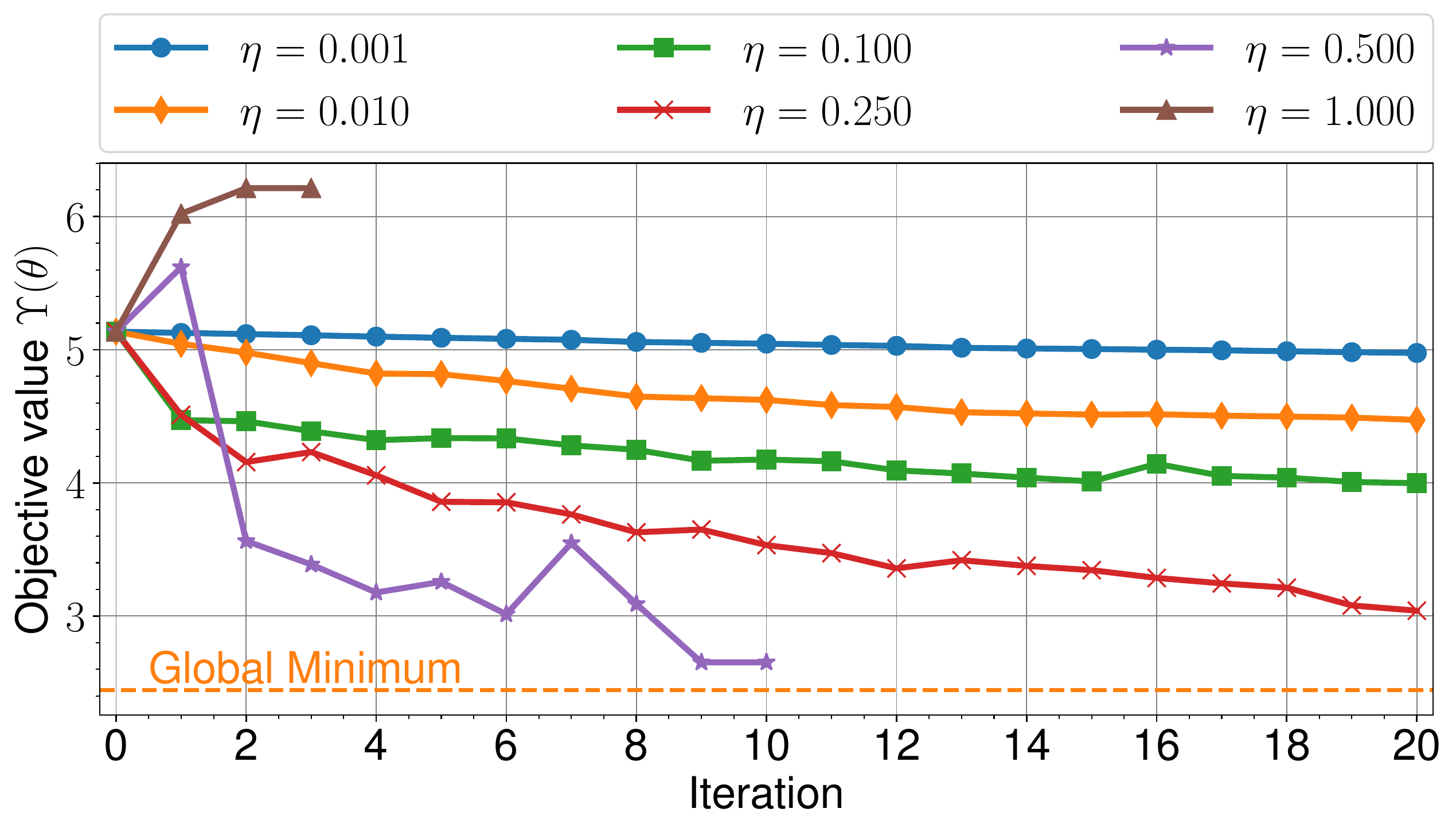}
          \quad
          \includegraphics[width=0.45\linewidth]{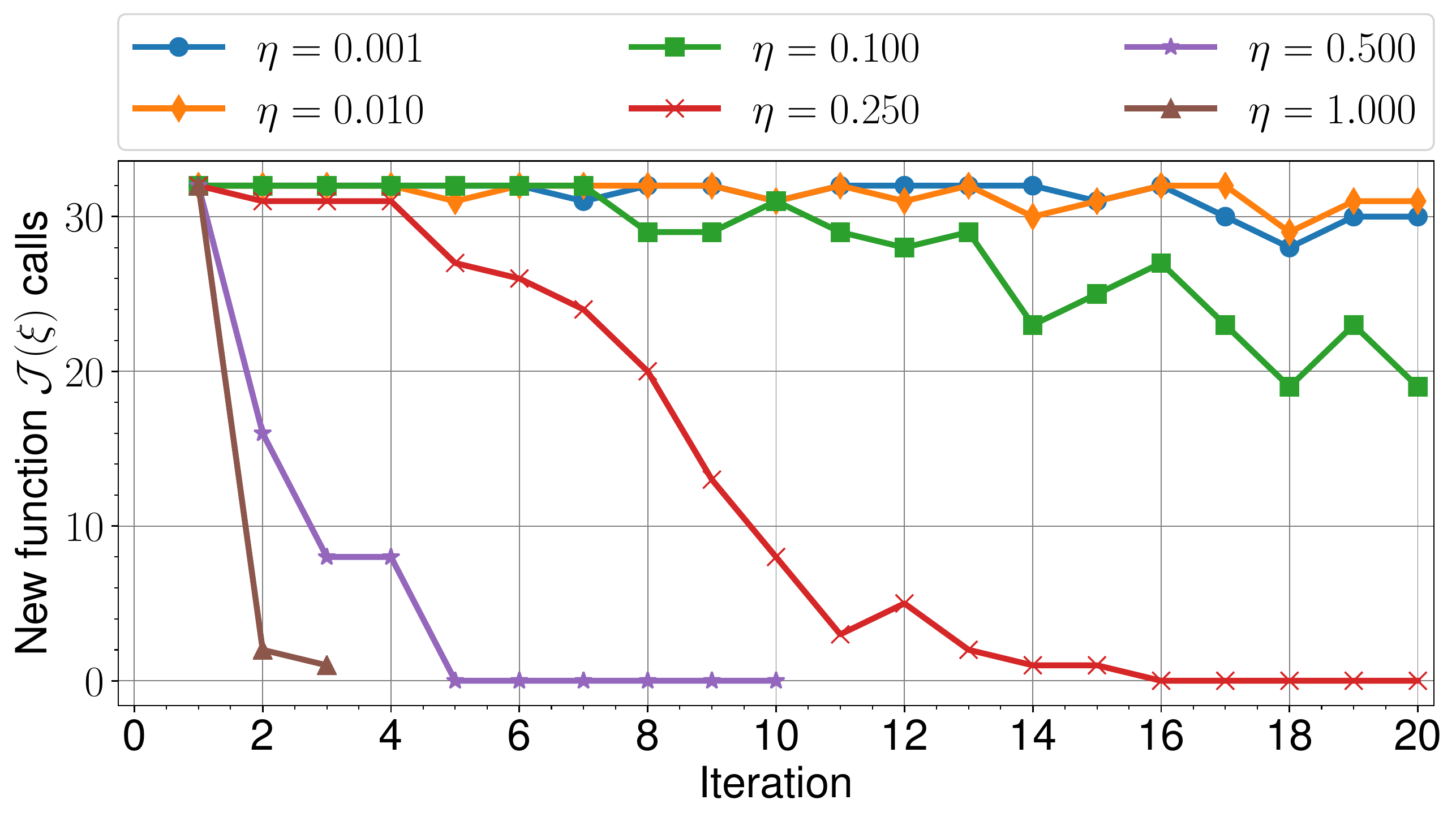}

          \caption{Similar to~\Cref{fig:AD_Set1_REINFORCE_NoPenalty_function_evaluations}.
            Here, we set the sparsity penalty parameter to $\regpenalty=1.0$ and use
            the budget constraint defined by
            $\Phi(\design):=\abs{\wnorm{\design}{0}-\budget}$, where $\budget=8$.
          }
          \label{fig:AD_Set1_learning_rates}
        \end{figure}
        
        \Cref{fig:AD_Set1_learning_rates} (left) shows the value of the
        stochastic objective function corresponding to the parameter
        $\hyperparam^{(k)}$ at the $k$th iteration of the optimization
        algorithm for various choices of the learning rate.
        \Cref{fig:AD_Set1_learning_rates} (right) shows the number of new
        calls to the function $\obj$ made at each iteration of the algorithm.
        We note that by increasing the learning rate $\eta$, the algorithm
        tends to converge quickly and explore the space of probability
        distributions near the global optimal policy very quickly. 
        However, this action is also associated with the risk of divergence.
       We note that $\eta=0.5$ is the best learning rate among the
        tested values. 
        
        In general, one can choose a small learning rate or even a decreasing
        sequence and run the optimization algorithm long enough 
        to guarantee convergence to an optimal policy. 
        Doing so, however, will likely increase the computational cost manifested 
        in the number of evaluations of $\obj$. 
        This problem is widely known as the \textit{exploration-exploitation trade-off} 
        in the reinforcement learning literature. 
        Finding an analytically optimal learning rate is beyond the scope of 
        this paper and will be explored in separate work.

\section{Discussion and Concluding Remarks}
  \label{sec:conclusions}
  In this work, we presented a new approach for the optimal design of experiments
  for Bayesian inverse problems constrained by expensive mathematical models,
  such as partial differential equations.
  The regularized utility function is cast 
  into a stochastic objective defined over the parameters of multivariate
  Bernoulli distribution.
  A policy gradient algorithm is used to optimize the new objective function 
  and thus yields an approximately optimal probability distribution, 
  that is, policy from which an approximately optimal design is sampled.
  The proposed approach does not require differentiability of the design 
  utility function nor the penalty function generally employed to enforce 
  sparsity or regularity conditions on the design. 
  Hence, the computational cost of the proposed methods, in terms of the number of
  forward model solves, is much less
  than the cost required by traditional gradient-based approach for optimal
  experimental design. The decrease in computational cost is due mostly  to the
  fact that the proposed method does not require evaluation of the simulation
  model for each entry of the gradient.
  Sparsity-enforcing penalty functions such as $\ell_0$ can be used directly,
  without the need to utilize a continuation procedure or apply a
  rounding technique.

  The main open issue
  pertains to the optimal selection of the learning rate parameter.
  While using a decreasing sequence satisfying the Robbins--Monro conditions
  guarantees convergence of the proposed algorithm almost surely, such a choice
  may require many iterations before convergence to a degenerate optimal policy.
  This issue will be addressed in detail in separate work.

  Note that the proposed stochastic formulation can be solved by other
  sample-based optimization algorithms, including sample average
  approximation~\cite{robbins1951stochastic,spall2005introduction,nemirovski2009robust}.
  The performance of these algorithms compared with that of the proposed algorithms will be also
  considered in separate works.

  We note that utilizing traditional cost-reduction methods, 
  including randomized matrix
  methods~\cite{AvronToledo11,saibaba2016randomized,saibaba2017randomized},
  and other reduced-order modeling approaches (see, e.g.,~\cite{spantini2015optimal,flath2011fast,cui2016scalable,attia2016reducedhmcsmoother}),
  to reduce the cost of the OED criterion $\obj$ 
  apply to both the relaxed approach and our proposed approach equally.
  This shows that the proposed algorithms introduce massive computational savings 
  to the OED solution process, compared with the traditional relaxation approach.

\section*{Acknowledgments}
This material is based upon work supported by the U.S. Department of Energy,
Office of Science, under contract number DE-AC02-06CH11357.

\appendix

\section{Multivariate Bernoulli Distribution}
\label{app:Bernoulli}
  The probabilities of a Bernoulli random variable $\design\in\{0, 1\}$ are described by 
  \begin{equation}\label{eqn:scalar_Bernoulli_prob}
    \CondProb{\design=v}{\hyperparam} :=
      \begin{cases}
        \hyperparam   &; \quad v=1 \,, \\
        1-\hyperparam &;  \quad v=0 \,,
      \end{cases}
  \end{equation}
  where $\hyperparam_i \in [0, 1]$ can be thought of as the probability of
  success in a one-trial experiment.
  The probability mass function (PMF) of this variable takes the compact form
  $\CondProb{\design}{\hyperparam} 
    = \hyperparam^{\design} \, (1-\hyperparam)^{(1-\design)}
  $. Moreover, the following identity holds:
  \begin{equation}\label{eqn:scalar_PMF_deriv}
    \del{\CondProb{\design}{\hyperparam}}{\hyperparam} = (-1) ^ {1-\design)} \,.
  \end{equation}
  %


  Assuming $\design_i,\,i=1,2,\ldots,\Nsens$ are mutually independent
  Bernoulli random variables with respective success probabilities
  $\hyperparam_i,\, i=1,2,\ldots,\Nsens$, then
  the joint probability mass function of the random variable
  $\design=(\design_1,\design_2,\ldots,\design_{\Nsens})\tran$, parameterized
  by $\hyperparam=(\hyperparam_1, \hyperparam_2,\ldots,\hyperparam_{\Nsens})\tran$,
  takes the form
  \begin{equation}\label{eqn:multivariate_Bernoulli}
    \CondProb{\design}{\hyperparam}
      = \prod_{i=1}^{\Nsens} {\hyperparam_i ^ {\design_i}\,
        \left(1-\hyperparam_i\right)^{1-\design_i}}
      \equiv 
        \prod_{i=1}^{\Nsens} 
        \Bigl(
          { \hyperparam_i {\design_i}\, +
            \left(1-\hyperparam_i\right) \left(1-\design_i \right) 
          }
        \Bigr) \,,
    \quad \design_i \in  \{0, 1\} \,.
  \end{equation}

  By using~\eqref{eqn:scalar_PMF_deriv}, the first-order derivative 
  of~\eqref{eqn:multivariate_Bernoulli} w.r.t the parameters $\hyperparam_i$
  is described by
  \begin{equation}\label{eqn:Bernoulli_PMF_derivative}
    \begin{aligned}
      \del{ \CondProb{\design}{\hyperparam} }{\hyperparam_j}
        &=\del{}{\hyperparam_j} 
          \left(
            \hyperparam_j^{\design_j}\,\left(1-\hyperparam_j \right)^{1-\design_j}
            \prod_{\substack{i=1\\ i\neq j}}^{\Nsens} {\hyperparam_i ^ {\design_i}
            \left(1-\hyperparam_i\right)^{1-\design_i}}
          \right)
        = (-1)^{1-\design_j}
          \prod_{\substack{i=1 \\ i\neq j}}^{\Nsens} 
            {\hyperparam_i ^ {\design_i}
            \left(1-\hyperparam_i\right)^{1-\design_i}}
        \,.
    \end{aligned}
  \end{equation}

  Thus, the gradient can be written as
  \begin{equation}
    \nabla_{\hyperparam}\, \CondProb{\design}{\hyperparam}
      = \sum_{j=1}^{\Nsens} \del{ \CondProb{\design}{\hyperparam} }{\hyperparam_j}
      = \sum_{j=1}^{\Nsens} (-1)^{1-\design_j}
        \prod_{\substack{i=1 \\ i\neq j}}^{\Nsens} {\hyperparam_i ^ {\design_i}
        \left(1-\hyperparam_i\right)^{1-\design_i}} \vec{e}_j \,.
  \end{equation}

  Note that the derivative given by~\eqref{eqn:Bernoulli_PMF_derivative} is the
  (signed) conditional probability of $\design$ conditioned by
  $\design_j$ and $\hyperparam_j$, respectively.
  The second-order derivatives follow as
  \begin{equation}\label{eqn:Bernoulli_PMF_Hessian}
    \begin{aligned}
      \delll{\CondProb{\design}{\hyperparam} }{\hyperparam_k}{\hyperparam_j}
        &= (1-\delta_{kj}) (-1)^{2-\design_j-\design_k}
          \prod_{\substack{i=1\\ i\notin \{j,k\}}}^{\Nsens} {\hyperparam_i ^ {\design_i}\,
          \left(1-\hyperparam_i\right)^{1-\design_i}}
        \,,
    \end{aligned}
  \end{equation}
  where $\delta_{kj}$ is the standard Kronecker delta function.
  The gradient of the log-probabilities, that is, the score
  function,
  of the multivariate Bernoulli PMF~\eqref{eqn:multivariate_Bernoulli}, 
  is given by
  \begin{equation}\label{eqn:grad_log_Bernoulli}
    \begin{aligned}
      \nabla_{\hyperparam} \log{ \CondProb{\design}{\hyperparam} }
        & = \nabla_{\hyperparam} \log{\prod_{i=1}^{\Nsens}
          {\hyperparam_i ^ {\design_i}\, \left(1-\hyperparam_i\right)^{1-\design_i}}}
        = \sum_{i=1}^{\Nsens} \nabla_{\hyperparam} \log{\hyperparam_i ^ {\design_i} }
          + \sum_{i=1}^{\Nsens} \nabla_{\hyperparam}
            \log{\left(1-\hyperparam_i\right)^{1-\design_i} }  \\
        & = \sum_{i=1}^{\Nsens} {\design_i} \nabla_{\hyperparam} \log{\hyperparam_i }
          + \sum_{i=1}^{\Nsens} (1-\design_i) \nabla_{\hyperparam}
            \log{\left(1-\hyperparam_i\right) }
         = \sum_{i=1}^{\Nsens} \left(
          \frac{\design_i}{\hyperparam_i} + \frac{\design_i-1}{1-\hyperparam_i}
        \right) \,\vec{e}_i  \,.
    \end{aligned}
  \end{equation}
  %


  It follows immediately from~\eqref{eqn:grad_log_Bernoulli} that
  \begin{equation}\label{eqn:Bernoulli_2nd_derivatives}
    \begin{aligned}
      \nabla_{\hyperparam} \nabla_{\hyperparam} \log{ \CondProb{\design}{\hyperparam} }
        & = \sum_{i=1}^{\Nsens} \left(
          \frac{-\design_i}{\hyperparam_i^2} - \frac{1-\design_i}{(1-\hyperparam_i)^2}
        \right) \,\vec{e}_i\vec{e}_i\tran \\
      \nabla_{\hyperparam} \log{ \CondProb{\design}{\hyperparam} }
        \left( \nabla_{\hyperparam} \log{ \CondProb{\design}{\hyperparam} } \right)\tran
        & = \sum_{i=1}^{\Nsens} \sum_{j=1}^{\Nsens}
          \left( \frac{\design_i}{\hyperparam_i}
            - \frac{1-\design_i}{1-\hyperparam_i} \right)
          \left( \frac{\design_j}{\hyperparam_j}
            - \frac{1-\design_j}{1-\hyperparam_j} \right) \,
          \vec{e}_i \vec{e}_j\tran .
    \end{aligned}
  \end{equation}

  In the rest of this Appendix, we prove some identities essential for
  convergence analysis of the algorithms proposed in this work.
  We start with the following basic relations.
  First note that
  $
  \Expect{}{\design_i\design_j} =
    \brCov{\design_i,\,\design_j}
    + \Expect{}{\design_i}\Expect{}{\design_j}
      = \delta_{ij} \hyperparam_i(1-\hyperparam_i) + \hyperparam_i \hyperparam_j$.
  This means that $\Expect{}{\design_i^2}=\Expect{}{\design_i}=\hyperparam_i$
  and $\Expect{}{\design_i\design_j}
    = \hyperparam_i\hyperparam_j\,\forall i\neq j$.
  Similarly,
  $
  \Expect{}{\design_i(\design_j-1)}
    = \Expect{}{\design_i \design_j} - \Expect{}{\design_i}
    = \delta_{ij} \hyperparam_i(1-\hyperparam_i)
      + \hyperparam_i \hyperparam_j - \hyperparam_i \,,
  $
  and
  $
  \Expect{}{(\design_i-1)(\design_j-1)}
    = \Expect{}{\design_i \design_j}
      - \hyperparam_i - \hyperparam_j
      + 1
    = \delta_{ij} \hyperparam_i(1-\hyperparam_i)
      + \hyperparam_i \hyperparam_j
      - \hyperparam_i - \hyperparam_j
      + 1
    \,.
  $
  We can summarize these identities as follows:
  \begin{equation}\label{eqn:Bernoulli_expect_identities}
    \begin{aligned}
      \Expect{}{\design_i\design_j}
        &= \begin{cases}
          \hyperparam_i,\,&\, i=j  \\
          \hyperparam_i\hyperparam_j ,\,&\, i\neq j
        \end{cases} \,,  \\
      \Expect{}{\design_i(\design_j\!-\!1)}
        &= \begin{cases}
          0,\,&\, i=j  \\
          \hyperparam_i(\hyperparam_j\!-\!1) ,\,&\, i\neq j
        \end{cases} \,,  \\
      \Expect{}{(\design_i\!-\!1)(\design_j\!-\!1)}
        &= \begin{cases}
          1\!-\!\hyperparam_i,\,&\, i=j  \\
          (1-\hyperparam_i) (1- \hyperparam_j) ,\,&\, i\neq j
        \end{cases} \,.
    \end{aligned}
  \end{equation}

  \begin{lemma}\label{lemma:grad_logvar_dist}
    Let $\design\in \Omega_{\design}:=\{0,1\}^{\Nsens}$ be a random variable
    following the joint Bernoulli distribution~\eqref{eqn:multivariate_Bernoulli},
    and assume that $\Var{(\design)}$ is the total variance operator that evaluates 
    the trace of the variance-covariance matrix of the random variable
    $\design$. Then the following identities hold:
    \begin{equation} \label{eqn:grad_logvar_dist}
      \Expect{}{\nabla_{\hyperparam} \log{ \CondProb{\design}{\hyperparam} }
      } = 0 \,;\qquad
      \brVar{\nabla_{\hyperparam} \log{
        \CondProb{\design}{\hyperparam} } }
        = \sum_{i=1}^{\Nsens}\frac{1}  {\hyperparam_i-\hyperparam_i^2} \,.
    \end{equation}
  \end{lemma}
  \begin{proof}
    The first identity follows as
    \begin{equation}
      \begin{aligned}
        \Expect{}{\nabla_{\hyperparam} \log{ \CondProb{\design}{\hyperparam}}}
          = \sum_{\design} \nabla_{\hyperparam}
            \log{ \CondProb{\design}{\hyperparam}} \CondProb{\design}{\hyperparam}
          = \sum_{\design} \nabla_{\hyperparam} \CondProb{\design}{\hyperparam}
          = \nabla_{\hyperparam} \sum_{\design} \CondProb{\design}{\hyperparam}
          = 0 \,.
      \end{aligned}
    \end{equation}

    By definition of the covariance matrix, and since
    $\Expect{}{\nabla_{\hyperparam} \log{
      \CondProb{\design}{\hyperparam} }}=0$, then 
    \begin{equation}
      \begin{aligned}
      \brVar{ \nabla_{\hyperparam} \log{ \CondProb{\design}{\hyperparam} } } 
        \:&= \Trace{\brCov{\nabla_{\hyperparam} \log{
          \CondProb{\design}{\hyperparam} }    ,
          \nabla_{\hyperparam} \log{
          \CondProb{\design}{\hyperparam} }  
        }} \\
        &= \Trace{ 
          \Expect{}{
            \left(
              \nabla_{\hyperparam} \log{ \CondProb{\design}{\hyperparam}}
            \right)
            \left(
              \nabla_{\hyperparam} \log{ \CondProb{\design}{\hyperparam}}
            \right)\tran
          }} \\ 
        &= \Expect{}{ 
          \Trace{
            \left(
              \nabla_{\hyperparam} \log{ \CondProb{\design}{\hyperparam}}
            \right)
            \left(
              \nabla_{\hyperparam} \log{ \CondProb{\design}{\hyperparam}}
            \right)\tran
          }} \\ 
        &= \Expect{}{
            \left(
              \nabla_{\hyperparam} \log{ \CondProb{\design}{\hyperparam}}
            \right) \tran
            \left(
              \nabla_{\hyperparam} \log{ \CondProb{\design}{\hyperparam}}
            \right)
          } \,,
      \end{aligned}
    \end{equation}
    where we utilized the circular property
    of the trace operator and the fact that the matrix trace is a linear 
    operator. Thus, 
    \begin{equation}
      \begin{aligned}
        \brVar{\nabla_{\hyperparam} \log{ \CondProb{\design}{\hyperparam} }}
          &= \Expect{}{
            \left(
              \nabla_{\hyperparam} \log{ \CondProb{\design}{\hyperparam}}
            \right)\tran
            \left(
              \nabla_{\hyperparam} \log{ \CondProb{\design}{\hyperparam}}
            \right)
          }  \\
          &= \Expect{}{
            \sum_{i=1}^{\Nsens}
              \left(
                \del{\log{ \CondProb{\design}{\hyperparam_i}}}{\hyperparam_i}
              \right)^2
            }
          = \Expect{}{
            \sum_{i=1}^{\Nsens}
              \left(
                \frac{\design_i}{\hyperparam_i}
                + \frac{\design_i-1}{1-\hyperparam_i}
              \right)^2
            }  \\
          &= \Expect{}{
            \sum_{i=1}^{\Nsens}
              \left(
                \frac{\design_i^2}{\hyperparam_i^2}
                + 2 \frac{\design_i^2-\design_i}{\hyperparam_i-\hyperparam_i^2}
                + \frac{\design_i^2-2\design_i+1}{(1-\hyperparam_i)^2}
              \right)
            }  \\
          &= \sum_{i=1}^{\Nsens}
              \left(
                \frac{\Expect{}{\design_i^2}}{\hyperparam_i^2}
                + \frac{\Expect{}{\design_i^2}-2\Expect{}{\design_i} + 1}
                  {(1-\hyperparam_i)^2}
              \right)
          = \sum_{i=1}^{\Nsens}
              \left(
                \frac{\hyperparam_i}{\hyperparam_i^2}
                + \frac{\hyperparam_i-2\hyperparam_i + 1}
                  {(1-\hyperparam_i)^2}
              \right)  \\
          &= \sum_{i=1}^{\Nsens}
              \left(
                \frac{1}{\hyperparam_i}
                + \frac{1-\hyperparam_i}
                  {(1-\hyperparam_i)^2}
              \right)
          = \sum_{i=1}^{\Nsens}
              \left(
                \frac{1}{\hyperparam_i}
                + \frac{1}
                  {1-\hyperparam_i}
              \right)
          = \sum_{i=1}^{\Nsens}\frac{1}  {\hyperparam_i-\hyperparam_i^2}
            \,,
      \end{aligned}
    \end{equation}
    where we used the fact that
    $ \Expect{}{\design_i^2}
      = \hyperparam_i
    $, as shown by~\eqref{eqn:Bernoulli_expect_identities}.
    This is also obvious since $\design^2=\design$. 
  \end{proof}

    \begin{lemma}\label{lemma:Bernoulli_PMF_gradient_norm_bound}
      Let $\design\in \Omega_{\design}:=\{0,1\}^{\Nsens}$ be a random variable
      following the joint Bernoulli distribution~\eqref{eqn:multivariate_Bernoulli}.
      Then for any $\design \in \Omega_{\design}$, the following bounds hold:
      \begin{align}
        \norm{ \nabla_{\hyperparam}\, \CondProb{\design}{\hyperparam} }
        & \leq
          \sqrt{\Nsens} \,
            \max\limits_{j=1,\ldots,\Nsens}
              \, \min\limits_{\substack{i=1,\ldots,\Nsens\\k\neq j}}
            \CondProb{\design_i}{\hyperparam_i}  
            \label{eqn:Bernoulli_PMF_gradient_norm_bound}
            \\
        \Expect{\design}{
          \sqnorm{ \nabla_{\hyperparam} \log{ \CondProb{\design}{\hyperparam} } }}
        \equiv \brVar{\nabla_{\hyperparam}{\log{\CondProb{\design}{\hyperparam}}}} 
        &\leq   \frac{ \Nsens}{\min\limits_i \hyperparam}
            + \frac{ \Nsens}{1-\max\limits_i \hyperparam} \,.
            \label{eqn:grad_log_Bernoulli_norm_bound}
      \end{align}
    \end{lemma}
    \begin{proof}
      \begin{equation}
        \begin{aligned}
          \sqnorm{ \nabla_{\hyperparam}\, \CondProb{\design}{\hyperparam} }
            &= \sqnorm{
              \sum_{j=1}^{\Nsens} (-1)^{1-\design_j}
              \prod_{\substack{i=1 \\ k\neq j}}^{\Nsens} {\hyperparam_i ^ {\design_i}
              \left(1-\hyperparam_i\right)^{1-\design_i}}
              \vec{e}_j
              } \\
            &\leq
            \sum_{j=1}^{\Nsens} \left| (-1)^{1-\design_j} \right|
            \left(
              \prod_{\substack{i=1 \\ k\neq j}}^{\Nsens} {\hyperparam_i ^ {\design_i}
              \left(1-\hyperparam_i\right)^{1-\design_i}} \right)^2  \\
            &\leq
              \sum_{j=1}^{\Nsens} \, \min\limits_{\substack{i=1,\ldots,\Nsens\\k\neq j}}
              \left( \CondProb{\design_i}{\hyperparam_i}\right)^2  
            \leq
              \Nsens \, \max\limits_{j=1,\ldots,\Nsens}
              \min\limits_{\substack{i=1,\ldots,\Nsens\\k\neq j}}
              \left(  \CondProb{\design_i}{\hyperparam_i} \right)^2  \,,
        \end{aligned}
      \end{equation}
      which prove the first 
      inequality~\eqref{eqn:Bernoulli_PMF_gradient_norm_bound}. 
      By utilizing~\eqref{eqn:grad_log_Bernoulli}, we have 
      \begin{equation}
        \begin{aligned}
          \sqnorm{ \nabla_{\hyperparam} \log{ \CondProb{\design}{\hyperparam} } }
          &=
            \sqnorm{\sum_{i=1}^{\Nsens}
            \left(
              \frac{\design_i}{\hyperparam_i} + \frac{\design_i-1}{1-\hyperparam_i}
            \right) \,\vec{e}_i
          }
          =
            \sum_{i=1}^{\Nsens}
            \left(
              \frac{\design_i}{\hyperparam_i} + \frac{\design_i-1}{1-\hyperparam_i}
            \right)^2
          \\
          &=
            \sum_{i=1}^{\Nsens} \left(
              \frac{\design_i^2}{\hyperparam_i^2}
            + 2
              \frac{\design_i(\design_i-1)}{\hyperparam_i(1-\hyperparam_i)}
            + \frac{\left(\design_i-1\right)^2}{\left(1-\hyperparam_i\right)^2}
            \right)
          =
            \sum_{i=1}^{\Nsens} \left(
              \frac{\design_i}{\hyperparam_i^2}
            + \frac{\design_i-1}{\left(1-\hyperparam_i\right)^2}
            \right)
            \,.
        \end{aligned}
      \end{equation}
      where the last relation follows given the fact
      that $\design_i\in\{0, 1\}$, and hence
      $\design_i^2=\design_i\,,\forall\, i=1, 2, \ldots,\Nsens$.
      Taking the expectation of both sides, we get
      \begin{equation}
        \begin{aligned}
          \Expect{\design}{
            \sqnorm{ \nabla_{\hyperparam} \log{ \CondProb{\design}{\hyperparam} } }
          }
          &= \Expect{\design}{
            \sum_{i=1}^{\Nsens} \left( 
            \frac{\design_i}{\hyperparam_i^2}
            + \frac{\design_i-1}{\left(1-\hyperparam_i\right)^2}
            \right)
          }
          =
            \sum_{i=1}^{\Nsens} \left(
            \frac{\Expect{}{\design_i }}{\hyperparam_i^2}
            + \frac{\Expect{}{\design_i} -1}{\left(1-\hyperparam_i\right)^2}
            \right)
          \\
          &=  \sum_{i=1}^{\Nsens}\left( \frac{\hyperparam_i}{\hyperparam_i^2}
            + \frac{\hyperparam_i -1}{\left(1-\hyperparam_i\right)^2}
            \right)
          =  \sum_{i=1}^{\Nsens} \left( 
            \frac{1}{\hyperparam_i}
            + \frac{1}{\left(1-\hyperparam_i\right)}
            \right)
          \\
          &=  \sum_{i=1}^{\Nsens} \frac{1}{\hyperparam_i}
            + \sum_{i=1}^{\Nsens} \frac{1}{\left(1-\hyperparam_i\right)}
          \leq  \frac{\Nsens}{\min\limits_i \hyperparam}
            + \frac{\Nsens}{1-\max\limits_i \hyperparam}
          \,,
        \end{aligned}
      \end{equation}
      which completes the proof of~\eqref{eqn:grad_log_Bernoulli_norm_bound}.
      $\quad$
    \end{proof}



\bibliographystyle{siamplain}
\bibliography{Bib/optimization_references,Bib/oed_references,Bib/ahmed_citations,Bib/data_assimilation}


\iftrue
\null \vfill
  \begin{flushright}
  \scriptsize \framebox{\parbox{3.2in}{
  The submitted manuscript has been created by UChicago Argonne, LLC,
  Operator of Argonne National Laboratory (``Argonne"). Argonne, a
  U.S. Department of Energy Office of Science laboratory, is operated
  under Contract No. DE-AC02-06CH11357. The U.S. Government retains for
  itself, and others acting on its behalf, a paid-up nonexclusive,
  irrevocable worldwide license in said article to reproduce, prepare
  derivative works, distribute copies to the public, and perform
  publicly and display publicly, by or on behalf of the Government.
  The Department of
  Energy will provide public access to these results of federally sponsored research in accordance
  with the DOE Public Access Plan. http://energy.gov/downloads/doe-public-access-plan. }}
  \normalsize
  \end{flushright}
\fi


\end{document}

%% file: shared.tex
\usepackage{amsfonts}
\usepackage{amsopn}
\usepackage{amsmath}
\usepackage{amssymb}
\usepackage{multirow}
\usepackage{tabularx}
\usepackage{hhline}
\usepackage{scalerel,stackengine}
\usepackage{arydshln}
\usepackage[normalem]{ulem}
\usepackage{soul}
\usepackage{tikz}
\usepackage{pgf}
\usepackage{collcell}
\usepackage[utf8]{inputenc}
\usepackage{geometry}

\usepackage{graphbox}
\usepackage{setspace}

\newcolumntype{E}{>{\collectcell\usermacro}c<{\endcollectcell}}
\newcommand\usermacro[1]{\pgfmathparse{10000*#1}\pgfmathprintnumber\pgfmathresult}

\usepackage{algorithmicx}
\usepackage{algorithm}
\usepackage{algpseudocode}

\makeatletter
\patchcmd{\ALG@step}{\addtocounter{ALG@line}{1}}{\refstepcounter{ALG@line}}{}{}
\newcommand{\ALG@lineautorefname}{Step}
\makeatother

\usepackage{cite}
\usepackage[shortlabels]{enumitem}

\input{notation}

\ifpdf
  \DeclareGraphicsExtensions{.eps,.pdf,.png,.jpg}
\else
  \DeclareGraphicsExtensions{.eps}
\fi

\newsiamremark{remark}{Remark}
\newsiamremark{hypothesis}{Hypothesis}
\crefname{hypothesis}{Hypothesis}{Hypotheses}
\newsiamthm{claim}{Claim}
\crefname{lemma}{Lemma}{Lemmas}

\patchcmd{\SetTagPlusEndMark}{$}{}{}{}
\patchcmd{\SetTagPlusEndMark}{$}{}{}{}

\headers{Stochastic Learning Approach for Binary Optimal Design of Experiments}{A. Attia, S. Leyffer,  and  T.  Munson}

\title{Stochastic Learning Approach to Binary Optimization for Optimal Design of Experiments\thanks{Submitted to the editors \today.
\funding{
  This material is based upon work supported by the U.S. Department of Energy,
Office of Science, under contract number DE-AC02-06CH11357.}}
}

\author{Ahmed Attia\thanks{Mathematics and Computer Science Division,
                   Argonne National Laboratory, USA
                  ( \email{attia@mcs.anl.gov} ).}
\and Sven Leyffer\thanks{Mathematics and Computer Science Division,
                   Argonne National Laboratory, USA
                   ( \email{leyffer@mcs.anl.gov} ).}
\and Todd Munson\thanks{Mathematics and Computer Science Division,
                   Argonne National Laboratory, USA
                   ( \email{tmunson@mcs.anl.gov} ).}
}



%% file: notation.tex
\DeclareMathAlphabet{\mathup}{OT1}{\familydefault}{m}{n}

\DeclareSymbolFont{yhlargesymbols}{OMX}{yhex}{m}{n}
\DeclareMathAccent{\wideparen}{\mathord}{yhlargesymbols}{"F3}
\xdefinecolor{green}{rgb}{0.04, 0.85, 0.32}
\xdefinecolor{darkgreen}{rgb}{0.24, 0.7, 0.44}
\xdefinecolor{mint}{rgb}{0.24, 0.71, 0.54}
\xdefinecolor{officegreen}{rgb}{0.0, 0.5, 0.0}
\xdefinecolor{napiergreen}{rgb}{0.16, 0.5, 0.0}
\xdefinecolor{maroon}{rgb}{0.65,0.06,0.17}
\xdefinecolor{blue}{rgb}{0,0.2,0.6}
\xdefinecolor{phthalogreen}{rgb}{0.07,0.21,0.14}
\definecolor{grassgreen}{RGB}{92,135,39}

\crefname{algocf}{alg.}{algs.}
\Crefname{algocf}{Algorithm}{Algorithms}

\newcommand{\Oh}[1]{\mathcal{O}{\left(#1\right)}}
\newcommand{\logdet}[1]{\log\det\left(#1\right)}                    
\newcommand{\diag}[1]{\mathsf{diag}\left(#1\right)}                 
\newcommand{\del}[2]{\frac{\partial{#1}}{\partial{#2}}}             
\newcommand{\delll}[3]{\frac{\partial^2{#1}}{\partial{#2}\,\partial{#3}}} 
\newcommand{\mat}[1]{\mathbf{{#1}}}                                 
\renewcommand{\vec}[1]{\mathbf{{#1}}}                                 

\DeclareMathOperator*{\argmin}{arg\,min}  

\newcommand{\proj}{\mathit{P}}       
\newcommand{\Proj}[2]{\proj_{#1}{\left(#2\right)}}       

\newcommand{\regpenalty}{\alpha}                    
\newcommand*{\opt}{^{\mkern-1.5mu\mathrm{opt}}}               
\newcommand*{\tran}{^{\mkern-1.5mu\mathsf{T}}}                
\newcommand*{\adj}{^{\mkern-1.5mu\mathsf{*}}}                 
\newcommand*{\inv}{^{\mkern-1.5mu\mathsf{-1}}}                
\newcommand{\domain}{\mathcal{D}}                             
\newcommand{\trace}{\mathrm{Tr}}                              
\newcommand{\Trace}[1]{\trace \left(#1\right)}                
\newcommand{\abs}[1]{\left| {#1} \right|}                  
\newcommand{\norm}[1]{\left\| {#1} \right\|}                  
\newcommand{\sqnorm}[1]{\left\| {#1} \right\|^2}              
\newcommand{\wnorm}[2]{\left\| {#1} \right\|_{#2}}            
\newcommand{\sqwnorm}[2]{\left\| {#1} \right\|^2_{#2}}        

\newcommand\restr[2]{{ \left.\kern-\nulldelimiterspace        
                     {#1}\vphantom{\big|} \right|_{#2}}}

\newcommand{\Rnum}{\mathbb{R}}  
\newcommand{\xcont}{u}                             

\newcommand{\y}{\mathbf{y}}                        
\newcommand{\obs}{\y}                              

\newcommand{\param}{\vec{\theta}}                  
\newcommand{\iparam}{\param}                       
\newcommand{\iparprior}{\iparam_{\rm pr}}          
\newcommand{\iparb}{\iparprior}                    
\newcommand{\iparpost}{\iparam_{\rm post}^\obs}    
\newcommand{\ipara}{\iparpost}                     

\newcommand{\Nstate}{\textsc{N}_{\rm state}}                    

\newcommand{\Nobs}{\textsc{N}_{\rm obs}}                        
\newcommand{\Nens}{\textsc{N}_{\rm ens}}                        
\newcommand{\nobs}{n_{t}}                                  
\newcommand{\nobstimes}{\nobs}                             
\newcommand{\Nsens}{n_{\rm s}}                         
\newcommand{\budget}{\lambda}                         
\newcommand{\Cparamprior}{\mat\Gamma_{{\rm pr}}}                
\newcommand{\Cparampost}{\mat\Gamma_{{\rm post}}}               
\newcommand{\Cobsnoise}{\mat{\Gamma}_{ {\rm noise}}}            
\newcommand{\Cparampriormat}{\Cparamprior}                      
\newcommand{\Cparampostmat}{\Cparampost}                        
\newcommand{\wCparampostmat}{\Cparampostmat(\design)}           

\newcommand{\F}{\mathbf{F}}                                      

\newcommand{\obj}{\mathcal{J}}            
\newcommand{\stochobj}{\Upsilon}          
\newcommand{\optcriterion}{\Psi}

\newcommand{\Prob}{\mathbb{P}}                                   
\newcommand{\CondProb}[2]{\mathbb{P}\left(#1|#2 \right)}  
\newcommand{\Var}{\mathsf{var}}                                  
\newcommand{\Cov}{\mathsf{cov}}                                  
\newcommand{\brCov}[1]{\Cov{\Bigl(#1\Bigr)}}                                  
\newcommand{\brVar}[1]{\Var\Bigl(#1\Bigr)}                         
\newcommand{\GM}[2]{\mathcal{N}\!\left( {#1}, {#2}\right)}       

\newcommand{\like}{\mathcal{L}}                          
\newcommand{\Like}[2]{\like{\left(#1|#2\right)}}                                  

\newcommand{\Expect}[2]{\mathbb{E}_{#1}{\Bigl[ #2 \Bigr]} }     
\newcommand{\expect}[2]{\mathbb{E}_{#1}{#2} }     
\newcommand{\baseline}{b}
\newcommand{\hyperparam}{\vec{\theta}}     
\newcommand{\design}{\vec{\zeta}}                                       
\newcommand{\designmat}{\mat{W}}                                    
\newcommand{\wdesignmat}{\designmat_{\Gamma}}                       

\stackMath
\newcommand\reallywidehat[1]{%
\savestack{\tmpbox}{\stretchto{%
  \scaleto{%
    \scalerel*[\widthof{\ensuremath{#1}}]{\kern-.6pt\bigwedge\kern-.6pt}%
    {\rule[-\textheight/2]{1ex}{\textheight}}
  }{\textheight}%
}{0.5ex}}%
\stackon[1pt]{#1}{\tmpbox}%
}
\newcommand{\commentout}[1]{\iffalse {#1} \fi}
\newcommand{\todo}[1]{\noindent\textcolor{cyan}{TODO: #1\,}}

%% file: main.bbl
\begin{thebibliography}{10}

\bibitem{alexanderian2020optimal}
{\sc A.~Alexanderian}, {\em Optimal experimental design for bayesian inverse
  problems governed by pdes: A review}, arXiv preprint arXiv:2005.12998,
  (2020).

\bibitem{alexanderian2016bayesian}
{\sc A.~Alexanderian, P.~J. Gloor, O.~Ghattas, et~al.}, {\em On {B}ayesian
  {A}-and {D}-optimal experimental designs in infinite dimensions}, {B}ayesian
  Analysis, 11 (2016), pp.~671--695.

\bibitem{AlexanderianPetraStadlerEtAl14}
{\sc A.~Alexanderian, N.~Petra, G.~Stadler, and O.~Ghattas}, {\em A-optimal
  design of experiments for infinite-dimensional {B}ayesian linear inverse
  problems with regularized $\ell_0$-sparsification}, SIAM Journal on
  Scientific Computing, 36 (2014), pp.~A2122--A2148,
  \url{https://doi.org/10.1137/130933381}.

\bibitem{AlexanderianPetraStadlerEtAl16}
{\sc A.~Alexanderian, N.~Petra, G.~Stadler, and O.~Ghattas}, {\em A fast and
  scalable method for {A}-optimal design of experiments for
  infinite-dimensional {B}ayesian nonlinear inverse problems}, SIAM Journal on
  Scientific Computing, 38 (2016), pp.~A243--A272,
  \url{https://doi.org/10.1137/140992564},
  \url{http://dx.doi.org/10.1137/140992564}.

\bibitem{AlexanderianSaibaba17}
{\sc A.~Alexanderian and A.~K. Saibaba}, {\em Efficient {D}-optimal design of
  experiments for infinite-dimensional {B}ayesian linear inverse problems},
  Submitted,  (2017), \url{https://arxiv.org/abs/1711.05878}.

\bibitem{arouna2004adaptative}
{\sc B.~Arouna}, {\em Adaptative monte carlo method, a variance reduction
  technique}, Monte Carlo Methods and Applications, 10 (2004), pp.~1--24.

\bibitem{attia2020DOERL}
{\sc A.~Attia}, {\em {DOERL:} design of experiments using reinforcement
  learning}, 2020, \url{https://gitlab.com/ahmedattia/doerl}.

\bibitem{attia2018goal}
{\sc A.~Attia, A.~Alexanderian, and A.~K. Saibaba}, {\em Goal-oriented optimal
  design of experiments for large-scale {B}ayesian linear inverse problems},
  Inverse Problems, 34 (2018), p.~095009,
  \url{http://stacks.iop.org/0266-5611/34/i=9/a=095009}.

\bibitem{attia2020optimal}
{\sc A.~Attia and E.~Constantinescu}, {\em Optimal experimental design for
  inverse problems in the presence of observation correlations}, arXiv preprint
  arXiv:2007.14476,  (2020).

\bibitem{attia2015hmcfilter}
{\sc A.~Attia and A.~Sandu}, {\em A hybrid {M}onte {C}arlo sampling filter for
  non-{Gaussian} data assimilation}, AIMS Geosciences, 1 (2015), pp.~4--1--78,
  \url{https://doi.org/http://dx.doi.org/10.3934/geosci.2015.1.41},
  \url{http://www.aimspress.com/geosciences/article/574.html}.

\bibitem{attia2016reducedhmcsmoother}
{\sc A.~Attia, R.~Stefanescu, and A.~Sandu}, {\em The reduced-order hybrid
  {Monte Carlo} sampling smoother}, International Journal for Numerical Methods
  in Fluids,  (2016), \url{https://doi.org/10.1002/fld.4255},
  \url{http://dx.doi.org/10.1002/fld.4255}.
\newblock fld.4255.

\bibitem{AvronToledo11}
{\sc H.~Avron and S.~Toledo}, {\em Randomized algorithms for estimating the
  trace of an implicit symmetric positive semi-definite matrix}, Journal of the
  ACM (JACM), 58 (2011), p.~17, \url{https://doi.org/10.1145/1944345.1944349}.

\bibitem{bannister2017review}
{\sc R.~Bannister}, {\em A review of operational methods of variational and
  ensemble-variational data assimilation}, Quarterly Journal of the Royal
  Meteorological Society, 143 (2017), pp.~607--633.

\bibitem{BertsekasTsitsiklis96}
{\sc D.~P. Bertsekas and J.~Tsitsiklis}, {\em Neuro--Dynamic Programming},
  Athena Scientific, Belmont, Massachusetts, 1996.

\bibitem{bertsekas1996neuro}
{\sc D.~P. Bertsekas and J.~N. Tsitsiklis}, {\em Neuro-dynamic programming},
  Athena Scientific, 1996.

\bibitem{bui2013computational}
{\sc T.~Bui-Thanh, O.~Ghattas, J.~Martin, and G.~Stadler}, {\em A computational
  framework for infinite-dimensional {B}ayesian inverse problems part i: The
  linearized case, with application to global seismic inversion}, SIAM Journal
  on Scientific Computing, 35 (2013), pp.~A2494--A2523.

\bibitem{Bui-ThanhGhattasMartinEtAl13}
{\sc T.~Bui-Thanh, O.~Ghattas, J.~Martin, and G.~Stadler}, {\em A computational
  framework for infinite-dimensional {B}ayesian inverse problems {P}art {I}:
  {T}he linearized case, with application to global seismic inversion}, SIAM
  Journal on Scientific Computing, 35 (2013), pp.~A2494--A2523,
  \url{https://doi.org/10.1137/12089586X}.

\bibitem{cui2016scalable}
{\sc T.~Cui, Y.~Marzouk, and K.~Willcox}, {\em Scalable posterior
  approximations for large-scale bayesian inverse problems via
  likelihood-informed parameter and state reduction}, Journal of Computational
  Physics, 315 (2016), pp.~363--387.

\bibitem{daley1993atmospheric}
{\sc R.~Daley}, {\em Atmospheric data analysis}, Cambridge University Press,
  1991.

\bibitem{defazio2014saga}
{\sc A.~Defazio, F.~Bach, and S.~Lacoste-Julien}, {\em {SAGA}: A fast
  incremental gradient method with support for non-strongly convex composite
  objectives}, in Advances in neural information processing systems, 2014,
  pp.~1646--1654.

\bibitem{dupacova1988asymptotic}
{\sc J.~Dupacov{\'a} and R.~Wets}, {\em Asymptotic behavior of statistical
  estimators and of optimal solutions of stochastic optimization problems}, The
  Annals of Statistics,  (1988), pp.~1517--1549.

\bibitem{FedorovLee00}
{\sc V.~Fedorov and J.~Lee}, {\em Design of experiments in statistics}, in
  Handbook of semidefinite programming, R.~S. H.~Wolkowicz and L.~Vandenberghe,
  eds., vol.~27 of Internat. Ser. Oper. Res. Management Sci., Kluwer Acad.
  Publ., Boston, MA, 2000, pp.~511--532.

\bibitem{flath2011fast}
{\sc H.~P. Flath, L.~C. Wilcox, V.~Ak{\c{c}}elik, J.~Hill, B.~van
  Bloemen~Waanders, and O.~Ghattas}, {\em Fast algorithms for bayesian
  uncertainty quantification in large-scale linear inverse problems based on
  low-rank partial hessian approximations}, SIAM Journal on Scientific
  Computing, 33 (2011), pp.~407--432.

\bibitem{HaberHoreshTenorio08}
{\sc E.~Haber, L.~Horesh, and L.~Tenorio}, {\em Numerical methods for
  experimental design of large-scale linear ill-posed inverse problems},
  Inverse Problems, 24 (2008), pp.~125--137.

\bibitem{HaberHoreshTenorio10}
{\sc E.~Haber, L.~Horesh, and L.~Tenorio}, {\em Numerical methods for the
  design of large-scale nonlinear discrete ill-posed inverse problems}, Inverse
  Problems, 26 (2010), p.~025002,
  \url{http://stacks.iop.org/0266-5611/26/i=2/a=025002}.

\bibitem{HaberMagnantLuceroEtAl12}
{\sc E.~Haber, Z.~Magnant, C.~Lucero, and L.~Tenorio}, {\em Numerical methods
  for {A}-optimal designs with a sparsity constraint for ill-posed inverse
  problems}, Computational Optimization and Applications,  (2012), pp.~1--22.

\bibitem{huan2014gradient}
{\sc X.~Huan and Y.~Marzouk}, {\em Gradient-based stochastic optimization
  methods in bayesian experimental design}, International Journal for
  Uncertainty Quantification, 4 (2014).

\bibitem{HuanMarzouk13}
{\sc X.~Huan and Y.~M. Marzouk}, {\em Simulation-based optimal {B}ayesian
  experimental design for nonlinear systems}, Journal of Computational Physics,
  232 (2013), pp.~288--317,
  \url{https://doi.org/http://dx.doi.org/10.1016/j.jcp.2012.08.013},
  \url{http://www.sciencedirect.com/science/article/pii/S0021999112004597}.

\bibitem{IsaacPetraStadlerEtAl15}
{\sc T.~Isaac, N.~Petra, G.~Stadler, and O.~Ghattas}, {\em Scalable and
  efficient algorithms for the propagation of uncertainty from data through
  inference to prediction for large-scale problems, with application to flow of
  the {A}ntarctic ice sheet}, Journal of Computational Physics, 296 (2015),
  pp.~348--368, \url{https://doi.org/10.1016/j.jcp.2015.04.047}.

\bibitem{king1993asymptotic}
{\sc A.~J. King and R.~T. Rockafellar}, {\em Asymptotic theory for solutions in
  statistical estimation and stochastic programming}, Mathematics of Operations
  Research, 18 (1993), pp.~148--162.

\bibitem{kleywegt2002sample}
{\sc A.~J. Kleywegt, A.~Shapiro, and T.~Homem-de Mello}, {\em The sample
  average approximation method for stochastic discrete optimization}, SIAM
  Journal on Optimization, 12 (2002), pp.~479--502.

\bibitem{koval2020optimal}
{\sc K.~Koval, A.~Alexanderian, and G.~Stadler}, {\em Optimal experimental
  design under irreducible uncertainty for linear inverse problems governed by
  pdes}, Inverse Problems,  (2020).

\bibitem{l1994efficiency}
{\sc P.~L'Ecuyer}, {\em Efficiency improvement and variance reduction}, in
  Proceedings of Winter Simulation Conference, IEEE, 1994, pp.~122--132.

\bibitem{mak1999monte}
{\sc W.-K. Mak, D.~P. Morton, and R.~K. Wood}, {\em {Monte Carlo} bounding
  techniques for determining solution quality in stochastic programs},
  Operations Research Letters, 24 (1999), pp.~47--56.

\bibitem{navon2009data}
{\sc I.~M. Navon}, {\em Data assimilation for numerical weather prediction: a
  review}, in Data assimilation for atmospheric, oceanic and hydrologic
  applications, Springer, 2009, pp.~21--65.

\bibitem{nemirovski2009robust}
{\sc A.~Nemirovski, A.~Juditsky, G.~Lan, and A.~Shapiro}, {\em Robust
  stochastic approximation approach to stochastic programming}, SIAM Journal on
  Optimization, 19 (2009), pp.~1574--1609.

\bibitem{nocedal2006numerical}
{\sc J.~Nocedal and S.~Wright}, {\em Numerical optimization}, Springer Science
  \& Business Media, 2006.

\bibitem{Pazman86}
{\sc A.~P{\'a}zman}, {\em Foundations of optimum experimental design}, D.
  Reidel Publishing Co., 1986.

\bibitem{PetraStadler11}
{\sc N.~Petra and G.~Stadler}, {\em Model variational inverse problems governed
  by partial differential equations}, Tech. Report 11-05, The Institute for
  Computational Engineering and Sciences, The University of Texas at Austin,
  2011.

\bibitem{Pukelsheim93}
{\sc F.~Pukelsheim}, {\em Optimal design of experiments}, John Wiley \& Sons,
  New-York, 1993.

\bibitem{reddi2016fast}
{\sc S.~J. Reddi, S.~Sra, B.~P{\'o}czos, and A.~Smola}, {\em Fast incremental
  method for nonconvex optimization}, arXiv preprint arXiv:1603.06159,  (2016).

\bibitem{robbins1951stochastic}
{\sc H.~Robbins and S.~Monro}, {\em A stochastic approximation method}, The
  Annals of Mathematical Statistics,  (1951), pp.~400--407.

\bibitem{saibaba2016randomized}
{\sc A.~K. Saibaba, A.~Alexanderian, and I.~C. Ipsen}, {\em Randomized
  matrix-free trace and log-determinant estimators}, Numerische Mathematik, 137
  (2017), pp.~353--395.

\bibitem{saibaba2017randomized}
{\sc A.~K. Saibaba, A.~Alexanderian, and I.~C. Ipsen}, {\em Randomized
  matrix-free trace and log-determinant estimators}, Numerische Mathematik, 137
  (2017), pp.~353--395.

\bibitem{schmidt2017minimizing}
{\sc M.~Schmidt, N.~Le~Roux, and F.~Bach}, {\em Minimizing finite sums with the
  stochastic average gradient}, Mathematical Programming, 162 (2017),
  pp.~83--112.

\bibitem{shapiro1991asymptotic}
{\sc A.~Shapiro}, {\em Asymptotic analysis of stochastic programs}, Annals of
  Operations Research, 30 (1991), pp.~169--186.

\bibitem{spall2005introduction}
{\sc J.~C. Spall}, {\em Introduction to stochastic search and optimization:
  estimation, simulation, and control}, vol.~65, John Wiley \& Sons, 2005.

\bibitem{spantini2015optimal}
{\sc A.~Spantini, A.~Solonen, T.~Cui, J.~Martin, L.~Tenorio, and Y.~Marzouk},
  {\em Optimal low-rank approximations of bayesian linear inverse problems},
  SIAM Journal on Scientific Computing, 37 (2015), pp.~A2451--A2487.

\bibitem{sutton2000policy}
{\sc R.~S. Sutton, D.~A. McAllester, S.~P. Singh, and Y.~Mansour}, {\em Policy
  gradient methods for reinforcement learning with function approximation}, in
  Advances in neural information processing systems, 2000, pp.~1057--1063.

\bibitem{Ucinski00}
{\sc D.~Uci{\'n}ski}, {\em Optimal sensor location for parameter estimation of
  distributed processes}, International Journal of Control, 73 (2000),
  pp.~1235--1248.

\bibitem{williams1992simple}
{\sc R.~J. Williams}, {\em Simple statistical gradient-following algorithms for
  connectionist reinforcement learning}, Machine learning, 8 (1992),
  pp.~229--256.

\bibitem{wolsey1999integer}
{\sc L.~A. Wolsey and G.~L. Nemhauser}, {\em Integer and combinatorial
  optimization}, vol.~55, John Wiley \& Sons, 1999.

\bibitem{yu2018scalable}
{\sc J.~Yu, V.~M. Zavala, and M.~Anitescu}, {\em A scalable design of
  experiments framework for optimal sensor placement}, Journal of Process
  Control, 67 (2018), pp.~44--55.

\end{thebibliography}
